\documentclass[11pt,a4paper]{article}
\usepackage{amsfonts,amsgen,amsmath,amstext,amsbsy,amsopn,amsthm,amsfonts,amssymb,amscd,mathtools}
\usepackage{arcs} 
\usepackage{epsf,epsfig}
\usepackage{float}
\usepackage{ebezier,eepic}
\usepackage{color}
\usepackage{tikz}
\usepackage{multirow}                                \usepackage{subcaption}         
\usepackage{graphicx}
\usepackage{mathrsfs}
\usepackage{graphics}
\usepackage{graphicx}
\usepackage{color}
\usepackage{ifpdf}
\usepackage{fancyhdr}
\usepackage{epstopdf}
\usepackage{epsfig}
\usepackage{indentfirst}
\usepackage{CJK}
\newtheorem{thm}{Theorem}[section]

\newtheorem{lem}[thm]{Lemma}
\newtheorem{false statement}{False statement}
\newtheorem{cor}[thm]{Corollary}

\theoremstyle{definition}

\newtheorem{claim}{Claim}

\makeatletter \@addtoreset{equation}{section}

\usepackage{amsthm} 
\allowdisplaybreaks

\title{}
\author{}

\begin{document}
	\title{On the structure of dense graphs with given odd girth}
	\author{Xingyan Lu\footnote{Department of Mathematics, Jiangsu University, Zhenjiang, Jiangsu 212013, China. E-mail:luxingyan15723@163.com. }\quad\quad
		Shipeng Wang\footnote{Department of Mathematics, Jiangsu University, Zhenjiang, Jiangsu 212013, China. E-mail:spwang22@ujs.edu.cn. Research supported by NSFC No.12001242.} \quad\quad
	\\
	}
	\date{}
	\maketitle
	\maketitle
	
	\begin{abstract}
		A classical theorem of Andrásfai, Erdős, and Sós states that every $n$-vertex graph $G$ with odd girth at least $2k+1$ and minimum degree $\delta(G)>\frac{2n}{2k+1}$ is  bipartite (i.e.,  homomorphic to $K_2$). Messuti and Schacht  proved that the same odd girth condition with $\delta(G)>\frac{3n}{4k}$ forces a homomorphism to $C_{2k+1}$. 
        
        In this paper, we strengthen the above results by showing that every $n$-vertex graph $G$ with odd girth at least $2k+1$ and  minimum degree $\delta(G)>\frac{4n}{6k-1}$ is homomorphic to the  Möbius ladder on $4k$ vertices. This answers a question of  Messuti and Schacht and generalizes a result of  Brandt and Ribe-Baumann.
	\end{abstract}
	
	
	\section{Introduction}
   We consider finite and simple graphs without loops and for any notation not defined here we refer to the textbooks~\cite{book2,book1}.
      A {\it homomorphism} from a graph $G$ to a graph $H$ is a mapping $\varphi:V(G)\rightarrow V(H)$ such that $\varphi(u)\varphi(v)\in E(H)$ whenever $uv\in E(G)$. If such a map exists, we say that $G$ is {\it homomorphic} to $H$. 
    For example, $G$ is homomorphic to $K_r$ if and only if $\chi(G)\leq r$.

     For $\ell\geq1,k\geq2$, let $F_{\ell,k}$  be the graph obtained from a cycle of length $(2k-1)(\ell-1)+2$ (an edge, when $\ell=1$) by adding all chords joining vertices at distance of the
form $j(2k-1)+1$ for $j=1,2,\ldots,\lfloor\frac{\ell-1}{2}\rfloor$. This graph is sometimes known as {\it generalized Andr\'{a}sfai graph}. It is not difficult to check that $F_{\ell,k}$ is $\{C_3,C_5,\ldots,C_{2k-1}\}$-free, $\ell$-regular and 3-colorable. 
For $t\geq3$, let $C_t$ be the cycle of length $t$.  For an even integer $r\geq 4$,  a chord in an even cycle $C_r$ is
{\it diagonal} if it joins two vertices at distance $\frac{r}{2}$ in the cycle.
For an even integer $r\geq 6$, we denote by $M_r$ the so-called {\it M\"{o}bius ladder}
(see, e.g., \cite{Möbius ladders}),
i.e., the graph obtained from a cycle $C_r$ by adding all diagonal chords to $C_r$. Note that $F_{2,k}=C_{2k+1}$,  and $F_{3,k}=M_{4k}$, as  shown in Figure~\ref{graphOfIntroM4k+Phi4k3}a.

A classical result of 
Andrásfai, Erdős and Sós~\cite{ErdosC3free} states that every $n$-vertex $K_{r+1}$-free graph $G$ 
 with minimum degree $\delta(G)>\frac{3r-4}{3r-1}n$ is  $r$-colorable.
This result strengthens the following classical consequence of Tur\'{a}n’s
theorem
: every $n$-vertex $K_{r+1}$-free graph $G$ has minimum degree $\delta(G)\leq\frac{r-1}{r}n$. Moreover, the chromatic number $\chi(G)$ is bounded by a constant independent of $n$.
In this direction,
  H\"{a}ggkvist~\cite{HäggkvistC3free} showed that every triangle-free graph $G$ with $\delta(G)>\frac{3n}{8}$ is homomorphic to $C_5$. 
  Jin~\cite{2-L-10Jin1993} 
    extended the result by showing that
every $2\leq \ell\leq10$, every $n$-vertex  triangle-free  graph $G$ with minimum degree  $\delta(G)>\frac{\ell n}{3\ell-1}$ is homomorphic to  $F_{\ell-1,2}$, and accordingly
has chromatic number at most 3.
 Later, Jin~\cite{4color}
 proved that if $G$ is an $n$-vertex triangle-free graph with $\chi(G)\geq 4$ and
 \(\delta(G)\geq\lfloor 10n/29 \rfloor\)
 then 
 $G$ is homomorphic to the  Grötzsch graph.
Extending this line of  work, Chen, Jin and Koh~\cite{k=2case} showed that every triangle-free $n$-vertex graph $G$ with  minimum degree \(\delta(G)>n/3\) satisfies the following: if 
 $\chi(G)\leq 3$ then $G$
is 
homomorphic to \(F_{\ell,2}\) for some $\ell$; if 
 $\chi(G)\geq 4$ then $G$
contains a Grötzsch graph as a subgraph. 

A natural question is whether this behaviour persists when the odd girth is required to be larger. 
 H\"{a}ggkvist and Jin~\cite{HäggkvistC3C5free} proved that every $\{C_3,C_5\}$-free graph $G$ with $\delta(G)>\frac{n}{4}$ is homomorphic to $C_7$. 
 This result was strengthened  by
 Brandt and Ribe-Baumann~\cite{k=3M4k}, who showed that   every $n$-vertex $\{C_3,C_5\}$-free graph $G$ with $\delta(G)>\frac{4n}{17}$ is homomorphic to $M_{12}$.    
     Later, Letzter and Snyder~\cite{k=3case}  showed that 
     for every $\ell\geq2$,  every $n$-vertex $\{C_3,C_5\}$-free graph $G$ with minimum degree $\delta(G)>\frac{\ell n}{5\ell-3}$ is homomorphic to $F_{\ell-1,3}$.
Several related questions were studied in
\cite{Bottcher2023,Luczak2006}.

The classical work of 
Andrásfai, Erdős, and
Sós \cite{ErdosC3free} implies that
every $n$-vertex graph $G$ with odd girth at least $2k+1$ and minimum degree $\delta(G)>\frac{2n}{2k+1}$ is  bipartite.
    Generalizing the results of 
Häggkvist~\cite{HäggkvistC3free} and of Häggkvist and Jin~\cite{HäggkvistC3C5free},
 Messuti and Schacht proved the following.
 

	\begin{thm}[\cite{MessutiC2k+1}]
		For every $k\geq2$,  every $n$-vertex graph $G$ with odd girth at least $2k+1$ and minimum degree $\delta(G)>\frac{3n}{4k}$ is homomorphic to $C_{2k+1}$.
	\end{thm}

 In the concluding remarks of~\cite{MessutiC2k+1}, the authors asked whether there exists some  small \(\varepsilon>0\) such that every $n$-vertex $G$ with odd girth at least $2k+1$ and
 minimum degree \(\delta(G) \ge \left(\frac{3}{4k}-\varepsilon\right)n\) is homomorphic to \(M_{4k}\). They further  
conjectured the optimality of bound of  \(\delta(G) > \frac{4n}{6k-1}\), with 
the blow-ups of the $(6k-1)$-cycle 
(where all chords connecting pairs of at distance $2k$ are added) providing the corresponding extremal construction.

In this paper, we give an affirmative answer to the above question and also generalize the result of Brandt and Ribe-Baumann~\cite{k=3M4k}.
 
    \begin{thm}\label{MainResult}
    For every $k\geq 2$, every $n$-vertex graph $G$ with odd girth at least $2k+1$ and  $\delta(G)>\frac{4n}{6k-1}$ is homomorphic to $M_{4k}$. 
	\end{thm}

\subsection{Organization and Notation}

The rest of the paper is organized as follows. In Section 2, we present two important lemmas. In Sections 3 and 4, we show that two certain configurations force the existence of an $M_{4k}$. In Section 5, we finish the proof of Theorem~\ref{MainResult}.
In the last section, we give some remarks.

\vspace{2mm}

We follow standard notation through. 
Let $G$ be a graph, and let $C$ be a cycle in $G$.
    For any three vertices $x,y,z$ in $C$, we use $C[x,y,z]$ to denote the sub-path of $C$ with endvertices $x,z$ and inner vertex $y$.  
We use  $\overrightarrow {C}$ to denote an orientation of $C$, and  use $\overrightarrow {C}(x,y)$ to denote the sub-path in $C$ from $x$ to $y$ along $\overrightarrow {C}$.  
For any two vertices $u,v$ in $G$, we use $(u,v)$  to denote a path with endvertices $u$ and $v$, and $dist_G(u,v)$ to denote the distance between $u$ and $v$ in $G$, i.e., the number of edges of the shortest $(u,v)$-path in $G$.
  For a subset $U\subseteq V(G)$, we denote by $G[U]$  the subgraph of $G$ induced by $U$. For two subgraphs $F_1,F_2$ in $G$,
the {\it symmetric difference} of \(F_1\) and \(F_2\),  denoted by \(F_1 \oplus F_2\),
    is the edge‑induced subgraph of $G$ on 
     edge set \((E(F_1) \cup E(F_2)) \setminus (E(F_1) \cap E(F_2))\).

For a path \(P := v_1v_2\ldots v_m\)
with \(1 \le i < j \le m\), we write \(v_iPv_j\) (and similarly \(v_jPv_i\)) for the sub-path \(v_i v_{i+1} \ldots v_j\) of \(P\). 
 The number of edges in $P$ (resp. $C$) is 
 called 
 the {\it length}, denoted by
 \(\ell(P)\) (resp. \(\ell(C)\)). 
Given a walk $W$, we define its length $\ell(W)$ as the number of edges, each counted as many times as it appears in the walk.

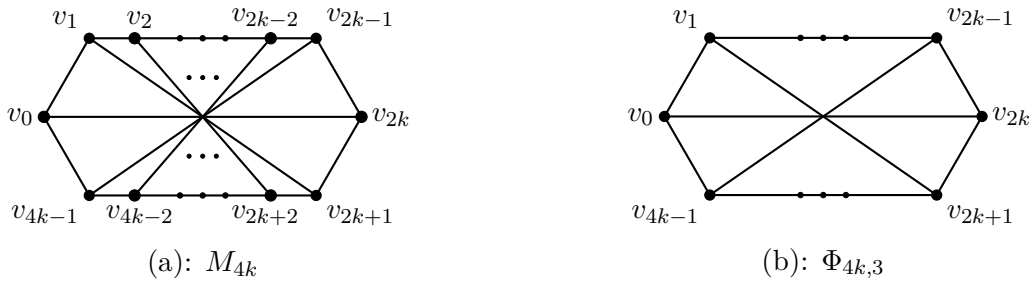
\begin{figure}[h]
\centering
\begin{minipage}{0.45\textwidth}
\centering
\begin{tikzpicture}[scale=0.6]
\coordinate (v_{4k-1}) at (0,0);
\coordinate (v_{2k+1}) at (5,0);
\coordinate (v_{2k}) at (6,1.732);  
\coordinate (v_{2k-1}) at (5,3.464);  
\coordinate (v_1) at (0,3.464);
\coordinate (v_0) at (-1,1.732);

\draw[thick] (v_{4k-1}) -- (v_{2k+1}) -- (v_{2k}) -- (v_{2k-1}) -- (v_1) -- (v_0) -- cycle;
\draw[thick] (v_{2k}) -- (v_0);

\foreach \point/\pos in {v_{4k-1}/below left, v_{2k+1}/below right, v_{2k-1}/above right, v_1/above left} {
    \node[circle, fill=black, inner sep=1.5pt] at (\point) {};
    \node[\pos] at (\point) {$\point$};
}

\coordinate (v_2) at (1,3.464);  
\coordinate (v_{2k-2}) at (4,3.464);  

\coordinate (v_{4k-2}) at (1,0);  
\coordinate (v_{2k+2}) at (4,0);  

\foreach \point in {v_{2k},v_0,v_{4k-2}, v_{2k+2}, v_2, v_{2k-2}} {
    \node[circle, draw=black, fill=black, inner sep=1.5pt, minimum size=2pt] at (\point) {};
}

\node[below, xshift=2pt] at (v_{4k-2}) {$v_{4k-2}$};
\node[below, xshift=-2pt] at (v_{2k+2}) {$v_{2k+2}$};
\node[above, xshift=2pt] at (v_2) {$v_2$};
\node[above, xshift=-2pt] at (v_{2k-2}) {$v_{2k-2}$};
\node[right] at (v_{2k}) {$v_{2k}$};
\node[left] at (v_0) {$v_0$};

\node[] at (2.5,-1.5) {(a): $M_{4k}$};
\draw[thick] (v_{2k+1}) -- (v_1);
\draw[thick] (v_{2k+2}) -- (v_2);
\draw[thick] (v_{4k-1}) -- (v_{2k-1});
\draw[thick] (v_{4k-2}) -- (v_{2k-2});

\coordinate (dot11) at (2.0,3.464);  
\coordinate (dot12) at (2.5,3.464);  
\coordinate (dot13) at (3.0,3.464);  
\coordinate (dot21) at (2.0,0);  
\coordinate (dot22) at (2.5,0);  
\coordinate (dot23) at (3.0,0);  

\foreach \point in {dot11,dot12,dot13,dot21,dot22,dot23} {
    \node[circle, draw=black, fill=black, inner sep=0.5pt, minimum size=2pt] at (\point) {};
}

\coordinate (dot31) at (2.2,2.598);  
\coordinate (dot32) at (2.5,2.598);  
\coordinate (dot33) at (2.8,2.598);  
\coordinate (dot41) at (2.2,0.866);  
\coordinate (dot42) at (2.5,0.866);  
\coordinate (dot43) at (2.8,0.866);  

\foreach \point in {dot31,dot32,dot33,dot41,dot42,dot43} {
    \node[circle, draw=black, fill=black, inner sep=0.5pt, minimum size=1.5pt] at (\point) {};
}

\end{tikzpicture}
\end{minipage}
\hspace{0.08\textwidth}
\begin{minipage}{0.45\textwidth}
\centering
\begin{tikzpicture}[scale=0.6]
    \coordinate (v_{4k-1}) at (0,0);
    \coordinate (v_{2k+1}) at (5,0);
    \coordinate (v_{2k}) at (6,1.732);  
    \coordinate (v_{2k-1}) at (5,3.464);  
    \coordinate (v_1) at (0,3.464);
    \coordinate (v_0) at (-1,1.732);
    
    \draw[thick] (v_{4k-1}) -- (v_{2k+1}) -- (v_{2k}) -- (v_{2k-1}) -- (v_1) -- (v_0) -- cycle;
    \draw[thick] (v_{2k}) -- (v_0);
    
        \node[circle, fill=black, inner sep=1.5pt] at (v_{4k-1}) {};
        \node[below left] at (v_{4k-1}) {$v_{4k-1}$}; 
        \node[circle, fill=black, inner sep=1.5pt] at (v_{2k+1}) {};
        \node[below right] at (v_{2k+1}) {$v_{2k+1}$}; 
        \node[circle, fill=black, inner sep=1.5pt] at (v_{2k}) {};
        \node[right] at (v_{2k}) {$v_{2k}$}; 
        \node[circle, fill=black, inner sep=1.5pt] at (v_{2k-1}) {};
        \node[above right] at (v_{2k-1}) {$v_{2k-1}$}; 
        \node[circle, fill=black, inner sep=1.5pt] at (v_1) {};
        \node[above left] at (v_1) {$v_1$}; 
        \node[circle, fill=black, inner sep=1.5pt] at (v_0) {};
        \node[left] at (v_0) {$v_0$};

    \draw[thick] (v_{4k-1}) -- (v_{2k-1});
    \draw[thick] (v_{2k+1}) -- (v_1);

\coordinate (v_2) at (1,3.464);  
\coordinate (v_{2k-2}) at (4,3.464);  

\coordinate (v_{4k-2}) at (1,0);  
\coordinate (v_{2k+2}) at (4,0);  


    \node[] at (2.5,-1.5) {(b): $\Phi_{4k,3}$};
    
    \coordinate (dot11) at (2.0,3.464);  
    \coordinate (dot12) at (2.5,3.464);  
    \coordinate (dot13) at (3.0,3.464);  
    \coordinate (dot21) at (2.0,0);  
    \coordinate (dot22) at (2.5,0);  
    \coordinate (dot23) at (3.0,0);  
    
    \foreach \point in {dot11,dot12,dot13,dot21,dot22,dot23} {
        \node[circle, draw=black, fill=black, inner sep=0.5pt, minimum size=2pt] at (\point) {};
    }
\end{tikzpicture}
\end{minipage}
\caption{Möbius ladder $M_{4k}$ and its spanning subgraph $\Phi_{4k,3}$}
\label{graphOfIntroM4k+Phi4k3}
\end{figure}
  We call a $\{C_3,C_5,\ldots,C_{2k-1}\}$-free graph $G$, i.e., $G$ has odd girth at least $2k+1$, is {\it maximal} if adding any edge to $G$ yields an odd cycle of length at most $2k-1$. For integers $k\geq2$ and $n$, we denote by $\mathcal{G}_{n,k}$ the set of all maximal $n$-vertex graphs which satisfy the conditions of the main theorem, i.e., $\mathcal{G}_{n,k}=\{G=(V,E):|V|=n,\delta(G)>\frac{4n}{6k-1}$, and $G$ is a maximal $\{C_3,C_5,\ldots,C_{2k-1}\}$-free graph$\}$.

	\section{Two useful lemmas}
    \begin{lem}\label{2C2k+1isspoke}
        Let $G$ be a graph with odd girth at least $2k+1$ containing two odd cycles $C$ and $C'$ of length $2k+1$ which share a common $(u,v)$-path $P_{uv}$ with $u\neq v$. If $C$
        and $C'$ have a common vertex 
        $w\notin V(P_{uv})$ and the vertices $u,v,w$ appear in this order along both orientations $\overrightarrow {C}$ and $\overrightarrow {C'}$,  then 
        $\overrightarrow {C}(v,w)$ 
       and  $\overrightarrow {C'}(v,w)$ 
       have the same length. Similarly,
     $\overrightarrow {C}(w,u)$ 
       and  $\overrightarrow {C'}(w,u)$ 
have the same length.
    \end{lem}
    \begin{proof}
        Suppose, to the contrary, that $\ell(\overrightarrow {C}(v,w)) \neq \ell(\overrightarrow {C'}(v,w))$. 
        If $\ell(\overrightarrow {C}(v,w)), \ell(\overrightarrow {C'}(v,w))$ have different parity, then the closed walks $v\overrightarrow {C}(v,w)w\overleftarrow {C'}(w,v)v$ and 
        $w\overrightarrow {C}(w,u)u\overleftarrow {C'}(u,w)w$ are both odd closed walks whose lengths sum up to   $2(2k+1)-2\ell(P_{uv})\leq4k$. Therefore, one of the two closed walks 
        has length at most $2k-1$, a contradiction. 
        If $\ell(\overrightarrow {C}(v,w)), \ell(\overrightarrow {C'}(v,w))$ have the same parity, then we may, without loss of generality,  assume that $\ell(\overrightarrow {C'}(v,w)) \geq \ell(\overrightarrow {C}(v,w))+2$ since $\ell(\overrightarrow {C}(v,w)) \neq \ell(\overrightarrow {C'}(v,w))$. But then replacing $\overrightarrow {C'}(v,w)$  by $\overrightarrow {C}(v,w)$ in $C'$ yields an odd cycle of length at most $2k-1$, a contradiction.
       This proves the lemma.
    \end{proof}

Given $k\geq 2$, let $\Theta_{Q_1PQ_2}$ be a $(2k+1)$-{\it Theta} graph consisting of 
\begin{itemize}
\item[(i)] three internally disjoint $(x,y)$-paths $Q_1,P,Q_2$ such that
\item[(ii)] $Q_1\cup P$ and $Q_2\cup P$ are two odd cycles of length $2k+1$.
\end{itemize}

 For convenience, we often abbreviate  
  $\Theta_{Q_1PQ_2}$ to $\Theta$.
 For any vertex $u\in V(Q_1)\setminus \{x,y\}$ and $v\in V(Q_2)\setminus \{x,y\}$, we define four paths
$$P_{1}=uQ_1xQ_2v,P_{2}:=uQ_1yQ_2v,P_{3}:=uQ_1xPyQ_2v,P_{4}:=uQ_1yPxQ_2v.$$
We call $P_{1},P_{2}$ the {$(u,v)$-\it outer paths}, and  $P_{3},P_{4}$ the {$(u,v)$-\it intersecting paths} in the graph $\Theta$. Clearly, $P_1\cup P_2=Q_1\cup Q_2=P_3\oplus P_4$.

\begin{lem}\label{P1P2P3P4property}
        Let $G$ be a graph with odd girth at least $2k+1$ containing 
        a $(2k+1)$-Theta graph $\Theta_{Q_1PQ_2}$,
        and let $u$ and $v$ be inner vertices of paths $Q_1$ and $Q_2$, respectively.
        If $uv\in E(G)$, then one of the following holds:
         \begin{itemize}
\item[(i)] if the $(u,v)$-outer paths $P_1,P_2$
 both have even length, then $P$ must have length one  and the cycles $uP_1vu,uP_2vu$ both have length $2k+1$;
\item[(ii)] if the $(u,v)$-outer paths $P_1,P_2$ both have odd length, then the $(u,v)$-intersecting paths $P_3,P_4$ have even length, and
one of the cycles $uP_3vu,uP_4vu$ has length $2k+1$ while the other has length $2k+3$.
\end{itemize}  
    \end{lem}

\begin{proof}
    Since $\ell (P_1\cup P_2)=4k+2-2\ell(P)$, we infer that $P_1$ and $P_2$ have the same parity. 
     First, suppose that $P_1$ and $P_2$ have even length. If $\ell(P)\geq2$,  then $\ell (P_1\cup P_2)\leq4k-2$. Consequently, one of $P_1,P_2$ has even length at most $2k-2$.     
     Together with the edge $uv$, it yields an odd cycle of length at most $2k-1$, a contradiction. Hence $\ell(P)=1$, and therefore 
      $\ell (P_1\cup P_2)=4k$, which implies that 
     the two odd cycles $uP_1vu,uP_2vu$ both have length $2k+1$.

Now consider that $P_1$ and $P_2$  have odd length.
Since $P_1\oplus P_3=P\cup Q_2=P_2\oplus P_4$, it follows that $P_3$ and $P_4$ have even length.
Note that $P_3$ and $P_4$ cover the two odd cycles $P\cup Q_1$ and $P\cup Q_2$. It follows that $\ell (P_3)+\ell(P_4)=4k+2$. If one of $P_3,P_4$ has length at most $2k-2$, then, together with the edge $uv$, it yields an odd cycle of length at most $2k-1$, a contradiction. So, one of $P_3,P_4$
has length $2k$ and the other one has length $2k+2$. Consequently, one of the cycles $uP_3vu,uP_4vu$ has length $2k+1$ and the other one has length $2k+3$. This proves the lemma.
\end{proof}

\section{Cycles of length $4k$ with many diagonal chords}

Let $C:=v_0v_1\ldots \ldots v_{4k-1}v_0$ be a cycle of length $4k$ $(k\geq 2)$
in a graph $G$.  Here and below addition in the indices of $v_i$ is taken modulo $4k$. 
For any $i\in\{0,\ldots, 4k-1\}$,
we call $\{v_{i-1},v_{i+1},v_{2k+i}\}$ a  {\it diagonal triple} of $C$. 
Set
$U:=\{u\in V(G): u$ has at least three neighbors in $C\}$. Note that $U$ may intersect $V(C)$.
We say that $C$ is {\it diagonal} if for every vertex $u\in U$,
 every triple of neighbors of $u$ in $C$
is of the form 
$\{v_{i-1},v_{i+1},v_{2k+i}\}$.
Let $\Phi_{4k,1},\Phi_{4k,3}$ be two graphs
 obtained from $C$ (called {\it rim cycle} ) by adding edges as follows:
 \begin{itemize}
\item  $\Phi_{4k,1}$ is obtained by adding single edge $v_jv_{2k+j}$;
\item $\Phi_{4k,3}$ is obtained by adding edges
$v_jv_{2k+j},v_{j-1}v_{2k+j-1},v_{j+1}v_{2k+j+1}$, the case $j=0$ is shown in Figure~\ref{graphOfIntroM4k+Phi4k3}b.
\end{itemize}  
 Furthermore, we call 
 $\Phi_{4k,1}$  {\it diagonal}
if for every vertex $u\in U$,
every triple of neighbors of $u$ in $\Phi_{4k,1}$ is either a diagonal triple of $C$, or one of the two triples $\{v_{j-2},v_{j+2},v_{j}\}$ and $\{v_{j+2k-2},v_{j+2k+2},v_{j+2k}\}$. The two sets  
$\{v_{j-2},v_{j+2},v_{j}\}$ and $\{v_{j+2k-2},v_{j+2k+2},v_{j+2k}\}$ are called 
{\it special triples} of $\Phi_{4k,1}$. 
In particular, if the rim cycle $C$ of $\Phi_{4k,1}$ is diagonal, then  $\Phi_{4k,1}$ itself is also diagonal.
Moreover, if $\Phi_{4k,1}$ or its rim cycle $C$ is diagonal, then every vertex $u\in U$ has exactly three neighbors in  $\Phi_{4k,1}$.

   The following lemma shows that any diagonal \(\Phi_{4k,1}\) contained in a graph \(G \in \mathcal{G}_{n,k}\) necessarily contains two additional diagonal chords, thus yielding a \(\Phi_{4k,3}\) graph.

\begin{lem}\label{phi4k1isM4k}
Suppose $G\in \mathcal{G}_{n,k}$ contains a copy $H$ of 
      a diagonal  $\Phi_{4k,1}$
      with the rim cycle $C=v_0v_1\ldots v_{4k-1}v_0$ and
      the diagonal chord $v_{0}v_{2k}$. Then $v_{1}v_{2k+1}$ and $v_{4k-1}v_{2k-1}$ are also diagonal chords
      in $G[V(H)]$. Consequently, 
       $G[V(H)]$ contains a copy of $\Phi_{4k,3}$. 
    \end{lem}

The following lemma shows that any  \(\Phi_{4k,3}\) contained in a graph \(G \in \mathcal{G}_{n,k}\) necessarily contains more diagonal chords, eventually yielding an \(M_{4k}\) graph.
 
\begin{lem}\label{phi4k3ism4k}
Suppose $G\in\mathcal{G}_{n,k}$  contains a copy $H$ of 
$\Phi_{4k,3}$ with  rim cycle $C=v_0v_1\ldots v_{4k-1}v_0$ and diagonal chords $v_0v_{2k},v_{4k-1}v_{2k-1},v_{1}v_{2k+1}$ . Then $G[V(H)]$ is an induced copy of $M_{4k}$. 
    \end{lem}

The following result is an easy consequence of Lemmas~\ref{phi4k1isM4k} and~\ref{phi4k3ism4k}.

  \begin{cor}\label{phi4k1isM4kcor}
Suppose $G\in\mathcal{G}_{n,k}$ contains a copy $H$ of a diagonal $\Phi_{4k,1}$. Then $G[V(H)]$ is an induced copy of  $M_{4k}$.
\end{cor}

Let \(C_{6,1}\) be the graph obtained from \(C_6\) by adding exactly one diagonal chord. 
 We remark that the authors obtain the following fact in the proof of 
\cite[Lemma 2.1]{MessutiC2k+1}.
 However, for a self contained presentation, we include a proof below.

 \begin{lem}\label{C61hasphi4k3}
           Let $G$ be a maximal $\{C_3,C_5,\ldots,C_{2k-1}\}$-free graph. If $G$ has an induced copy $H$ of $C_{6,1}$ with rim cycle $a_0a_1\ldots a_5a_0$ and diagonal chord $a_1a_4$, then there exist some $(a_0,a_3),(a_2.a_5)$-paths of length $2k-2$. Moreover, for any two $(a_0,a_3),(a_2.a_5)$-paths, say $P_{a_0a_3},P_{a_2a_5}$, of length $2k-2$ in $G$, they cannot intersect. Consequently, $P_{a_0a_3}\cup P_{a_2a_5}\cup H$ is  a copy of $\Phi_{4k,3}$.
   \end{lem}
   \begin{proof}
       Since $G$ is maximal and $a_0a_3\notin E(G)$, there exist $(a_0,a_3)$-paths of even length at most $2k-2$ in $G$, and let $P_{a_0a_3}$ be one of these paths. Note that $a_0a_5a_4a_3$ is a path of length three in $H$.
       Hence, $P_{a_0a_3}$ must have length exactly $2k-2$ and its inner vertices cannot include  $a_4,a_5$; otherwise together with the path $a_0a_5a_4a_3$ it yields an odd closed walk of length at most $2k-1$,  forcing the existence of an odd cycle of length at most $2k-1$, a contradiction. Notice also that $a_0a_1a_2a_3$ is a path of length three in $H$.
       Similarly, we conclude that the inner vertices of $P_{a_0a_3}$ cannot include $a_1,a_2$. So, the inner vertices of $P_{a_0a_3}$ are disjoint from $V(H)$.

The same reasoning can be applied to the other missing diagonal chord $a_2a_5$ to show that there exists another even path, say $P_{a_2a_5}$, of length $2k - 2$ whose inner vertices are disjoint from $V(H)$.
We now show that $P_{a_0a_3}$ and $P_{a_2a_5}$ are vertex disjoint. Suppose to the contrary, and let $v$ be a common vertex of $P_{a_0a_3}$ and $P_{a_2a_5}$. Consider the walks $W_1=a_0P_{a_0a_3}vP_{a_2a_5}a_5$, $W_2=a_2P_{a_2a_5}vP_{a_0a_3}a_3$.    Clearly,  $\ell(W_1)+\ell(W_2)=4k-4$. Consequently, one of the walks, say $W_1$,
has length at most $2k-2$. If one of $W_1,W_2$, say $W_1$,  has even length, then, together with the edge $a_0a_5$, it yields an
odd closed walk of length at most $2k-1$ and hence a short odd cycle. 
Thus, both $W_1$
and $W_2$ have odd length. In this case, the walks $W_3:=a_0P_{a_0a_3}vP_{a_2a_5}a_2$, $W_4:=a_5P_{a_2a_5}vP_{a_0a_3}a_3$ also have  odd length. 
Since $\ell(W_3)+\ell(W_4)=4k-4$, it follows that one of them, say $W_3$, has odd length at most $2k-3$.
Together with the path $a_0a_1a_2$ it yields an odd closed walk with length at most $2k-1$, again 
forcing the existence of an odd cycle of length at most $2k-1$, a contradiction.
   \end{proof}

The following result is an easy consequence of Lemmas~\ref{phi4k3ism4k} and~\ref{C61hasphi4k3}.

  \begin{cor}\label{C}
Suppose $G\in\mathcal{G}_{n,k}$ contains an induced copy of  $C_{6,1}$. Then
     $G$ contains  an induced copy of $M_{4k}$.
    \end{cor}

Now we present the proof of Lemmas~\ref{phi4k1isM4k} and~\ref{phi4k3ism4k}.

\begin{proof}[Proof of Lemma~\ref{phi4k1isM4k}]
By contradiction, we may assume, without loss of generality, that
$v_1v_{2k+1}\notin E(G[V(H)])$. 
Since $G$ is maximal and $v_1v_0v_{2k}v_{2k+1}$
is a path of length three in $G[V(H)]$, there exist $(v_1,v_{2k+1})$-paths of length $2k-2$ in $G$ which, together with $v_1v_0v_{2k}v_{2k+1}$, yield odd cycles of length $2k+1$.
 Under all choices of such paths we pick one, say $P_{v_1v_{2k+1}}$, such that $E(C)\cap E(P_{v_1v_{2k+1}})$ has  maximum cardinality.
Let $H':=H\cup P_{v_1v_{2k+1}}$, and let 
$C':=v_1P_{v_1v_{2k+1}}v_{2k+1}v_{2k}v_0v_1$ be an odd cycle of length 
  $2k+1$ in $H'$. 
Since the two odd cycles $v_0v_1\ldots v_{2k}v_0$ and $C'$ have a common edge $v_0v_1$, it follows from  Lemma~\ref{2C2k+1isspoke} that
$\ell(v_1P_{v_1v_{2k+1}}v)=\ell(v_1v_2\ldots v)$ for any common vertex $v$ of odd cycles $C'$ and $v_0v_1\ldots v_{2k}v_0$.
By the choice of $P_{v_1v_{2k+1}}$,
 we conclude that $G[V(P_{v_1v_{2k+1}})\cap \{v_1,\ldots,v_{2k-1}\}]$ is a path. 
 Similarly, since $v_{2k}v_{2k+1}$ is a common edge of odd cycles $C'$ and $v_0v_{2k}v_{2k+1}\ldots v_{4k-1}v_0$, we can conclude that
$G[V(P_{v_1v_{2k+1}})\cap \{v_{2k+1},\ldots,v_{4k-1}\}]$ is also a path. 
So, assume that $G[V(P_{v_1v_{2k+1}})\cap \{v_1,\ldots,v_{2k-1}\}]:=v_1\ldots v_s$  and  $G[V(P_{v_1v_{2k+1}})\cap \{v_{2k+1},\ldots,v_{4k-1}\}]:=v_{2k+1}\ldots v_{2k+t}$
for some
$s,t\in \{1,\dots, 2k-2\}$. Consider the paths
$$P_1=v_1v_2\ldots v_s,P_2=v_{2k+1}v_{2k+2}\ldots v_{2k+t},P_3=v_{s+1}v_{s+2}\ldots v_{2k-1},$$ $$
P_4=v_{2k+t+1}v_{2k+t+2}\ldots v_{4k-1}, Q=P_{v_1v_{2k+1}}-V(P_1\cup P_2),$$
as shown in Figure~\ref{graphOfDiaPhi4k1-M4k1}a. 
Clearly, $P_1\cup P_2$ contains exactly $s+t$ vertices and $H'$ contains exactly $4k+(2k-1-s-t)=6k-1-s-t$ vertices.
Owing to the odd girth assumption, every vertex in $G$ has at most two neighbors on $C'$, and if it has two, then these two neighbors are at distance two on $C'$.

\begin{figure}[h]
\centering

\begin{subfigure}{0.42\textwidth} 
\centering
\begin{tikzpicture}[scale=0.8] 
\coordinate (v_{4k-1}) at (0,0);
\coordinate (v_{2k+1}) at (5,0);
\coordinate (v_{2k}) at (6,1.732);  
\coordinate (v_{2k-1}) at (5,3.464);  
\coordinate (v_1) at (0,3.464);
\coordinate (v_0) at (-1,1.732);
\coordinate (v_2) at (0.67,3.464);  
\coordinate (v_{2k-2}) at (4.33,3.464);  
\coordinate (v_s) at (2,3.464); 
\coordinate (v_{4k-2}) at (0.67,0);  
\coordinate (v_{2k+2}) at (4.33,0);  
\coordinate (v_{2k+t}) at (3,0); 
\coordinate (v_{s+1}) at (2.6,3.464); 
\coordinate (v_j) at (3.4,3.464); 
\coordinate (v_{j+1}) at (3.9,3.464); 
\coordinate (v_{2k+t+1}) at (2.4,0); 

\draw[thick] (v_{2k+t}) -- (v_{2k+1});
\draw[thick] (v_{4k-1}) -- (v_{2k+1}) -- (v_{2k}) -- (v_{2k-1}) -- (v_1) -- (v_0) -- cycle;
\draw[thick] (v_{2k}) -- (v_0);

\draw[thick] (v_1) -- (v_s);
\draw[thick] (v_s) -- (v_{2k+t});
\draw[thick] (v_{2k+1}) -- (v_{2k+t});
\draw[red,line width=1pt] (v_{s+1}) -- (v_{2k-1});
\draw[purple,line width=1pt] (v_{2k+t}) -- (v_{2k+1});
\draw[green,line width=1pt] (v_{2k+t+1}) -- (v_{4k-1});
\draw[cyan,line width=1pt] (v_s) -- (v_1);

\node[circle, fill=black, inner sep=1.5pt] at (v_{s+1}) {};
    \node[above] at (v_{s+1}) {$v_{s+1}$};
    \node[circle, fill=black, inner sep=1.5pt] at (v_{4k-1}) {};
    \node[below left] at (v_{4k-1}) {$v_{4k-1}$};
    \node[circle, fill=black, inner sep=1.5pt] at (v_{2k+1}) {};
    \node[below right] at (v_{2k+1}) {$v_{2k+1}$};
    \node[circle, fill=black, inner sep=1.5pt] at (v_{2k}) {};
    \node[right] at (v_{2k}) {$v_{2k}$};
    \node[circle, fill=black, inner sep=1.5pt] at (v_{2k-1}) {};
    \node[above right] at (v_{2k-1}) {$v_{2k-1}$};
    \node[circle, fill=black, inner sep=1.5pt] at (v_1) {};
    \node[above left] at (v_1) {$v_1$};
    \node[circle, fill=black, inner sep=1.5pt] at (v_0) {};
    \node[left] at (v_0) {$v_0$};

\foreach \point in {v_{2k+t+1},v_s,v_{2k+t}} {
    \node[circle, draw=black, fill=black, inner sep=1.5pt, minimum size=4pt] at (\point) {};
}

\node[below,xshift=3pt] at (v_{2k+t}) {$v_{2k+t}$};
\node[above,xshift=-9pt] at (v_{2k+t+1}) {$v_{2k+t+1}$};
\node[above, xshift=-6pt, yshift=2pt] at (v_s) {$v_s$};

\coordinate (y_1) at (2.1,3.1177);
\coordinate (y_2) at (2.9,0.3464);
\draw[yellow,thick] (y_1) -- (y_2);

\node[cyan,thick] at (0.5,2.7) {$P_{1}$};
\node[green,thick] at (0.5,0.7) {$P_{4}$};
\node[red,thick] at (3.7,2.7) {$P_{3}$};
\node[purple,thick] at (3.7,0.7) {$P_{2}$};
\node[brown,thick] at (2,1.4) {$Q$};

\coordinate (dot51) at (2.4,2.0784);  
\coordinate (dot52) at (2.5,1.732);  
\coordinate (dot53) at (2.6,1.3856);  
\foreach \point in {dot51,dot52,dot53} {
    \node[circle, draw=black, fill=black, inner sep=0.5pt, minimum size=2pt] at (\point) {};
}

\node[] at (2.5,-1.5) {(a) $H'$};
\end{tikzpicture}

\end{subfigure}
\hspace{0.12\textwidth} 
\begin{subfigure}{0.42\textwidth} 
\centering
\begin{tikzpicture}[scale=0.8] 
\coordinate (v_{4k-1}) at (0,0);
\coordinate (v_{2k+1}) at (5,0);
\coordinate (v_{2k}) at (6,1.732);  
\coordinate (v_{2k-1}) at (5,3.464);  
\coordinate (v_1) at (0,3.464);
\coordinate (v_0) at (-1,1.732);
\coordinate (v_2) at (0.67,3.464);  
\coordinate (v_{2k-2}) at (4.33,3.464);  
\coordinate (v_s) at (2,3.464); 
\coordinate (v_{4k-2}) at (0.67,0);  
\coordinate (v_{2k+2}) at (4.33,0);  
\coordinate (v_{2k+t}) at (3,0); 
\coordinate (v_{j-1}) at (3.0,3.464); 
\coordinate (v_j) at (3.5,3.464); 
\coordinate (v_{j+1}) at (4.0,3.464); 
\coordinate (v_{2k+j}) at (1.6,0); 

\draw[thick] (v_{2k+t}) -- (v_{2k+1});
\draw[thick] (v_{4k-1}) -- (v_{2k+1}) -- (v_{2k}) -- (v_{2k-1}) -- (v_1) -- (v_0) -- cycle;
\draw[thick] (v_{2k}) -- (v_0);


\coordinate (x) at (3.8,2.45);

\coordinate (y) at (2.35,2.2516);
\coordinate (y_1) at (2.25,2.598);
\coordinate (y_2) at (2.45,1.9052);
\draw[thick] (y_2) -- (y_1);

\draw[thick] (x) -- (y_1);
\draw[thick] (x) -- (y_2);
\draw[thick,dashed] (x) -- (v_{j-1});
\draw[thick] (x) -- (v_{j+1});

\coordinate (y_12) at (2.29,2.598);
\coordinate (v_s2) at (2.04,3.464); 

\coordinate (y_22) at (2.5,1.9052);
\coordinate (v_{2k+t}2) at (3.05,0);

\draw[purple,line width=1.5pt] (3.8,2.4) -- (2.45,1.8552);
\draw[purple,line width=1.5pt] (y_22) -- (v_{2k+t}2) -- (v_{2k+1}) -- (v_{2k}) -- (v_{2k-1}) -- (v_{j+1});

\draw[cyan,line width=1.5pt] (3.8,2.5) -- (2.25,2.648);
\draw[cyan,line width=1.5pt] (y_12) -- (v_s2) -- (v_{j+1});

\draw[red,line width=1.5pt] (3.8,2.4) -- (2.25,2.548);
\draw[red,line width=1.5pt] (2.21,2.598) -- (1.96,3.464) -- (v_1) -- (v_0) -- (v_{4k-1}) -- (v_{2k+j});

\draw[green,line width=1.5pt] (3.8,2.5) -- (2.45,1.9552);
\draw[green,line width=1.5pt] (2.4,1.9052) -- (2.95,0) -- (v_{2k+j});

\draw[thick] (x) -- (v_{2k+j});
\node[circle, fill=black, inner sep=1.5pt] at (v_{2k+j}) {};
\node[below, xshift=2pt] at (v_{2k+j}) {$v_{2k+j}$};

\node[circle, fill=black, inner sep=1.5pt] at (v_{j-1}) {};
\node[above, xshift=-2pt] at (v_{j-1}) {$v_{j-1}$};

\node[circle, fill=black, inner sep=1.5pt] at (v_{j+1}) {};
\node[above, xshift=2pt] at (v_{j+1}) {$v_{j+1}$};

\node[circle, fill=black, inner sep=1.5pt] at (v_{4k-1}) {};
\node[below left] at (v_{4k-1}) {$v_{4k-1}$};

\node[circle, fill=black, inner sep=1.5pt] at (v_{2k+1}) {};
\node[below right] at (v_{2k+1}) {$v_{2k+1}$};

\node[circle, fill=black, inner sep=1.5pt] at (v_{2k}) {};
\node[right, xshift=5pt, yshift=2pt] at (v_{2k}) {$v_{2k}$};

\node[circle, fill=black, inner sep=1.5pt] at (v_{2k-1}) {};
\node[above right, xshift=-5pt, yshift=2pt] at (v_{2k-1}) {$v_{2k-1}$};

\node[circle, fill=black, inner sep=1.5pt] at (v_1) {};
\node[above left, xshift=2pt] at (v_1) {$v_1$};

\node[circle, fill=black, inner sep=1.5pt] at (v_0) {};
\node[left, xshift=-2pt] at (v_0) {$v_0$};

\node[circle, fill=black, inner sep=1.5pt] at (v_j) {};

\foreach \point in {v_s,v_{2k+t}} {
    \node[circle, draw=black, fill=black, inner sep=1.5pt, minimum size=4pt] at (\point) {};
}

\node[below] at (v_{2k+t}) {$v_{2k+t}$};
\node[above, xshift=-6pt, yshift=2pt] at (v_s) {$v_s$};

\node[circle, draw=black, fill=black, inner sep=2pt, minimum size=4pt] at (x) {};
\node[right, xshift=2pt] at (x) {$x$};

\foreach \point in {y,y_1, y_2} {
    \node[circle, draw=black, fill=black, inner sep=1.5pt, minimum size=4pt] at (\point) {};
}
\node[left, xshift=-3pt] at (y_1) {$y_1$};
\node[left, xshift=-3pt] at (y_2) {$y_2$};

\node[red,thick] at (0.5,2.7) {\Large$P_{3}'$};
\node[green,thick] at (1.8,1) {\Large$P_{4}'$};
\node[cyan,thick] at (2.8,2.9) {$P_{1}'$};
\node[purple,thick] at (4.2,0.7) {\Large$P_{2}'$};

\node[] at (2.5,-1.5) {(b) The proof of Claim~\ref{P4k4}};
\end{tikzpicture}

\end{subfigure}
\caption{The proof of Lemma~\ref{phi4k1isM4k}}
\label{graphOfDiaPhi4k1-M4k1}
\end{figure}
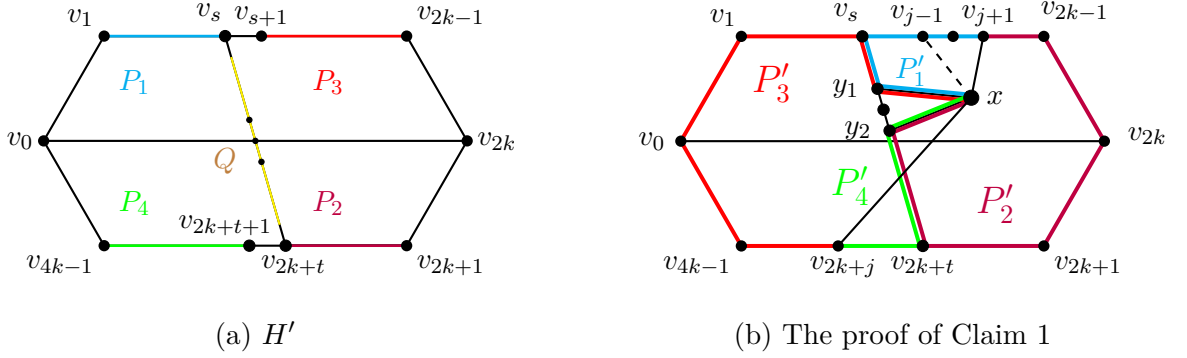

\begin{claim}\label{P4k4}
For any vertex $x$ in $G$, it holds that $|N(x)\cap V(H')|\leq 4$.
\end{claim}

\begin{proof}
Suppose not, and let $x$ be a vertex in $G$
such that $|N(x)\cap V(H')|\geq 5$. Note that $V(H')=V(C')\cup V(H)$. Since  $x$
has at most two neighbors in $C'$, it must have at least three other neighbors in $H$. Because $H$ is a diagonal
$\Phi_{4k,1}$, $x$ has exactly three neighbors in $H$. This implies that $x$ has exactly five neighbors in $H'$, and two of them, say $y_1,y_2$, are on
the path $Q$ with $dist_{Q}(y_1,y_2)=2$ while the other three, say $y_3,y_4,y_5$, are in $P_3\cup P_4$. Clearly, $\{y_3,y_4,y_5\}$ is a diagonal triple of $C$.
Without loss of generality, we may assume that $y_3=v_{j-1},y_4=v_{j+1},y_5=v_{2k+j}$ with $j\in \{s+2,\ldots,2k- 2\}$, i.e., $y_3,y_4$ are on the path $P_3$ and $y_5$ is on the path $P_4$.

Consider the paths  $$P_1'=v_{j+1}v_{j}\ldots v_sP_{v_1v_{2k+1}}y_{1}x,\text{ }P_2'=v_{j+1}v_{j+2}\ldots v_{2k+1}P_{v_1v_{2k+1}}y_2x,$$$$P_3'=v_{2k+j}v_{2k+j+1}\ldots v_{4k-1}v_0v_1P_{v_1v_{2k+1}}y_{1}x,\text{ }P_4'=v_{2k+j}v_{2k+j-1}\ldots v_{2k+t}P_{v_1v_{2k+1}}y_2x.$$
Note that the three paths $v_sv_{s+1}\ldots v_{2k},$ $v_sv_{s-1}\ldots v_0v_{2k}$ and $v_sP_{v_1v_{2k+1}}y_1xy_2P_{v_1v_{2k+1}}v_{2k+1}v_{2k}$
form a $(2k+1)$-Theta graph, say $\Theta_1$. 
Clearly, $P_1',P_2'$ are the two $(x,v_{j+1})$-outer paths of $\Theta_1$. 
By Lemma~\ref{P1P2P3P4property}, 
if  $P_1',P_2'$ both have even length, then the cycle 
$xP_{1}'v_{j+1}x$ has length $2k+1$, but now $xv_{j-1}$ is a chord of $xP_{1}'v_{j+1}x$, a contradiction. 
So, by Lemma~\ref{P1P2P3P4property}, $P_1',P_2'$ both have odd length. Recall that $\{y_3,y_4,y_5\}$ is a diagonal triple of $C$. It follows that $P_1'\oplus P_3'$ is a path of length $2k+1$ and $P_2'\oplus P_4'$ is a path of length $2k-1$, which implies that $P_3',P_4'$ both have even length.

On the other hand, the three paths $v_0v_1P_{v_1v_{2k+1}}y_1xy_2P_{v_1v_{2k+1}}v_{2k+t}$, $v_0v_{2k}v_{2k+1}\ldots v_{2k+t}$ and $v_0v_{4k-1}v_{4k-2}\ldots v_{2k+t}$
form a $(2k+1)$-Theta graph, say $\Theta_2$. Note that $P_3',P_4'$ are the two $(x,v_{2k+j})$-outer paths of $\Theta_2$. So, by Lemma~\ref{P1P2P3P4property}, the path $v_0v_{2k}v_{2k+1}\ldots v_{2k+t}$ has length one, which is impossible. 
This proves the claim.
\end{proof}

\begin{claim}\label{P4k3}
For any vertex $x$ in $G$, if $x$ has at least one neighbor in $P_1\cup P_2$ then $|N(x)\cap V(H')|\leq 3$. 
\end{claim}

\begin{proof}

Suppose,  to the contrary, that
$|N(x)\cap V(H')|\geq 4$. 
Without loss of generality, assume that
$x$ has a neighbor $v_j$ $(1\leq j\leq s)$ on $P_1$.
Since $C'$ is an odd cycle of length $2k+1$, it follows that $C'$ contains at most 
two neighbors of $x$, one of which is $v_j$. Because $v_j$ lies on $P_1$, the path $Q$ contains at most one neighbor of $x$.
Note that $H$ contains at most three neighbors of $x$
because $H$ is a diagonal $\Phi_{4k,1}$.
Since $|N(x)\cap V(H')|\geq 4$ and $V(H')=V(H)\cup V(Q)$, 
it  follows that $Q$ contains exactly one neighbor, say $y_1$, with $dist_{C'}(v_j,y_1)=2$, and $P_3\cup P_4$ contains exactly two neighbors, say $y_2,y_3$, of $x$ in $H'$, respectively.
Thus, $\ell({y_1P_{v_1v_{2k+1}}v_{2k+t}})=2k-2-(t-1)-(j-1)-2=2k-2-t-j$. 
Since $H$ is a diagonal $\Phi_{4k,1}$, it follows that $\{v_j,y_2,y_3\}$ is a  diagonal triple of $C$. 
Clearly, $P_3$ cannot contain both $y_2,y_3$.
If both $y_2,y_3$ are on the path $P_4$,
then 
 $\{v_j,y_2,y_3\}=\{v_{j},v_{2k+j-1},v_{2k+j+1}\}$, where $2k+j-1>2k+t$ since 
$C'$ cannot contain three neighbors of $x$. But now 
the odd cycle $xy_1P_{v_1v_{2k+1}}v_{2k+t}v_{2k+t+1}\ldots v_{2k+j-1}x$ has length $(2k-2-t-j)+(j-1-t)+2=
2k-2t-1$, a contradiction.

So, $P_3$ contains exactly one of $y_2,y_3$, say $y_2$, and  $P_4$ contains $y_3$. This implies that $y_2=v_{j+2}$
and $y_3=v_{2k+j+1}$ with $j\geq t$.
But now $xy_1P_{v_1v_{2k+1}}v_{2k+t}\ldots v_{2k+j+1}x$ is an odd cycle of length
 $(2k-2-t-j)+(j+1-t)+2=2k-2t+1$,
 a contradiction.
This proves the claim.
\end{proof}

\begin{claim}\label{P4k2}
For any vertex $x$ in $G$, if $v$ has exactly two neighbors on $P_1\cup P_2$, then $|N(x)\cap V(H')|=2$.
\end{claim}
\begin{proof}
    Suppose, by contradiction, that $|N(x)\cap V(H')|\geq 3$. Let $y_1,y_2$ be the two  neighbors of $x$ in $P_1\cup P_2$, and let $y_3$ be the other  neighbor of $x$ distinct from $y_1$ and $y_2$. Note that $y_1,y_2$ lie in the odd cycle $C'$ of length $2k+1$. This implies that $y_3\notin V(C')$, i.e., $y_3$ lies in $P_3\cup P_4$. Since $H$ is a diagonal $\Phi_{4k,1}$, it follows that $\{y_1,y_2,y_3\}$ is a  diagonal triple of $C$.

    Suppose that one of $P_1,P_2$, say $P_1$, contains both $y_1$ and $y_2$. Without loss of generality, we may assume that $y_1=v_{j-2},y_2=v_j$ with
    $3\leq j\leq s$, then $y_3=v_{2k+j-1}$ because 
    $\{y_1,y_2,y_3\}$ is a  diagonal triple of $C$.
    Since $C'$ cannot contain three neighbors of $x$, it follows that $2k+j-1>2k+t$. Note that $\ell(v_{j}P_{v_1v_{2k+1}}v_{2k+t})=2k-2-(t-1)-(j-1)=2k-t-j$. However, 
    the odd cycle $xv_{j}P_{v_1v_{2k+1}}v_{2k+t}\ldots v_{2k+j-1}x$ has length $(
    2k-t-j)+(j-1-t)+2=2k-2t+1$, which is smaller than
$2k+1$, a contradiction.

Now suppose that each of $P_1$ and $P_2$  contains precisely  one of $y_1,y_2$.
 Since $y_1,y_2$ are two neighbors of $x$ in $C'$, it follows that $dist_{C'}(y_1,y_2)=2$, i.e., $y_1P_{v_1v_{2k+1}}y_2$ is a path of length two.
 Recall that $y_3$ is in $P_3\cup P_4$.
 We may assume that $y_3$ is on the path $P_3$ without loss of generality. Furthermore, we may assume that
 $y_1=v_{j},y_2=v_{2k+j+1},y_3=v_{j+2}$ with $v_j\in V(P_1)$ and $v_{2k+j+1}\in V(P_2)$ since $\{y_1,y_2,y_3\}$ is a  diagonal triple of $C$.
 This implies that $\ell(P_{v_1v_{2k+1}})=\ell(v_1P_{v_1v_{2k+1}}y_1)+\ell(y_1P_{v_1v_{2k+1}}y_2)+\ell(y_2P_{v_1v_{2k+1}}v_{2k+1})=(j-1)+2+(2k+j+1-(2k+1))=2j+1$, which contradicts the fact that $P_{v_1v_{2k+1}}$ has length $2k-2$. This proves the claim.
\end{proof}

For $i\in\{0,1,2\}$, let 
$V_i=\{x\in V(G):|N(v)\cap V(P_1\cup P_2)|=i\}$.
Since $P_1\cup P_2$ is in the odd cycle $C'$ of length $2k+1$,
it follows that every vertex in $G$ has at most two neighbors in $P_1\cup P_2$, which implies that $(V_0,V_1,V_2)$ is a partition of $G$. By the definition of $V_i$, we 
infer that $\sum_{x\in V(P_1\cup P_2)}d(x)=|V_1|+2|V_2|$.
Therefore, by Claims~\ref{P4k4},~\ref{P4k3}
and~\ref{P4k2} we have
\begin{align*}
\sum_{x\in V_0\cup V_1\cup V_2}|N(x)\cap V(H')|&\leq 4|V_0|+3|V_1|+2|V_2 | \\
&=4(n-|V_1|-|V_2|)+3|V_1|+2|V_2|\\
&=4n-\sum_{x\in V(P_1\cup P_2)}d(x)\\
&<4n-(s+t)\frac{4n}{6k-1}.
\end{align*}
On the other hand, we have

\begin{align*}
\sum_{x\in V_0\cup V_1\cup V_2}|N(x)\cap V(H')| =\sum_{x\in V(H')}d(x) >(6k-1-s-t)\frac{4n}{6k-1} =4n-(s+t)\frac{4n}{6k-1}, 
\end{align*}
 a final contradiction. This completes the proof of Lemma~\ref{phi4k1isM4k}.
 \end{proof}

\begin{proof}[Proof of Lemma~\ref{phi4k3ism4k}]
If we show that $C$ is diagonal, then 
the graph $C\cup \{v_1v_{2k+1}\}$ 
is a diagonal $\Phi_{4k,1}$. Hence,  
by Lemma~\ref{phi4k1isM4k}, we obtain
 $v_2v_{2k+2}\in E(G)$. 
Repeating the process recursively and applying Lemma~\ref{phi4k1isM4k},
we obtain  that
 $C\cup \{v_iv_{2k+i}\}$ 
is a diagonal $\Phi_{4k,1}$, and consequently 
$v_iv_{2k+i}\in E(G)$  for each $i=2,3,\ldots,2k-2$.  Hence we obtain our desired $M_{4k}$ and complete the proof. 

Suppose, to the contrary, that there exists a vertex $u$ in $G$ which has three neighbors $w_1, w_2, w_3$ in $C$ such that $\{w_1, w_2, w_3\}$ is not a diagonal triple of $C$.

\begin{claim}\label{phi4k3NeiDia}
Let $x$ be any vertex in $G$ such that $|N(x)\cap V(C)|\geq 3$. Then for any three distinct vertices $y_1,y_2,y_3$ in $ N(x)\cap V(C)$, the set $\{y_1,y_2,y_3\}$ is either a diagonal triple of $C$ or  one of the two triples $\{v_0,v_2,v_{4k-2}\}$ and $\{v_{2k-2},v_{2k},v_{2k+2}\}$. Consequently, every vertex in $G$ has at most three neighbors in $C$.
\end{claim}
\begin{proof}
    Suppose, to the contrary, that $\{y_1,y_2,y_3\}$ is neither equal to $\{v_0,v_2,v_{4k-2}\}$ nor to $\{v_{2k-2},v_{2k},v_{2k+2}\}$, and  is not a diagonal triple of $C$. 
    Let $C':=v_0v_1\ldots v_{2k}v_0$ and $C'':=v_0v_{2k}v_{2k+1}\ldots v_{4k-1}v_0$ be two odd cycles of length $2k+1$ in $H$. Since $G$ has odd girth at least $2k+1$, one of $C',C''$, say $C'$, contains exactly two of $y_1,y_2,y_3$, say $y_1,y_2$. Clearly, $dist_{C'}(y_1,y_2) = 2$ and $y_3$ lies in $C''$.
    If $\{y_1,y_2\} \cap \{v_0,v_{2k}\}  \neq \emptyset$, then we may assume that $y_1 = v_0$ without loss of generality. So, we have $y_2 \in \{v_2,v_{2k-1}\}$ and $y_3\in \{v_{2k+1},v_{4k-2}\}$. However, if $y_2=v_2$ and $y_3=v_{4k-2}$ then $\{y_1,y_2,y_3\}=\{v_0,v_2,v_{4k-2}\}$; otherwise we can conclude that 
$\{y_1,y_2,y_3\}$ is  a diagonal triple of $C$.  Both cases lead to a contradiction.

So, $\{y_1,y_2\} \cap \{v_0,v_{2k}\}=\emptyset$, which implies that
$y_1,y_2\in \{v_1,\ldots,v_{2k-1}\}$ and $y_3\in \{v_{2k+1},\ldots,v_{4k-1}\}$. Without loss of generality, we may assume that $y_1=v_{j-1},y_2=v_{j+1}$ for some $j\in\{2,\ldots,2k-2\}$, and $y_3=v_{2k+r}$ for some $r\in\{1,\ldots,2k-1\}$. 
Clearly, $r\neq j$ because  $\{y_1,y_2,y_3\}$ is not a diagonal triple of $C$.
Consider the paths, as shown in Figure~\ref{graphOfPhi4k3-M4k111},

$$P_1=xv_{j-1}\ldots v_0v_{4k-1}\ldots v_{2k+r},P_2=xv_{j+1}\ldots v_{2k+r},$$$$P_3=xv_{j-1}\ldots v_0v_{2k}\ldots v_{2k+r},P_4=xv_{j+1}\ldots v_{2k}v_0v_{4k-1}\ldots v_{2k+r}.$$

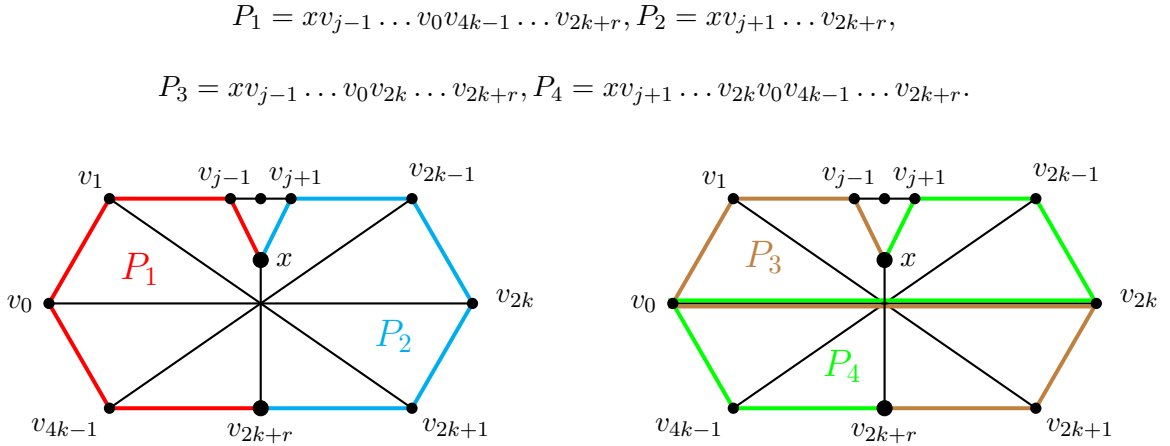
\begin{figure}[h]
\centering
\begin{subfigure}{0.45\textwidth} 
\centering
\begin{tikzpicture}[scale=0.8] 
\coordinate (v_{4k-1}) at (0,0);
\coordinate (v_{2k+1}) at (5,0);
\coordinate (v_{2k}) at (6,1.732);  
\coordinate (v_{2k-1}) at (5,3.464);  
\coordinate (v_1) at (0,3.464);
\coordinate (v_0) at (-1,1.732);
\coordinate (v_2) at (0.67,3.464);  
\coordinate (v_{2k-2}) at (4.33,3.464);  
\coordinate (v_s) at (2,3.464); 
\coordinate (v_{4k-2}) at (0.67,0);  
\coordinate (v_{2k+2}) at (4.33,0);  
\coordinate (v_{2k+t}) at (3,0); 
\coordinate (v_{j-1}) at (2.0,3.464); 
\coordinate (v_j) at (2.5,3.464); 
\coordinate (v_{j+1}) at (3.0,3.464); 
\coordinate (v_{2k+j}) at (2.5,0); 

\draw[thick] (v_{4k-1}) -- (v_{2k+1}) -- (v_{2k}) -- (v_{2k-1}) -- (v_1) -- (v_0) -- cycle;
\draw[thick] (v_{2k}) -- (v_0);
\draw[thick] (v_{2k+1}) -- (v_1);
\draw[thick] (v_{2k-1}) -- (v_{4k-1});


\coordinate (x) at (2.5,2.45); 

\draw[thick] (x) -- (v_{j-1});
\draw[thick] (x) -- (v_{j+1});

\coordinate (y_12) at (2.3,2.598);
\coordinate (v_s2) at (2.05,3.464); 

\coordinate (y_22) at (2.5,1.9052);
\coordinate (v_{2k+t}2) at (3.05,0); 

\draw[cyan,line width=1.5pt] (v_{2k+j}) -- (v_{2k+1}) -- (v_{2k}) -- (v_{2k-1}) -- (v_{j+1}) -- (x);

\draw[red,line width=1.5pt]  (x) -- (v_{j-1}) -- (v_1) -- (v_0) -- (v_{4k-1}) -- (v_{2k+j});

\draw[thick] (x) -- (v_{2k+j});
\node[circle, fill=black, inner sep=1.5pt] at (v_{j-1}) {};
\node[above, xshift=-2pt] at (v_{j-1}) {$v_{j-1}$};

\node[circle, fill=black, inner sep=1.5pt] at (v_{j+1}) {};
\node[above, xshift=2pt] at (v_{j+1}) {$v_{j+1}$};

\node[circle, fill=black, inner sep=1.5pt] at (v_{4k-1}) {};
\node[below left] at (v_{4k-1}) {$v_{4k-1}$};

\node[circle, fill=black, inner sep=1.5pt] at (v_{2k+1}) {};
\node[below right] at (v_{2k+1}) {$v_{2k+1}$};

\node[circle, fill=black, inner sep=1.5pt] at (v_{2k}) {};
\node[right, xshift=5pt, yshift=2pt] at (v_{2k}) {$v_{2k}$};

\node[circle, fill=black, inner sep=1.5pt] at (v_{2k-1}) {};
\node[above right, xshift=-5pt, yshift=2pt] at (v_{2k-1}) {$v_{2k-1}$};

\node[circle, fill=black, inner sep=1.5pt] at (v_1) {};
\node[above left, xshift=2pt] at (v_1) {$v_1$};

\node[circle, fill=black, inner sep=1.5pt] at (v_0) {};
\node[left, xshift=-2pt] at (v_0) {$v_0$};

\node[circle, fill=black, inner sep=1.5pt] at (v_j) {};

\node[circle, draw=black, fill=black, inner sep=2pt, minimum size=4pt] at (x) {};
\node[right, xshift=2pt] at (x) {$x$};

\node[circle, draw=black, fill=black, inner sep=2pt, minimum size=4pt] at (v_{2k+j}) {};
\node[below,yshift=-2pt] at (v_{2k+j}) {$v_{2k+r}$};

\node[red,thick] at (0.5,2.3) {\Large$P_1$};
\node[cyan,thick] at (4.7,1.2) {\Large$P_2$};

\end{tikzpicture}

\end{subfigure}
\hfill
\begin{subfigure}{0.45\textwidth} 
\centering
\begin{tikzpicture}[scale=0.8] 
\coordinate (v_{4k-1}) at (0,0);
\coordinate (v_{2k+1}) at (5,0);
\coordinate (v_{2k}) at (6,1.732);  
\coordinate (v_{2k-1}) at (5,3.464);  
\coordinate (v_1) at (0,3.464);
\coordinate (v_0) at (-1,1.732);
\coordinate (v_2) at (0.67,3.464);  
\coordinate (v_{2k-2}) at (4.33,3.464);  
\coordinate (v_s) at (2,3.464); 
\coordinate (v_{4k-2}) at (0.67,0);  
\coordinate (v_{2k+2}) at (4.33,0);  
\coordinate (v_{2k+t}) at (3,0); 
\coordinate (v_{j-1}) at (2.0,3.464); 
\coordinate (v_j) at (2.5,3.464); 
\coordinate (v_{j+1}) at (3.0,3.464); 
\coordinate (v_{2k+j}) at (2.5,0); 

\draw[thick] (v_{4k-1}) -- (v_{2k+1}) -- (v_{2k}) -- (v_{2k-1}) -- (v_1) -- (v_0) -- cycle;
\draw[thick] (v_{2k}) -- (v_0);
\draw[thick] (v_{2k+1}) -- (v_1);
\draw[thick] (v_{2k-1}) -- (v_{4k-1});


\coordinate (x) at (2.5,2.45); 

\draw[thick] (x) -- (v_{j-1});
\draw[thick] (x) -- (v_{j+1});

\coordinate (y_12) at (2.3,2.598);
\coordinate (v_s2) at (2.05,3.464); 

\coordinate (y_22) at (2.5,1.9052);
\coordinate (v_{2k+t}2) at (3.05,0);

\draw[brown,line width=1.5pt] (v_{2k+j}) -- (v_{2k+1}) -- (5.9713,1.682) -- (-1.0289,1.682) -- (v_1)  -- (v_{j-1}) -- (x);

\draw[green,line width=1.5pt] (v_{2k+j}) -- (v_{4k-1}) -- (-1.0289,1.782) -- (5.9713,1.782) -- (v_{2k-1}) -- (v_{j+1}) -- (x);

\draw[thick] (x) -- (v_{2k+j});
\node[circle, fill=black, inner sep=1.5pt] at (v_{j-1}) {};
\node[above, xshift=-2pt] at (v_{j-1}) {$v_{j-1}$};

\node[circle, fill=black, inner sep=1.5pt] at (v_{j+1}) {};
\node[above, xshift=2pt] at (v_{j+1}) {$v_{j+1}$};

\node[circle, fill=black, inner sep=1.5pt] at (v_{4k-1}) {};
\node[below left] at (v_{4k-1}) {$v_{4k-1}$};

\node[circle, fill=black, inner sep=1.5pt] at (v_{2k+1}) {};
\node[below right] at (v_{2k+1}) {$v_{2k+1}$};

\node[circle, fill=black, inner sep=1.5pt] at (v_{2k}) {};
\node[right, xshift=5pt, yshift=2pt] at (v_{2k}) {$v_{2k}$};

\node[circle, fill=black, inner sep=1.5pt] at (v_{2k-1}) {};
\node[above right, xshift=-5pt, yshift=2pt] at (v_{2k-1}) {$v_{2k-1}$};

\node[circle, fill=black, inner sep=1.5pt] at (v_1) {};
\node[above left, xshift=2pt] at (v_1) {$v_1$};

\node[circle, fill=black, inner sep=1.5pt] at (v_0) {};
\node[left, xshift=-2pt] at (v_0) {$v_0$};

\node[circle, fill=black, inner sep=1.5pt] at (v_j) {};

\node[circle, draw=black, fill=black, inner sep=2pt, minimum size=4pt] at (x) {};
\node[right, xshift=2pt] at (x) {$x$};

\node[circle, draw=black, fill=black, inner sep=2pt, minimum size=4pt] at (v_{2k+j}) {};
\node[below,yshift=-2pt] at (v_{2k+j}) {$v_{2k+r}$};

\node[green,thick] at (1.8,0.7) {\Large$P_4$};
\node[brown,thick] at (0.5,2.5) {\Large$P_3$};

\end{tikzpicture}

\end{subfigure}
\caption{The proof of Claim~\ref{phi4k3NeiDia}}
\label{graphOfPhi4k3-M4k111}
\end{figure}

Note that the three paths $v_0v_1\ldots v_{j-1}xv_{j+1}v_{j+2}\ldots v_{2k},$ $       v_0v_{2k}$ and $v_0v_{4k-1}v_{4k-2}\ldots v_{2k}$ form a  $(2k+1)$-Theta graph $\Theta$.
 Clearly, $P_1,P_2$ are the two $(x,v_{2k+r})$-outer paths, and
$P_3,P_4$ are the two $(x,v_{2k+r})$-intersecting paths in $\Theta$. By Lemma~\ref{P1P2P3P4property}, we have that either the cycles $xP_1v_{2k+r}x$ and $xP_2v_{2k+r}x$ both have length $2k+1$, or one of the cycles $xP_3v_{2k+r}x,xP_4v_{2k+r}x$ has length $2k+1$.  If the first case holds, then $r=j$, which contradicts the fact that $\{y_1,y_2,y_3\}$ is not a diagonal triple of $C$.
If the latter case holds, then we may, without loss of generality, assume that $xP_3v_{2k+r}x$ has length $2k+1$, then $v_1v_{2k+1}$ is a chord in $xP_3v_{2k+r}x$, which yields an odd cycle of length at most $2k-1$, a contradiction.
This proves the claim.
\end{proof}

It follows from Claim~\ref{phi4k3NeiDia} that every vertex in $G$ has at most three neighbors in $C$. 
Since $\{w_1,w_2,w_3\}$ is not a diagonal triple of $C$, we conclude that $\{w_1,w_2,w_3\}$ is equal to $\{v_0,v_2,v_{4k-2}\}$ or $\{v_{2k-2},v_{2k},v_{2k+2}\}$ by Claim~\ref{phi4k3NeiDia}.
 Without loss of generality, we may assume that $\{w_1,w_2,w_3\}=\{v_0,v_2,v_{4k-2}\}$.
 Let $H':=H\cup \{uv_{2},uv_{0},uv_{4k-2}\}$ be a subgraph of $G$, and let $P:=v_{2}v_3\ldots v_{2k-1}$ be the path shown in Figure~\ref{graphOfPhi4k3-M4k222}a.
Moreover, let $D_1:=v_1v_{2k+1}v_{2k+2}\ldots v_{4k-2}uv_2v_1, D_2:=v_0uv_2v_3\ldots v_{2k-1}v_{4k-1}v_0$ be two odd cycles in $G$, as shown in Figure~\ref{graphOfPhi4k3-M4k222}b. Clearly, both $D_1$ and $D_2$ have length $2k+1$, and $P$ has length $2k-3$.

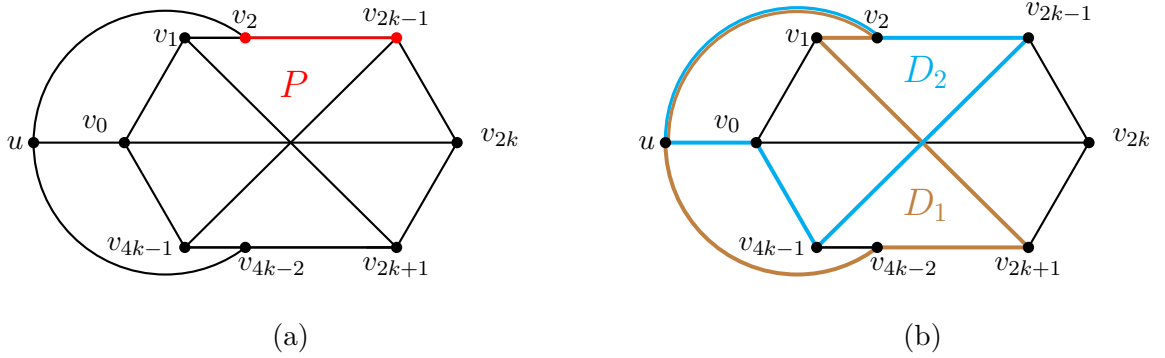
\begin{figure}[h]
\centering

\begin{subfigure}{0.42\textwidth} 
\centering
\begin{tikzpicture}[scale=0.8] 
\coordinate (v_{4k-1}) at (0,0);
\coordinate (v_{2k+1}) at (3.5,0);
\coordinate (v_{2k}) at (4.5,1.732);  
\coordinate (v_{2k-1}) at (3.5,3.464);  
\coordinate (v_1) at (0,3.464);
\coordinate (v_0) at (-1,1.732);
\coordinate (v_2) at (1,3.464);  
\coordinate (v_{2k-2}) at (4.33,3.464);  
\coordinate (v_s) at (2,3.464); 
\coordinate (v_{4k-2}) at (1,0);  
\coordinate (v_{2k+2}) at (4.25,0);  
\coordinate (v_{2k+t}) at (3,0); 
\coordinate (v_{s+1}) at (2.6,3.464); 
\coordinate (v_j) at (3.4,3.464); 
\coordinate (v_{j+1}) at (3.9,3.464); 
\coordinate (v_{2k+t+1}) at (2.4,0); 

\coordinate (v) at (-2.5,1.732);
\coordinate (v_3) at (1.5,3.464);
\coordinate (v_4) at (2.25,3.464);
\coordinate (v_{2k+3}) at (3.5,0);

\draw[thick] (v_{2k+t}) -- (v_{2k+1});
\draw[thick] (v_{4k-1}) -- (v_{2k+1}) -- (v_{2k}) -- (v_{2k-1}) -- (v_1) -- (v_0) -- cycle;
\draw[thick] (v_{2k}) -- (v_0);

\node[circle, fill=black, inner sep=1.5pt] at (v_{4k-1}) {};
\node[left] at (v_{4k-1}) {$v_{4k-1}$};

\node[circle, fill=black, inner sep=1.5pt] at (v_{2k}) {};
\node[right, xshift=5pt, yshift=2pt] at (v_{2k}) {$v_{2k}$};

\node[circle, fill=black, inner sep=1.5pt] at (v_1) {};
\node[left, xshift=2pt] at (v_1) {$v_1$};

\node[circle, fill=black, inner sep=1.5pt] at (v_0) {};
\node[above left, xshift=-2pt] at (v_0) {$v_0$};

\draw[thick] (v_1) -- (v_{2k+1});
\draw[thick] (v_{4k-1}) -- (v_{2k-1});
\draw[thick] (v_1) -- (v_{2k-1});
\draw[thick] (v) -- (v_0);
\draw[thick] (v_{2k+1}) -- (v_{4k-1});
\draw[red,line width=1pt] (v_2) -- (v_{2k-1});
\draw[black,thick](v)arc(180:51:2.18);
\draw[black,thick](v)arc(-180:-51:2.18);

\node[circle, fill=red, inner sep=1.5pt] at (v_2) {};
\node[above] at (v_2) {$v_2$};

\node[circle, fill=red, inner sep=1.5pt] at (v_{2k-1}) {};
\node[above] at (v_{2k-1}) {$v_{2k-1}$};

\node[circle, fill=black, inner sep=1.5pt] at (v) {};
\node[left] at (v) {$u$};

\node[circle, fill=black, inner sep=1.5pt] at (v_{4k-2}) {};
\node[below,xshift=10pt] at (v_{4k-2}) {$v_{4k-2}$};

\node[circle, fill=black, inner sep=1.5pt] at (v_{2k+1}) {};
\node[below] at (v_{2k+1}) {$v_{2k+1}$};

\node[red,thick] at (1.8,2.7) {\Large$P$};

\node[] at (1.75,-1.5) {(a)};
\end{tikzpicture}

\end{subfigure}
\hspace{0.12\textwidth} 
\begin{subfigure}{0.42\textwidth} 
\centering
\begin{tikzpicture}[scale=0.8] 
\coordinate (v_{4k-1}) at (0,0);
\coordinate (v_{2k+1}) at (3.5,0);
\coordinate (v_{2k}) at (4.5,1.732);  
\coordinate (v_{2k-1}) at (3.5,3.464);  
\coordinate (v_1) at (0,3.464);
\coordinate (v_0) at (-1,1.732);
\coordinate (v_2) at (1,3.464);  
\coordinate (v_{2k-2}) at (4.33,3.464);  
\coordinate (v_s) at (2,3.464); 
\coordinate (v_{4k-2}) at (1,0);  
\coordinate (v_{2k+2}) at (4.25,0);  
\coordinate (v_{2k+t}) at (3,0); 
\coordinate (v_{s+1}) at (2.6,3.464); 
\coordinate (v_j) at (3.4,3.464); 
\coordinate (v_{j+1}) at (3.9,3.464); 
\coordinate (v_{2k+t+1}) at (2.4,0); 

\coordinate (v) at (-2.5,1.732);
\coordinate (v_3) at (1.5,3.464);
\coordinate (v_4) at (2.25,3.464);
\coordinate (v_{2k+3}) at (3.5,0);

\draw[thick] (v_{2k+t}) -- (v_{2k+1});
\draw[thick] (v_{4k-1}) -- (v_{2k+1}) -- (v_{2k}) -- (v_{2k-1}) -- (v_1) -- (v_0) -- cycle;
\draw[thick] (v_{2k}) -- (v_0);

\draw[thick] (v_1) -- (v_{2k+1});
\draw[thick] (v_{4k-1}) -- (v_{2k-1});
\draw[thick] (v_1) -- (v_{2k-1});
\draw[cyan, line width=1.5pt] (v) -- (v_0);
\draw[thick] (v_{2k+1}) -- (v_{4k-1});
\draw[cyan,line width=1.2pt](-2.5,1.782)arc(180:51:2.1806);
\draw[brown,line width=1.2pt](-2.45,1.732)arc(181:51:2.13);
\draw[brown,line width=1.5pt](v)arc(-180:-51:2.18);

\draw[brown,line width=1.5pt] (v_2) -- (v_1) -- (v_{2k+1}) -- (v_{4k-2});

\draw[cyan,line width=1.5pt] (v_2) -- (v_{2k-1}) -- (v_{4k-1}) -- (v_0);

\node[circle, fill=black, inner sep=1.5pt] at (v_{4k-1}) {};
\node[left] at (v_{4k-1}) {$v_{4k-1}$};

\node[circle, fill=black, inner sep=1.5pt] at (v_{2k}) {};
\node[right, xshift=5pt, yshift=2pt] at (v_{2k}) {$v_{2k}$};

\node[circle, fill=black, inner sep=1.5pt] at (v_{2k-1}) {};
\node[above right, xshift=-5pt, yshift=2pt] at (v_{2k-1}) {$v_{2k-1}$};

\node[circle, fill=black, inner sep=1.5pt] at (v_1) {};
\node[left, xshift=2pt] at (v_1) {$v_1$};

\node[circle, fill=black, inner sep=1.5pt] at (v_0) {};
\node[above left, xshift=-2pt] at (v_0) {$v_0$};

\node[circle, fill=black, inner sep=1.5pt] at (v_2) {};
\node[above] at (v_2) {$v_2$};

\node[circle, fill=black, inner sep=1.5pt] at (v) {};
\node[left] at (v) {$u$};

\node[circle, fill=black, inner sep=1.5pt] at (v_{4k-2}) {};
\node[below,xshift=10pt] at (v_{4k-2}) {$v_{4k-2}$};

\node[circle, fill=black, inner sep=1.5pt] at (v_{2k+1}) {};
\node[below] at (v_{2k+1}) {$v_{2k+1}$};

\node[] at (1.75,-1.5) {(b)};

\node[brown,thick] at (1.8,0.7) {\Large$D_1$};
\node[cyan,thick] at (1.8,2.9) {\Large$D_2$};

\end{tikzpicture}

\end{subfigure}
\caption{$H'=H\cup \{uv_{2},uv_{0},uv_{4k-2}\}$}

\label{graphOfPhi4k3-M4k222}
\end{figure}

\begin{claim}\label{phi4k3TriLe3}
    For any vertex $x$ in $G$, if $x$ has at least one neighbor on $P$, then $|N(x)\cap V(H')|\leq3$.
\end{claim}
\begin{proof}
    Suppose not, and let $|N(x)\cap V(H')|\geq4$. Recall that every vertex in $G$ has at most three neighbors in $C$.
    Hence $x$ has exactly three neighbors,
     say $y_1,y_2,y_3$,    
    in $C$, and $u$ is another neighbor of $x$. By the choice of $x$, one of $y_1,y_2,y_3$ lies on $P$.
  Since the vertices on $D_1$ at distance two from $u$ are 
  $v_1,v_{4k-3}$,
it follows that $\{y_1,y_2,y_3\}\cap V(D_1)\subseteq\{v_1,v_{4k-3}\}$. 
    Similarly, 
     since the vertices on $D_2$ at distance two from $u$ are 
  $v_3,v_{4k-1}$, we have $\{y_1,y_2,y_3\}\cap V(D_2)\subseteq\{v_3,v_{4k-1}\}$. 
    Note that $V(H')=V(D_1)\cup V(D_2)\cup\{v_{2k}\}$.  Consequently, $y_1,y_2,y_3\in\{v_1,v_3,v_{2k},v_{4k-3},v_{4k-1}\}$. Therefore, 
 $\{y_1,y_2,y_3\}$ is a diagonal triple of $C$ by Claim~\ref{phi4k3NeiDia},
which should be $\{v_1,v_{4k-1},v_{2k}\}$. But then none of $y_1,y_2,y_3$ is contained in $P$, a contradiction.
 This proves the claim.
\end{proof}

\begin{claim}\label{phi4k3TriLe2}
    For any vertex $x$ in $G$, if $x$ has exactly two neighbors in $P$, then $|N(x)\cap V(H')|=2$. 
\end{claim}
\begin{proof}
Suppose,  to the contrary, that $|N(x)\cap V(H')|\geq 3$.  Let $y_1,y_2$ be  two neighbors of $x$ on $P$ and $y_3$ be the other neighbor of $x$ in $H'$ distinct from $y_1,y_2$.
     Since $y_1,y_2,u\in V(D_2)$, by the odd girth assumption,
     we have $u\notin N(x)$, which implies that $y_1,y_2,y_3\in V(C)$.  Since $y_1,y_2$ are on the path $P$,
      $\{y_1,y_2,y_3\}$ is a diagonal triple of $C$ by Claim~\ref{phi4k3NeiDia}. Thus, we may, without loss of generality, assume that $y_1=v_{i-1},y_2=v_{i+1},y_3=v_{2k+i}$
     for some $i\in\{3,\ldots,2k-2\}$. But now the cycle $xv_{i+2k}v_{i+2k+1}\ldots v_{4k-2}uv_2v_3\ldots v_{i-1}x$ has length $((4k-2)-(i+2k))+(i-1-2)+4=2k-1$, a contradiction. This proves the claim.
\end{proof}

For $i\in\{0,1,2\}$, let 
$V_i:=\{x\in V(G):|N(x)\cap V(P)|=i\}$.
Since $P$ is contained in the cycle $D_2$, it follows that
 every vertex in $G$ has at most two neighbors in $P$, implying that $(V_0,V_1,V_2)$ is a partition of $V(G)$. By the definition of $V_i$, we 
infer that $\sum_{x\in V(P)}d(x)=|V_1|+2|V_2|$.
Recall that every vertex in $G$ has at most three neighbors in $C$, and therefore  at most four neighbors in $H'$.
Therefore, by Claims~\ref{phi4k3TriLe3} and ~\ref{phi4k3TriLe2}, 
 we have
\begin{align*}
\sum_{x\in V_0\cup V_1\cup V_2}|N(x)\cap V(H)|&\leq 4|V_0|+3|V_1|+2|V_2 | \\
&=4(n-|V_1|-|V_2|)+3|V_1|+2|V_2|\\
&=4n-\sum_{x\in V(P)}d(x)\\
&<4n-(2k-2)\frac{4n}{6k-1}.
\end{align*}

On the other hand, we have

\begin{align*}
\sum_{x\in V_0\cup V_1\cup V_2}|N(x)\cap V(H)|=\sum_{x\in V(H)}d(x)>(4k+1)\frac{4n}{6k-1}=4n-(2k-2)\frac{4n}{6k-1}, 
\end{align*}
 a final contradiction. 
The proof of Lemma~\ref{phi4k3ism4k} is complete.
\end{proof}

\section{$(2k+1)$-tetrahedrons with odd faces}

Given $k\geq 2$, let $\mathcal{T}_k$
be the set of 
 $(2k+1)$-tetrahedron graphs
 $T_{\ell_a,\ell_b,\ell_c}$ 
 consisting of 
\begin{itemize}
\item[(i)] one cycle $C(T_{\ell_a,\ell_b,\ell_c})$ (called {\it rim cycle}) with three branch vertices $a,b,c$ in $C(T_{\ell_a,\ell_b,\ell_c})$,
\item[(ii)] a center vertex $z$, and
\item[(iii)] internally vertex disjoint paths (called {\it spokes}) $P_{az}$, $P_{bz}$, $P_{cz}$ connecting the branch vertices with the center, and their lengths are $\ell_a,\ell_b,\ell_c$, respectively,
such that
\item[(iv)] each cycle in $T_{\ell_a,\ell_b,\ell_c}$
containing $z$ and exactly two of the branch
vertices
must have length $2k+1$ 
and each spoke has length at least one.
\end{itemize}     

Let $\mathscr{T}_1:=\{T\in \mathcal{T}_k$: at least two spokes of $T$ have length at least two$\}$
and  $\mathscr{T}_2:=\{T\in \mathcal{T}_k$: at least two spokes of $T$ have length one$\}$.
Clearly, $\mathcal{T}_k=\mathscr{T}_1\cup \mathscr{T}_2$.
 For the rim cycle $C(T_{\ell_a,\ell_b,\ell_c})$ of tetrahedron $T_{\ell_a,\ell_b,\ell_c}$ in $\mathcal{T}_k$ 
 with branch vertices $a,b,c$,
  we denote by $\mathbb{P}_{ab}$ the unique path in $C(T_{\ell_a,\ell_b,\ell_c})$ between $a$ and $b$ which does not contain $c$. Similarly, we define $\mathbb{P}_{ac}$ and $\mathbb{P}_{bc}$. 

    For tetrahedron $T_{\ell_a,\ell_b,\ell_c}$, the rim cycle $C(T_{\ell_a,\ell_b,\ell_c})$  has odd length $3(2k+1)-2(\ell_a+\ell_b+\ell_c)$.  We call 
    the cycle containing $z$ and the two branch vertices $a,b$ in $T_{\ell_a,\ell_b,\ell_c}$ {\it facial cycle}, and  denote it by $C_{ab}$.
    Similarly, we define $C_{ac}$ and $C_{bc}$. 
    For any two facial cycles, say $C_{ab}$ and $C_{ac}$, their symmetric difference $C_{ab}\oplus C_{ac}$ is an even cycle of length
     $\ell (C_{ab}\oplus C_{ac})=\ell(C_{ab})+\ell(C_{ac})-2\ell(P_{az})=4k+2-2\ell(P_{az}).$

\begin{lem}\label{xCycleIsT}
    Let $G$ be a maximal $\{C_3,C_5\ldots,C_{2k-1}\}$-free graph containing a cycle $C$ of length $2k+1$, and let $v$ be a vertex 
    in $V(G)\setminus V(C)$ such that the shortest path from $v$ to a vertex $u\in V(C)$, say $P_{vu}$, has length at most two and $N(v)\cap V(C)\subseteq \{u\}$. Then there exist two vertices $u_1,u_2$ in $C$ satisfying
    \[
dist_C(u,u_i) =
\begin{cases}
2, & \text{if } \ell(P_{vu})=1, \\
1, & \text{if } \ell(P_{vu})=2,
\end{cases}
\]
    together with two $(v,u_1),(v,u_2)$-paths, say $P_{vu_1}$ and $P_{vu_2}$, respectively, of length $2k-2$ in $G$ such that $T:=C\cup P_{vu}\cup P_{vu_1}\cup P_{vu_2}$ is a tetrahedron in $\mathcal{T}_k$
    with center $u$. In particular, if $\ell(P_{vu})=1$, then $T\in\mathscr{T}_1$.
\end{lem}
\begin{proof}
We fix an orientation $\overrightarrow {C}$ of $C$.
Let $u_1,u_2$ be the two vertices in $C$ such that, along $\overrightarrow {C}$, they appear in the order $u_1,u,u_2$. Since $N(v)\cap V(C)\subseteq\{u\}$, by the maximality of $G$, there exist  $(v,u_1),(v,u_2)$-paths in $G$ of length  $2k-2$. Under all choices of such paths  we pick two, say $P_{vu_1}$ and $P_{vu_2}$, respectively,  such that $E(C)\cap E(P_{vu})\cap E(P_{vu_1})\cap E(P_{vu_2})$ has the maximum cardinality. 

Let $D_1:=vP_{vu_1}u_1\overrightarrow{C}(u_1,u)uP_{vu}v$ and $D_2:=vP_{vu}u\overrightarrow{C}(u,u_2)u_2P_{vu_2}v$ be two odd cycles of length $2k+1$ in $G$. 
Note that $\overrightarrow{C}(u_1,u)$ is a common path of $D_1$ and $C$.
If there exists a vertex $w\in V(C)\cap V(D_1)\setminus V(\overrightarrow{C}(u_1,u))$, then by Lemma~\ref{2C2k+1isspoke} $\ell(wP_{vu_1}u_1)=\ell(\overrightarrow{C}(w,u_1) )$. So, by the choice of $P_{vu_1}$, we conclude that 
$G[V(D_1)\cap V(C)]$ is a path. Similarly, since $\overrightarrow{C}(u,u_2)$ is a common path of cycles $C,D_2$, and $P_{vu}$ is a common path of cycles $D_1,D_2$,   we can conclude that $G[V(D_2)\cap V(C)]$ and $G[V(D_1)\cap V(D_2)]$ are both paths. Then $T=D_1\cup D_2\cup C=C\cup P_{vu_1}\cup P_{vu_2}\cup P_{vu}$ is in $\mathcal{T}_k$ with center $u$.
In particular, if $P_{vu}$ is an edge, then 
 the spokes $G[V(D_1)\cap V(C)],G[V(D_2)\cap V(C)]$ of $T$ both have length at least two, which implies that 
 $T\in \mathscr{T}_1$.
 This proves the lemma.
\end{proof}

\begin{lem}\label{4kIndOrDia}
        Let $G$ be a graph with odd girth at least $2k+1$ containing a tetrahedron $T_{\ell_a,\ell_b,\ell_c}\in \mathscr{T}_1$. If $uv\in E(G)$ but $uv\notin  E(T_{\ell_a,\ell_b,\ell_c})$ for some $u,v\in V(C_{ab})\cup V(C_{ac})$, then one of the following holds: 
        \begin{itemize}
        \item[(i)]  there is a new tetrahedron   $T_{\ell_a',\ell_b',\ell_c'}\in \mathscr{T}_1$ such that $\ell _a'>\ell_a,\ell_b'=\ell_b,\ell_c'=\ell_c$;
        \item[(ii)] there is a vertex $x\in\{a,b,c\}$ such that  $\ell_{x}=1$ and $uv$ is a diagonal chord in the cycle $C_{xy_1}\oplus C_{xy_2}$ where $\{y_1,y_2\}=\{a,b,c\}\setminus \{x\}$.
        \end{itemize}             
    \end{lem}    
    
    \begin{proof}
     For simplicity, we denote the graph $T_{\ell_a,\ell_b,\ell_c}$ by $T$. By the odd girth assumption, $u,v$ cannot lie in the same facial cycle of $T$. 
   Note that none of $C_{ab},C_{ac},C_{bc}$ contains a chord.
   Thus, we 
    may without loss of generality assume that $u$ is on the path $C_{ab}[a,b,z]$ and $v$ is on the path 
     $C_{ac}[a,c,z]$.
     Note that  $C_{ab}\cup C_{ac}$ is a $(2k+1)$-Theta graph, say $\Theta_1$. Let $P_1$ and $P_2$ be the $(u,v)$-outer paths in $\Theta_1$ containing $a$ and $z$, respectively, and let $P_3,P_4$  be the $(u,v)$-intersecting paths in $\Theta_1$ such that $P_1\oplus P_3=C_{ac}=P_2\oplus P_4$.  
     Note that the lengths of $P_1,P_2$ have the same parity. If $P_1,P_2$ both have even length, then        
     by Lemma~\ref{P1P2P3P4property}, $\ell_a=1$ and $uv$ is a diagonal chord of $C_{ab}\oplus C_{ac}$. This proves $(ii)$.

 Hence we assume that $P_1,P_2$ both have odd length. By Lemma~\ref{P1P2P3P4property},      $P_3,P_4$ both have even length, and one of the cycles $uP_3vu,uP_4vu$ has length $2k+1$ while the other  has length $2k+3$.
         Note that $P_{bz}\cup P_{cz}$ contains at most one of $u,v$; otherwise $uv$ is a chord in the facial cycle $C_{bc}$. Suppose first that $P_{bz}\cup P_{cz}$ contains exactly one of $u,v$. 
         Without loss of generality, we may assume that $v$ is on the spoke $P_{cz}$, then $u$ is on the path $\mathbb{P}_{ab}-\{a,b\}$. 
          Note that $C_{ab}\cup C_{bc}$ is a $(2k+1)$-Theta graph, say $\Theta_2$.
        Clearly,  $P_3$ is a $(u,v)$-outer path of $\Theta_2$. Since $P_3$ has even length, 
        it follows that
        the other $(u,v)$-outer path of $\Theta_2$ also has even length. 
       Thus,
        we have $\ell_b=1$ and $uv$ is a diagonal chord of $C_{ab}\oplus C_{bc}$  by Lemma~\ref{P1P2P3P4property}.
       This proves $(ii)$.
  
  Hence, we assume that $u,v$ are contained in the rim cycle $C(T)$, other than the branch vertices $a,b,c$ of $T$. Recall that one of the cycles $uP_3vu,uP_4vu$ has length $2k+1$. Without loss of generality, we assume that $\ell(uP_3vu)=2k+1$. Then replacing $C_{ac}$ by $uP_3vu$ in $T$ yields a tetrahedron $T_{\ell_a',\ell_b',\ell_c'}\in\mathcal{T}_k$. 
  Note that the spokes $P_{bz},P_{cz}$ of $T$ are also spokes of $T_{\ell_a',\ell_b',\ell_c'}$, and the other spoke $P_{az}$ of $T$ is replaced by a longer spoke  $C_{ab}[z,a,u]$ in $T_{\ell_a',\ell_b',\ell_c'}$.
  Since $\ell_a'>\ell_a,\ell_b'=\ell_b,\ell_c'=\ell_c$, we have $T_{\ell_a',\ell_b',\ell_c'}\in\mathscr{T}_1$. This proves $(i)$.
    \end{proof}

The following lemma shows that any tetrahedron in \(\mathscr{T}_1\) from a graph \(G \in \mathcal{G}_{n,k}\) necessarily contains a spoke of length one.

\begin{lem}\label{Tspoke1}
    Suppose $G\in\mathcal{G}_{n,k}$  contains a subgraph $T_{\ell_a,\ell_b,\ell_c}\in\mathscr{T}_1$. Then one of $\ell_a,\ell_b,\ell_c$ is equal to one.
    \end{lem}


The following lemma shows that if a tetrahedron in \(\mathscr{T}_1\) from a graph \(G \in \mathcal{G}_{n,k}\) has a spoke of length one, then the two facial cycles incident with this spoke induce an \(M_{4k}\) in \(G\).

 \begin{lem}\label{T1spokeisM4k}
 Suppose $G\in\mathcal{G}_{n,k}$ 
has  a subgraph $T_{\ell_a,\ell_b,\ell_c}\in\mathscr{T}_1$ such that its length-one spoke 
belongs to 
both
        facial cycles $C_{ab},C_{ac}$.
        Then $G[V(C_{ab})\cup V(C_{ac})]$ is an induced copy of $M_{4k}$.
    \end{lem}

    The following result is a consequence of Lemmas~\ref{Tspoke1} and~\ref{T1spokeisM4k}.

\begin{cor}\label{T}  
Suppose $G\in\mathcal{G}_{n,k}$ contains a tetrahedron in $\mathscr{T}_1$. Then $G$ contains an induced copy of $M_{4k}$.     
    \end{cor}

Now  we present the proof of Lemmas~\ref{Tspoke1} and~\ref{T1spokeisM4k}.

    \begin{proof}[Proof of Lemma~\ref{Tspoke1}]
 Suppose not, and let \(T_{\ell_a,\ell_b,\ell_c} \in \mathscr{T}_1\) be a subgraph of \(G\) with \(\ell_a,\ell_b,\ell_c \ge 2\) whose rim cycle \(C(T_{\ell_a,\ell_b,\ell_c})\) has the smallest possible length among all such tetrahedrons.  Let $\ell$ be the length of $C(T_{\ell_a,\ell_b,\ell_c})$.
For convenience, we abbreviate   $T_{\ell_a,\ell_b,\ell_c}$ to $T$, and  $C(T_{\ell_a,\ell_b,\ell_c})$ to $C$.

\begin{claim}\label{induced}
For every subgraph  tetrahedron $T':=T_{\ell_a',\ell_b',\ell_c'}\in \mathscr{T}_1$ in $G$ with $\ell_a',\ell_b',\ell_c' \ge 2$ and $\ell (C(T'))=\ell$,
$T'$ is induced, i.e., $T'=G[V(T')]$.
\end{claim}

\begin{proof}
Suppose,  to the contrary, that there is an edge $uv\in E(G[V(T')])\setminus E(T')$. Clearly, $u,v$ cannot lie in the same facial cycle  of $T$. Without loss of generality, we may assume that $u,v\in V(C_{ab})\cup V(C_{ac})$. Since all spokes of $T'$ have length at least two,  there exists a new tetrahedron  $T_{\ell_a'',\ell_b'',\ell_c''}$ such that $\ell_a''>\ell_a',\ell_b''=\ell_b',\ell_c''=\ell_c'$  by Lemma~\ref{4kIndOrDia}. Thus, we have $\ell(C(T_{\ell_a'',\ell_b'',\ell_c''}))<\ell$, a contradiction.
\end{proof}

\begin{claim} \label{C(T)}
    For any vertex $x$ in $G$, $x$ cannot have three neighbors in the rim cycle $C$.
\end{claim}

\begin{proof}
Suppose not, and let $x$ be a vertex in $G$ with three neighbors $y_1,y_2,y_3$ in the rim cycle $C(T)$. 
Note that $C=\mathbb{P}_{ab}\cup \mathbb{P}_{ac}\cup \mathbb{P}_{bc}$.
If one of $\mathbb{P}_{ab},\mathbb{P}_{ac},\mathbb{P}_{bc}$, say $\mathbb{P}_{ab}$, contains two of $y_1,y_2,y_3$, say $y_1,y_2$, then $dist_{C_{ab}}(y_1,y_2)=2$ by the odd girth assumption. Let $u$ be the common neighbor of $y_1,y_2$ in $C_{ab}$. Then replacing $u$ by $x$ in $T$ yields a new tetrahedron $T'\in \mathscr{T}_1$ which has the same spokes as $T$. However, $T'$ is not an induced subgraph of $G$ because $xy_3\in E(G)$, contradicting Claim~\ref{induced}.

Hence, each of $\mathbb{P}_{ab},\mathbb{P}_{ac},\mathbb{P}_{bc}$ contains exactly one of $y_1,y_2,y_3$. Without loss of generality, we may assume that $y_1\in V(\mathbb{P}_{ab}),y_2\in V(\mathbb{P}_{ac}),y_3\in V(\mathbb{P}_{bc})$. Since the rim cycle 
$C:=C[y_1,a,y_2]\cup C[y_2,c,y_3]\cup C[y_3,b,y_1]$
in $T$ is an odd cycle, we infer that one of the three paths $C[y_1,a,y_2],C[y_2,c,y_3],C[y_3,b,y_1]$ has odd length, say $C[y_1,a,y_2]$. Consider the path $P:=y_1C_{ab}[y_1,b,z]zC_{ac}[z,c,y_2]y_2$.
Since $\ell(C[y_1,a,y_2]\cup P)=\ell(C_{ab}\oplus C_{ac})=4k+2-2\ell_a$, 
 we infer that 
$P$ has odd length. Consequently, both \(C[y_1,a,y_2]\) and \(P\) have length \(2k-1\); otherwise one of them would be an odd path of length at most \(2k-3\), which together with \(x\)  would yield an odd cycle of length at most \(2k-1\). Therefore, the odd cycle $xy_1Py_2x$ has length $2k+1$. As a result, replacing $C_{bc}$ by $xy_1Py_2x$ in $T$ yields a new tetrahedron  $T_{\ell_a',\ell_b',\ell_c'}\in\mathscr{T}_1$.
Note that the spoke $P_{az}$ of $T$ is also a spoke of $T_{\ell_a',\ell_b',\ell_c'}$, and the other spokes $P_{bz},P_{cz}$ of $T$ are replaced by longer spokes $C_{ab}[y_1,b,z],C_{ac}[z,c,y_2]$, respectively. So
$\ell_a'=\ell_a, \ell_b'>\ell_b,\ell_c'>\ell_c$, which implies that $\ell(C(T_{\ell_a',\ell_b',\ell_c'}))<\ell$.
But this contradicts the choice of $T$, thereby proving the claim.
\end{proof}

Let $t:=|V(T)\setminus V(C)|=(\ell_a-1)+(\ell_b-1)+(\ell_c-1)+1=\ell_a+\ell_b+\ell_c-2$.
Since $\ell_a,\ell_b,\ell_c\geq 2$, we have $t\geq 4$. On the other hand,
since the rim cycle $C$  has length
$\ell=3(2k+1)-2(\ell_a+\ell_b+\ell_c)=6k-1-2t$ and 
 $G$ has odd girth at least $2k+1$, we infer that
$t \leq 2k-1$, which implies that 
$k\geq3$ as $t\geq 4$.
Set $S_1:=(N(a)\cup N(b)\cup N(c)\cup N(z))\cap V(T)$
and $S_2=V(T)\setminus S_1$. 
Clearly, $|S_1|\leq 12$. Possibly, $a,b,c\in S_1$ if some of $\mathbb{P}_{ab},\mathbb{P}_{ac},\mathbb{P}_{bc}$ are edges.
Note that $|V(T)|=t+|V(C)|=6k-1-t$. Hence $|S_2|=6k-1-t-|S_1|$. In particular, if $k=3$, then $t=4$ or $5$. Since $t=\ell_a+\ell_b+\ell_c-2$,
 $T$ is isomorphic to $T_{2,2,2}$ or $T_{2,2,3}$, it is easy to check that $|S_1|\leq 9$.

\begin{claim}\label{leq3}
For any vertex $x$ in $G$ with no neighbor in $S_2$, we have $|N(x)\cap V(T)|\leq 3$.
\end{claim}
\begin{proof}
    Suppose not, and let $x$ be a vertex in $G$ with four neighbors $y_1,y_2,y_3,y_4$ in $S_1$.
By Claim~\ref{C(T)}, $C$  contains at most two of $y_i$. So, at least two of $y_1,y_2,y_3,y_4$, say $y_1,y_2$, are on the spokes of $T$ excluding branch vertices
$a,b,c$. If $y_1,y_2$ belong to the same spoke of $T$, then $y_1,y_2$ are at distance two along this spoke. Replacing  the common neighbor of $y_1,y_2$ on this spoke by $x$ in $T$ yields a new tetrahedron  $T'\in \mathscr{T}_1$  
with $\ell (C(T'))=\ell$. However, $T'$ is not an induced subgraph of $G$ because $xy_3\in E(G)$, contradicting Claim~\ref{induced}.

Hence, without loss of generality, we may assume that $y_1,y_2$ are on the spokes  $P_{az}-\{a\},P_{bz}-\{b\}$, respectively. 
Then $dist_{C_{ab}}(y_1,y_2)=2.$, which implies that $z$ is the common neighbor of $y_1,y_2$.
Since none of $C_{ab},C_{ac},C_{bc}$ can contain three of $y_1,y_2,y_3,y_4$, the spoke $P_{cz}$ cannot contain any of $y_3,y_4$.
Thus,  we may without loss of generality assume that  $y_3$ and $y_4$ are contained in
the paths $\mathbb{P}_{bc}-\{b,c\}$ and $\mathbb{P}_{ac}-\{a,c\}$, respectively. Since
$y_2,y_3$ are contained  in the same odd cycle $C_{bc}$, it follows that 
$dist_{C_{bc}}(y_2,y_3)=2$, we can conclude that 
$b$ is a common neighbor of $y_2,y_3$. Therefore, $\ell_b=2$ because $b,z$ are both neighbors of $y_2$.

Let $C':=xy_1C_{ab}[y_1,a,z]zC_{bc}[z,c,y_3]y_3x$ 
be a cycle of $G$.
Clearly, $C'$ has length $2k+(2k+1-3)+2=4k$.
Note that $zy_1,y_2x,y_3b$ are three diagonal chords in $G[V(C')]$, which together with the cycle $C'$, yield a copy of $\Phi_{4k,3}$.
By Lemma~\ref{phi4k3ism4k}, $G[V(C')]$ is an induced copy of $M_{4k}$, and let $va$ be the diagonal chord in this new $M_{4k}$. Clearly,  $v$ is the neighbor of $z$ on the spoke $P_{cz}$.
But then $va$ is a chord of $C_{ac}$, which yields an odd cycle of length at most $2k-1$, a contradiction. This proves the claim.
\end{proof}

\begin{claim}\label{leq2}
For any vertex $x$ in $G$ with at least one neighbor in $S_2$, we have $|N(x)\cap V(T)|\leq 2$.
Consequently, every vertex in $G$ has at most two neighbors in $S_2$.
\end{claim}

\begin{proof}
    Suppose not, and let $x$ be a vertex in $G$ with three neighbors $y_1,y_2,y_3$ in $T$, one of which lies in $S_2$.
By Claim~\ref{C(T)}, at least one of $y_1,y_2,y_3$ is not in the rim cycle $C$. Consequently, two of $y_1,y_2,y_3$ are contained in one of the facial cycles $C_{ab},C_{ac},C_{bc}$.
Without loss of generality, we may assume that $y_1,y_2$ lie in $C_{ab}$, and let $y$ be the common neighbor of $y_1,y_2$ in $C_{ab}$. If $y\notin \{z,a,b\}$, then replacing $y$ by $x$ in $T$ yields a new tetrahedron in $\mathscr{T}_1$  
with the same spokes as $T$. However, it is not an induced subgraph of $G$ because $xy_3\in E(G)$, contradicting Claim~\ref{induced}.
Hence $y\in \{z,a,b\}$, which implies $y_1,y_2\in S_1$ and hence $y_3\in S_2$. If 
 $y_3$ and some of $y_1,y_2$, say $y_1$, are in the same facial cycle of $T$, then similarly as previous proof,
 we can conclude that $y_1,y_3\in  S_1$,
contradicting the fact $y_3\in S_2$.
So,  $y_3$ is not in the same facial cycle as $y_1,y_2$, which  implies that $y_3\notin V(C_{ab})$ and $y\neq z$. Therefore,  $y\in\{a,b\}$.
Without loss of generality, we may assume that $y=b$ and  $y_1,y_2,y_3$ are on the paths $\mathbb{P}_{ab}-\{a,b\},P_{bz}-\{b,z\},\mathbb{P}_{ac}-\{a,c\}$, respectively.

Consider the paths
$$P_1=xy_1C[y_1,a,y_3]y_3,P_2=xy_2P_{bz}zC_{ac}[z,c,y_3]y_3,$$$$P_3=xy_1C_{ab}[y_1,a,z]zC_{ac}[z,c,y_3]y_3,\text{ and }P_4=xy_2P_{bz}zC_{ac}[z,a,y_3]y_3,$$
as shown in Figure~\ref{graphOfTspoke1}a.
Note that the three paths $a\mathbb{P}_{ab}y_1xy_2\mathbb{P}_{bz}z,$ $P_{az}$ and $C_{ac}[a,c,z]$
form a $(2k+1)$-Theta graph, say $\Theta$.
Clearly, $P_1,P_2$ are the two $(x,y_3)$-outer paths, and $P_3,P_4$  are the two $(x,y_3)$-intersecting paths in $\Theta$.  
Since $\ell_a\geq2$ and by Lemma~\ref{P1P2P3P4property}, we conclude that
 $P_3,P_4$ have even length, and
one of the cycles $D_1:=xP_3y_3x,D_2:=xP_4y_3x$ has length $2k+1$ and the other has length $2k+3$, as  shown in Figure~\ref{graphOfTspoke1}b.

Suppose first that  $D_1$ has length $2k+1$. Then replacing $C_{ac}$ by $D_1$ in $T$ yields a tetrahedron $T'\in\mathcal{T}_k$ with the center vertex  $z$ and  the branch vertices $b,c,y_1$. 
Note that the spokes $P_{bz},P_{cz}$ of length at least two in $T$ are also spokes in $T'$, and the other spoke $P_{az}$ in $T$ is replaced by a longer spoke  $C_{ab}[z,a,y_1]$ in $T'$. It follows that $\ell(C(T'))<\ell$. But this contradicts the choice of $T$.

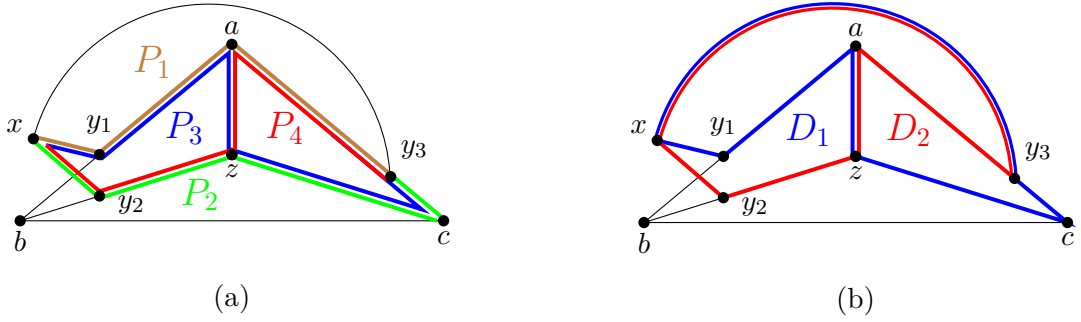
\begin{figure}[h]
\centering
\begin{minipage}{0.45\textwidth}
\centering
\begin{tikzpicture}[scale=0.7]

\coordinate (b) at (0,0);
\coordinate (c) at (8,0);
\coordinate (a) at (4,3.3333);  
\coordinate (z) at (4,1.25); 
\coordinate (v_{2k}) at (6,0.625); 
\coordinate (v_{2k+1}) at (3.2,2.6667); 

\coordinate (z_2) at (1.1,0.9167);
\coordinate (u) at (1.5,1.25);
\coordinate (z_1) at (1.9,1.5833);
\coordinate (v) at (6,0.625);

\coordinate (y_1) at (1.5,1.25);
\coordinate (y_2) at (1.5,0.46875);
\coordinate (y_3) at (7,0.8333);

\coordinate (x) at (0.25,1.55);

\draw[black] (c) --(b);
\draw[black] (b) -- (y_1);
\draw[black] (b) -- (y_2);

\draw[] (x) arc (165:4:3.43);

\draw[brown,line width=1.5pt] (0.25,1.59) -- (1.5,1.29) -- (4,3.3733) -- (7,0.8733);

\draw[green,line width=1.5pt] (0.25,1.51) -- (1.5,0.42875) -- (4,1.21) -- (7.872,0);
\draw[green,line width=1.5pt] (8,0.04) -- (7,0.8733);

\draw[blue,line width=1.5pt] (0.4911, 1.4293) -- (1.5267, 1.1808);
\draw[blue,line width=1.5pt] (1.5267, 1.1658) -- (3.94, 3.1667) -- (3.94, 1.3583) -- (7.58, 0.2183) -- (5.9, 1.6254);

\draw[red,line width=1.5pt]  (0.4911, 1.4293) -- (1.5,0.5566);
\draw[red,line width=1.5pt] (1.5,0.5566)-- (4.06,1.358333) -- (4.06,3.166667) -- (6.92275,0.7695);

\foreach \point/\pos in {z/below} {
    \node[circle, fill=black, inner sep=1.5pt] at (\point) {};
    \node[\pos] at (\point) {$\point$};
}
\node[circle, fill=black, inner sep=1.5pt] at (z) {};

\node[circle, fill=black, inner sep=1.5pt] at (a) {};
\node[above] at (a) {$a$};
\node[circle, fill=black, inner sep=1.5pt] at (b) {};
\node[below] at (b) {$b$};
\node[circle, fill=black, inner sep=1.5pt] at (c) {};
\node[below] at (c) {$c$};

\node[circle, fill=black, inner sep=1.5pt] at (y_1) {};
\node[above,yshift=4pt] at (y_1) {$y_1$};
\node[circle, fill=black, inner sep=1.5pt] at (y_2) {};
\node[below right,xshift=3pt,yshift=4pt] at (y_2) {$y_2$};
\node[circle, fill=black, inner sep=1.5pt] at (y_3) {};
\node[above right,yshift=3pt] at (y_3) {$y_3$};

\node[circle, fill=black, inner sep=1.5pt] at (x) {};
\node[below,left, yshift=5pt] at (x) {$x$};

%
\node[] at (5,1.75) {\Large\textcolor{red}{$P_4$}};
\node[] at (3.4,0.5) {\Large\textcolor{green}{$P_2$}};
\node[] at (3.1,1.75) {\Large\textcolor{blue}{$P_3$}};
\node[] at (2.5,3) {\Large\textcolor{brown}{$P_1$}};
\node[below] at (4,-1) {(a)};

\end{tikzpicture}
\end{minipage}
\hfill
\begin{minipage}{0.45\textwidth}
\centering
\begin{tikzpicture}[scale=0.7]

\coordinate (b) at (0,0);
\coordinate (c) at (8,0);
\coordinate (a) at (4,3.3333);  
\coordinate (z) at (4,1.25); 
\coordinate (v_{2k}) at (6,0.625); 
\coordinate (v_{2k+1}) at (3.2,2.6667); 

\coordinate (z_2) at (1.1,0.9167);
\coordinate (u) at (1.5,1.25);
\coordinate (z_1) at (1.9,1.5833);
\coordinate (v) at (6,0.625);

\coordinate (y_1) at (1.5,1.25);
\coordinate (y_2) at (1.5,0.46875);
\coordinate (y_3) at (7,0.8333);

\coordinate (x) at (0.25,1.55);

\draw[black] (c) --(b);
\draw[black] (b) -- (y_1);
\draw[black] (b) -- (y_2);


\draw[blue,line width=1.5pt] (x) -- (y_1) -- (3.94, 3.2833) -- (3.94, 1.26875) -- (c) -- (y_3);

\draw[red,line width=1.5pt]  (x) -- (y_2) -- (4.06,1.26875) -- (4.06,3.2833) -- (y_3);

\draw[red,line width=1.2pt] (0.29,1.5393) arc (165:4:3.388);
\draw[blue,line width=1.2pt] (0.21,1.5607) arc (165:4:3.471);

\foreach \point/\pos in {z/below, a/above} {
    \node[circle, fill=black, inner sep=1.5pt] at (\point) {};
    \node[\pos] at (\point) {$\point$};
}
\node[circle, fill=black, inner sep=1.5pt] at (z) {};

\node[circle, fill=black, inner sep=1.5pt] at (b) {};
\node[below] at (b) {$b$};
\node[circle, fill=black, inner sep=1.5pt] at (c) {};
\node[below] at (c) {$c$};

\node[circle, fill=black, inner sep=1.5pt] at (y_1) {};
\node[above,yshift=4pt] at (y_1) {$y_1$};
\node[circle, fill=black, inner sep=1.5pt] at (y_2) {};
\node[below right,xshift=3pt,yshift=4pt] at (y_2) {$y_2$};
\node[circle, fill=black, inner sep=1.5pt] at (y_3) {};
\node[above right,yshift=3pt] at (y_3) {$y_3$};

\node[circle, fill=black, inner sep=1.5pt] at (x) {};
\node[below,left, yshift=5pt] at (x) {$x$};

%
\node[] at (5,1.75) {\Large\textcolor{red}{$D_2$}};
\node[] at (3.1,1.75) {\Large\textcolor{blue}{$D_1$}};
\node[below] at (4,-1) {(b)};

\end{tikzpicture}
\end{minipage}
\caption{$T\cup \{xy_1,xy_2,xy_3\}$}
\label{graphOfTspoke1}
\end{figure}

So, the cycle $D_1$ has length $2k+3$ and the cycle $D_2$ has length $2k+1$, which implies that $\ell(P_3)=2k+2$ and $\ell (P_4)=2k$.
If $ \ell (zP_{bz}y_2)=1$, i.e., $zy_2\in E(G)$, then $\ell(C_{ac}[z,a,y_3])=\ell(D_2)-\ell(zy_2xy_3)=2k-2$. 
Since $y_3\in S_2$, it follows that 
the sub-path $y_3\mathbb{P}_{ac}c$ of $\mathbb{P}_{ac}$
 has length at least two.
 Recall that the path $P_{cz}$ has length at least two.
Therefore $\ell(C_{ac}[z,c,y_3])=\ell(y_3\mathbb{P}_{ac}c)+\ell_c \geq 4$.
But then $\ell (C_{ac})=\ell(C_{ac}[z,a,y_3])+\ell( C_{ac}[z,c,y_3])\geq 2k-2+4=2k+2$, contradicting the fact that $C_{ac}$ has length $2k+1$.

So,
$ \ell (zP_{bz}y_2)\geq 2$. 
Then replacing $C_{ab}$ by $D_2$ in $T$ again yields
a tetrahedron $T''\in\mathcal{T}_k$ 
with the center vertex $z$ and the branch vertices $y_2,y_3,c$. 
Note that the spoke $P_{cz}$ of length two in $T$ is also a spoke of $T''$, and the other two spokes 
$zP_{bz}y_2,C_{ac}[z,a,y_3]$ of $T''$ have length at least two. 
Since $y_3\in S_2$, it follows that 
the sub-path $a\mathbb{P}_{ac}y_3$ of $\mathbb{P}_{ac}$
 has length at least two.
Moreover, since $y_2b$ is an edge, 
we have $\ell (zP_{bz}y_2)+\ell(C_{ac}[z,a,y_3])=(\ell_b-1)+(\ell_a+\ell(a\mathbb{P}_{ac}y_3))> \ell_a+\ell_b$, which implies  that $\ell(C(T''))>\ell$.
But this contradicts the choice of $T$ and therefore proves the claim. 
\end{proof}

For each  $i\in\{0,1,2\}$, let 
$V_i=\{u\in V(G):|N(u)\cap S_2|=i\}$.
By Claim~\ref{leq2}, $(V_0,V_1,V_2)$ is a partition of $V(G)$.
By the definition of $V_i$, we 
infer that $\sum_{v\in S_2}d(v)=|V_1|+2|V_2|$.
By Claims~\ref{leq3} and~\ref{leq2}, we have

\begin{align*}
\sum_{u\in V_0\cup V_1\cup V_2}|N(u)\cap V(H)|\leq 3|V_0|+2|V_1|+2|V_2| 
&=3(n-|V_1|-|V_2|)+2|V_1|+2|V_2|\\
&\leq3n-\frac{|V_1|+2|V_2|}{2}\\
&=3n-\frac{1}{2}\sum_{v\in S_2}d(v).
\end{align*}

Since $|S_2|=6k-1-t-|S_1|$ and $\delta (G)>\frac{4n}{6k-1}$ and  $t\leq 2k-1$, 
we arrive  that 
$$\sum_{u\in V_0\cup V_1\cup V_2}|N(u)\cap V(H)| <3n-(6k-1-t-|S_1|)\frac{2n}{6k-1}\leq 3n-(4k-|S_1|)\frac{2n}{6k-1}.$$

On the other hand, we have
\begin{align*}
\sum_{u\in V_0\cup V_1\cup V_2}|N(u)\cap V(T)|=
\sum_{v\in V(T)}d(v)>(6k-1-t )\frac{4n}{6k-1}
\geq 4k\frac{4n}{6k-1}.
\end{align*}

Combining the  above two inequalities, we conclude that $6k<2|S_1|-3$,  but this contradicts the fact 
that if $k=3$ then $|S_1|\leq 9$ and if $k\geq 4$ then
$|S_1|\leq 12$. This completes the proof of Lemma~\ref{Tspoke1}.
    \end{proof}

     \begin{proof}[Proof of Lemma~\ref{T1spokeisM4k}]
     For convenience, we abbreviate   $T_{\ell_a,\ell_b,\ell_c}$ to $T$.
       Since two odd cycles $C_{ab},C_{ac}$ share
       the edge $az$, it follows that $C_{ab}\cup C_{ac}$ is a copy of $\Phi_{4k,1}$ with rim cycle $C_{ab}\oplus C_{ac}$ and diagonal chord $az$. Let $C:=C_{ab}\oplus C_{ac}=v_0\ldots v_{4k-1}v_0$ with $a=v_0,z=v_{2k},b=v_s, c=v_t$ where $1\leq s\leq 2k-1, 2k+1\leq t \leq 4k-1$.
       By Corollary~\ref{phi4k1isM4kcor}, it suffices to show that $C_{ab}\cup C_{ac}$
       has diagonal property. Suppose not, and let $x$ be a vertex in $G$ such that $x$ has at least three neighbors $y_1,y_2,y_3$ in $C$, but $\{y_1,y_2,y_3\}$
       is neither a diagonal triple of $C$ nor a special triple of $C_{ab}\cup C_{ac}$. Owing to the odd girth assumption, 
     it follows that one of $C_{ab},C_{ac}$, say $C_{ab}$, contains exactly two of $y_1,y_2,y_3$, say $y_1,y_2$, and $dist_{C_{ab}}(y_1,y_2)=2$.
If $\{y_1,y_2\}\cap \{v_0,v_{2k}\}\neq \emptyset$, then we may assume that $y_1=v_0$ without loss of generality. Since $dist_{C_{ab}}(y_1,y_2)=2$, we have $y_2\in\{v_2,v_{2k-1}\}$.
Similarly, since $dist_{C_{ac}}(y_1,y_3)=2$,
we have $y_3\in\{v_{2k+1},v_{4k-2}\}$. However, if $y_2=v_2$ and $y_3=v_{4k-2}$ then $\{y_1,y_2,y_3\}$ is a special triple of  $C_{ab}\cup C_{ac}$; otherwise we can conclude that 
$\{y_1,y_2,y_3\}$ is  a diagonal triple of $C$.  Both cases lead to a contradiction.

Hence $\{y_1,y_2\}\cap \{v_0,v_{2k}\}=\emptyset$.
Note that $y_1,y_2$ are on the path $v_1v_2\ldots v_{2k-1}$, and $y_3$ is on the path $v_{2k+1}v_{2k+2}\ldots v_{4k-1}$. Let $y$ be the common neighbor of $y_1,y_2$
in $C_{ab}$. Suppose that $y\neq b$
shown in Figure~\ref{graphOfT1-M4kCase1}a.
Then replacing $y$ by $x$ in $T$ yields a new tetrahedron $T_{\ell_a',\ell_b',\ell_c'}\in\mathcal{T}_k$ with the same spokes as $T$. Clearly, $\ell_a'=\ell_a=1,\ell_b'=\ell_b\geq2,\ell_c'=\ell_c\geq2$ and hence $T_{\ell_a',\ell_b',\ell_c'}\in\mathscr{T}_1$. Let $C'_{ab}$ denote the odd cycle 
 obtained from  $C_{ab}$ by replacing $y$ with $x$. Since $x\in V(C'_{ab}),y_3\in V(C_{ac})$ and $xy_3\in E(G)$ but $xy_3\notin E(T_{\ell_a',\ell_b',\ell_c'})$,
by Lemma~\ref{4kIndOrDia}, either there exists another tetrahedron   $T'':=T_{\ell_a'',\ell_b'',\ell_c''}\in \mathscr{T}_1$ such that $\ell_a''>\ell_a'=1,\ell_b''=\ell_b',\ell_c''=\ell_c'$,
or $xy_3$ is a diagonal chord of $C_{ab}'\oplus C_{ac}$. If the first case holds, then all spokes of $T''$ have length at least two, contradicting Lemma~\ref{Tspoke1}. If
the latter case holds, then this contradicts our assumption that $\{y_1,y_2,y_3\}$ is not a diagonal triple of $C$. 

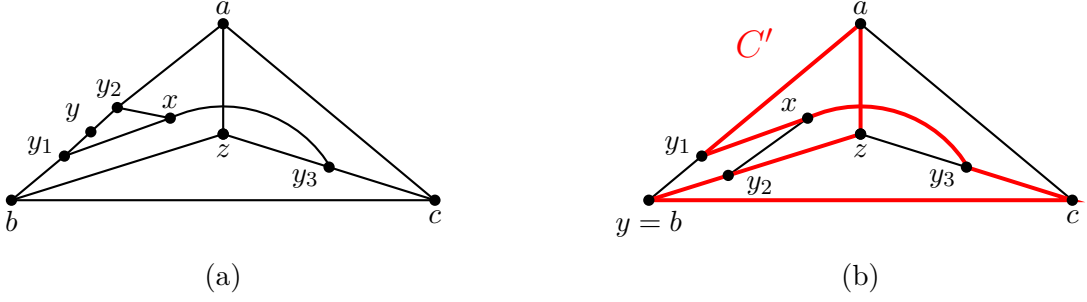
\begin{figure}[ht]
\centering
\begin{minipage}{0.45\textwidth}
\centering
\begin{tikzpicture}[scale=0.7]

\coordinate (b) at (0,0);
\coordinate (c) at (8,0);
\coordinate (a) at (4,3.3333);  
\coordinate (z) at (4,1.25); 
\coordinate (v_{2k}) at (6,0.625); 
\coordinate (v_{2k+1}) at (3.2,2.6667); 

\coordinate (z_2) at (1.1,0.9167);
\coordinate (u) at (1.5,1.25);
\coordinate (z_1) at (1.9,1.5833);
\coordinate (v) at (6,0.625);

\coordinate (y_1) at (1,0.8333);
\coordinate (y_2) at (2,1.75);
\coordinate (y) at (1.5, 1.29165);
\coordinate (y_3) at (6,0.625);

\coordinate (x) at (3,1.55);

\draw[black, thick] (z) -- (a);
\draw[black, thick] (z) -- (b);
\draw[black, thick] (z) -- (c);
\draw[black, thick] (y_2) -- (a) -- (c) --(b) -- (y_1);
\draw[black, thick] (y_2) -- (y_1);


\foreach \point/\pos in {z/below, a/above} {
    \node[circle, fill=black, inner sep=1.5pt] at (\point) {};
    \node[\pos] at (\point) {$\point$};
}
\node[circle, fill=black, inner sep=1.5pt] at (z) {};

\node[circle, fill=black, inner sep=1.5pt] at (y) {};
\node[above left] at (y) {$y$};
\node[circle, fill=black, inner sep=1.5pt] at (b) {};
\node[below] at (b) {$b$};
\node[circle, fill=black, inner sep=1.5pt] at (c) {};
\node[below] at (c) {$c$};

\node[circle, fill=black, inner sep=1.5pt] at (y_1) {};
\node[left,yshift=4pt] at (y_1) {$y_1$};
\node[circle, fill=black, inner sep=1.5pt] at (y_2) {};
\node[above,xshift=-3pt] at (y_2) {$y_2$};
\node[circle, fill=black, inner sep=1.5pt] at (y_3) {};
\node[below left,yshift=3pt] at (y_3) {$y_3$};

\node[circle, fill=black, inner sep=1.5pt] at (x) {};
\node[above] at (x) {$x$};

\draw[black, thick] (x) -- (y_1);
\draw[black, thick] (x) -- (y_2);
\draw[black, thick] (x) arc (115:29:2.365);

\node[below] at (4,-1) {(a)};

\end{tikzpicture}
\end{minipage}
\hfill
\begin{minipage}{0.45\textwidth}
\centering
\begin{tikzpicture}[scale=0.7]

\coordinate (b) at (0,0);
\coordinate (c) at (8,0);
\coordinate (a) at (4,3.3333);  
\coordinate (z) at (4,1.25); 
\coordinate (v_{2k}) at (6,0.625); 
\coordinate (v_{2k+1}) at (3.2,2.6667); 

\coordinate (z_2) at (1.1,0.9167);
\coordinate (u) at (1.5,1.25);
\coordinate (z_1) at (1.9,1.5833);
\coordinate (v) at (6,0.625);

\coordinate (y_1) at (1,0.8333);
\coordinate (y_2) at (1.5,0.46875);
\coordinate (y_3) at (6,0.625);

\coordinate (x) at (3,1.55);

\draw[black, thick] (z) -- (a);
\draw[black, thick] (z) -- (y_2);
\draw[black, thick] (z) -- (c);
\draw[black, thick] (y_1) -- (a) -- (c) --(b);
\draw[black, thick] (b) -- (y_1);
\draw[black, thick] (b) -- (y_2);


\draw[black, thick] (x) -- (y_2);
\draw[red, line width=1.5pt] (x) arc (115:29:2.365);

\draw[red,line width=1.5pt] (y_1) -- (a) -- (z) --(b);
\draw[red,line width=1.5pt] (b) -- (c) -- (y_3);
\draw[red,line width=1.5pt] (x) -- (y_1);

\foreach \point/\pos in {z/below, a/above} {
    \node[circle, fill=black, inner sep=1.5pt] at (\point) {};
    \node[\pos] at (\point) {$\point$};
}
\node[circle, fill=black, inner sep=1.5pt] at (z) {};

\node[circle, fill=black, inner sep=1.5pt] at (b) {};
\node[below] at (b) {$y=b$};
\node[circle, fill=black, inner sep=1.5pt] at (c) {};
\node[below] at (c) {$c$};

\node[circle, fill=black, inner sep=1.5pt] at (y_1) {};
\node[left,yshift=4pt] at (y_1) {$y_1$};
\node[circle, fill=black, inner sep=1.5pt] at (y_2) {};
\node[right,xshift=3pt,yshift=-4pt] at (y_2) {$y_2$};
\node[circle, fill=black, inner sep=1.5pt] at (y_3) {};
\node[below left,yshift=3pt] at (y_3) {$y_3$};

\node[circle, fill=black, inner sep=1.5pt] at (x) {};
\node[below,left, yshift=5pt] at (x) {$x$};

%
\node[] at (2,2.9833) {\Large\textcolor{red}{$C'$}};
\node[below] at (4,-1) {(b)};

\end{tikzpicture}
\end{minipage}
\caption{$T\cup\{xy_1,xy_2,xy_3\}$}
\label{graphOfT1-M4kCase1}
\end{figure}

So, we have $y=b$.
Without loss of generality, we may assume that $y_1$ and $y_2$ are on the paths $\mathbb{P}_{ab}-\{a,b\}$
and $P_{bz}-\{b,z\}$, respectively.
Thus, $y_3$ is on the path $C_{ac}[z,c,a]$.
We distinguish  the following two cases.

	\vspace{3mm}
		{\bf \noindent Case 1. } 
        $y_3$ is on the path $P_{cz}-\{z\}$.
        
Since $y_2,y_3$ are in the cycle $C_{bc}$,
 it follows that $dist_{C_{bc}}(y_2,y_3)=2$, which implies that either $y_2,y_3\in N(z)$, or $y_3=c$ and $\ell(\mathbb{P}_{bc})=1$.
If $y_3=c$ and $\mathbb{P}_{bc}$ is an edge, 
then
$C_{bc}[b,z,c]$ has length $2k$, which implies that $\{y_1,y_2,y_3\}$ is a diagonal triple of $C$, a contradiction.

So, $y_2,y_3\in N(z)$. 
Let $C':=xy_1C_{ab}[y_1,z,b]bC_{bc}[b,c,y_3]y_3x$ be a cycle of $G$, as shown in Figure~\ref{graphOfT1-M4kCase1}b.
Clearly, $C'$ has length $2k+(2k+1-3)+2=4k$. Note that 
$y_1b,xy_2,zy_3$ are three diagonal chords in $G[V(C')]$, which together with the cycle $C'$ yield a copy of $\Phi_{4k,3}$.
So, by Lemma~\ref{phi4k3ism4k}, $G[V(C')]$ is an induced copy of $M_{4k}$.
So,  $v_{2k+2}a$ is also a diagonal chord of $G[V(C')]$ because $v_{2k+2}$ is the neighbor of $y_3$ on the path $P_{cz}-\{z\}$. But now $v_{2k+2}a$ is a chord of $C_{ac}$, which yields an odd cycle of length at most $2k-1$, a contradiction. 

	\vspace{3mm}
		{\bf \noindent Case 2. } 
        $y_3$ is on the path $\mathbb{P}_{ac}-\{a,c\}$.

Consider the paths
$$P_1=xy_1C[y_1,a,y_3]y_3,P_2=xy_2C[y_2,z,y_3]y_3,$$$$P_3=xy_1C_{ab}[y_1,a,z]zC_{ac}[z,c,y_3]y_3,\text{ and }P_4=xy_2P_{bz}zC_{ac}[z,a,y_3]y_3,$$
as shown in Figure~\ref{graphOfT1-M4kCase2}a.
Note that the three paths $a\mathbb{P}_{ab}y_1xy_2P_{bz}z,$ $az$ and $C_{ac}[a,c,z]$
form a $(2k+1)$-Theta graph, say $\Theta$. Clearly, $P_1,P_2$ are the two $(x,y_3)$-outer paths and
$P_3,P_4$ are the two $(x,y_3)$-intersecting paths in $\Theta$. By Lemma~\ref{P1P2P3P4property}, we have that either cycles $xP_1y_3x,xP_2y_3x$ both have length $2k+1$, or exactly one of the cycles $xP_3y_3x,xP_4y_3x$ has length $2k+1$.  If the first case holds, then $\{y_1,y_2,y_3\}$ is a diagonal triple of $C$, a contradiction. 
So, one of the cycles $xP_3y_3x,xP_4y_3x$ has length $2k+1$, and the other has length $2k+3$.
Suppose first that the cycle $xP_3y_3x$ has length $2k+1$. Then replacing $C_{ac}$ by $xP_3y_3x$ in $T$ yields a tetrahedron $T_1\in\mathscr{T}_1$ with the center vertex $z$ and the branch vertices $b,c,y_1$.
Note that the spokes $P_{bz},P_{cz}$ of length at least two in $T$ are also spokes in $T_1$, and the spoke $az$ of $T$ is replaced by the longer spoke $C_{ab}[z,a,y_1]$ of $T_1$. It follows that all spokes of $T_1$ have length at least two. But this contradicts Lemma~\ref{Tspoke1}.

\begin{figure}[h]
\centering
\begin{minipage}{0.45\textwidth}
\centering
\begin{tikzpicture}[scale=0.7]

\coordinate (b) at (0,0);
\coordinate (c) at (8,0);
\coordinate (a) at (4,3.3333);  
\coordinate (z) at (4,1.25); 
\coordinate (v_{2k}) at (6,0.625); 
\coordinate (v_{2k+1}) at (3.2,2.6667); 

\coordinate (z_2) at (1.1,0.9167);
\coordinate (u) at (1.5,1.25);
\coordinate (z_1) at (1.9,1.5833);
\coordinate (v) at (6,0.625);

\coordinate (y_1) at (1.5,1.25);
\coordinate (y_2) at (1.5,0.46875);
\coordinate (y_3) at (7,0.8333);

\coordinate (x) at (0.25,1.55);

\draw[black, thick] (c) --(b);
\draw[black, thick] (b) -- (y_1);
\draw[black, thick] (b) -- (y_2);

\draw[] (x) arc (165:4:3.43);

\draw[brown,line width=1.5pt] (0.25,1.59) -- (1.5,1.29) -- (4,3.3733) -- (7,0.8733);

\draw[green,line width=1.5pt] (0.25,1.51) -- (1.5,0.42875) -- (4,1.21) -- ((7.872,0);
\draw[green,line width=1.5pt] (8,0.04) -- (7,0.8733);

\draw[blue,line width=1.5pt] (0.4911, 1.4293) -- (1.5267, 1.1808);
\draw[blue,line width=1.5pt] (1.5267, 1.1658) -- (3.94, 3.1667) -- (3.94, 1.3583) -- (7.58, 0.2183) -- (5.9, 1.6254);

\draw[red,line width=1.5pt]  (0.4911, 1.4293) -- (1.5,0.5566);
\draw[red,line width=1.5pt] (1.5,0.5566)-- (4.06,1.358333) -- (4.06,3.166667) -- (6.92275,0.7695);

\foreach \point/\pos in {z/below, a/above} {
    \node[circle, fill=black, inner sep=1.5pt] at (\point) {};
    \node[\pos] at (\point) {$\point$};
}
\node[circle, fill=black, inner sep=1.5pt] at (z) {};

\node[circle, fill=black, inner sep=1.5pt] at (b) {};
\node[below] at (b) {$b$};
\node[circle, fill=black, inner sep=1.5pt] at (c) {};
\node[below] at (c) {$c$};

\node[circle, fill=black, inner sep=1.5pt] at (y_1) {};
\node[above,yshift=4pt] at (y_1) {$y_1$};
\node[circle, fill=black, inner sep=1.5pt] at (y_2) {};
\node[below right,xshift=3pt,yshift=4pt] at (y_2) {$y_2$};
\node[circle, fill=black, inner sep=1.5pt] at (y_3) {};
\node[above right,yshift=3pt] at (y_3) {$y_3$};

\node[circle, fill=black, inner sep=1.5pt] at (x) {};
\node[below,left, yshift=5pt] at (x) {$x$};

%
\node[] at (5,1.75) {\Large\textcolor{red}{$P_4$}};
\node[] at (3.4,0.5) {\Large\textcolor{green}{$P_2$}};
\node[] at (3.1,1.75) {\Large\textcolor{blue}{$P_3$}};
\node[] at (2.5,3) {\Large\textcolor{brown}{$P_1$}};
\node[below] at (4,-1) {(a)};

\end{tikzpicture}
\end{minipage}
\hfill
\begin{minipage}{0.45\textwidth}
\centering
\begin{tikzpicture}[scale=0.7]

\coordinate (b) at (0,0);
\coordinate (c) at (8,0);
\coordinate (a) at (4,3.3333);  
\coordinate (z) at (4,1.25); 
\coordinate (v_{2k}) at (6,0.625); 
\coordinate (v_{2k+1}) at (3.2,2.6667); 

\coordinate (z_2) at (1.1,0.9167);
\coordinate (u) at (1.5,1.25);
\coordinate (z_1) at (1.9,1.5833);
\coordinate (v) at (6,0.625);

\coordinate (y_1) at (1.5,1.25);
\coordinate (y_2) at (1.5,0.46875);
\coordinate (y_3) at (7,0.8333);

\coordinate (x) at (0.25,1.55);
\draw[white] (0.75,0.625) arc (200:0:2.71);
\draw[black] (z) -- (a);
\draw[black] (z) -- (y_2);
\draw[black] (z) -- (c);
\draw[black] (y_1) -- (a) -- (c) --(b);
\draw[black] (b) -- (y_1);
\draw[black] (b) -- (y_2);

\draw[red,line width=2pt] (y_3) -- (a) -- (z) -- (y_2) -- (x);

\draw[black] (x) -- (y_1);
\draw[red,line width=2pt] (x) arc (165:4:3.43);

\foreach \point/\pos in {z/below, a/above} {
    \node[circle, fill=black, inner sep=1.5pt] at (\point) {};
    \node[\pos] at (\point) {$\point$};
}
\node[circle, fill=black, inner sep=1.5pt] at (z) {};

\node[circle, fill=black, inner sep=1.5pt] at (b) {};
\node[below] at (b) {$b$};
\node[circle, fill=black, inner sep=1.5pt] at (c) {};
\node[below] at (c) {$c$};

\node[circle, fill=black, inner sep=1.5pt] at (y_1) {};
\node[above,yshift=4pt] at (y_1) {$y_1$};
\node[circle, fill=black, inner sep=1.5pt] at (y_2) {};
\node[right,xshift=3pt,yshift=-4pt] at (y_2) {$y_2$};
\node[circle, fill=black, inner sep=1.5pt] at (y_3) {};
\node[above right,yshift=3pt] at (y_3) {$y_3$};

\node[circle, fill=black, inner sep=1.5pt] at (x) {};
\node[below,left, yshift=5pt] at (x) {$x$};

%
\node[] at (7,3.4833) {\Large\textbf{\textcolor{red}{$D$}}};
\node[below] at (4,-1) {(b)};

\end{tikzpicture}
\end{minipage}
\caption{$T\cup \{xy_1,xy_2,xy_3\}$}
\label{graphOfT1-M4kCase2}
\end{figure}
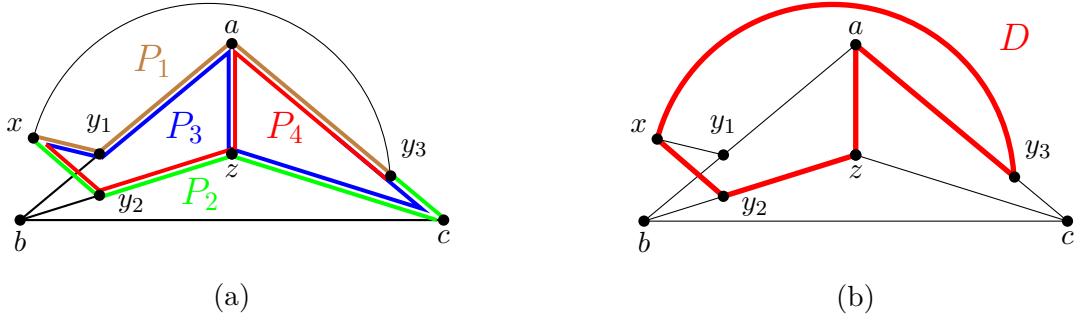

Now suppose that the cycle  $xP_4y_3x$, denoted by $D$, has length $2k+1$, as shown in Figure~\ref{graphOfT1-M4kCase2}b.
Then replacing $C_{ab}$ by $D$ in $T$ yields a tetrahedron 
$T_2\in \mathcal{T}_k$
with center vertex $z$ and branch vertices $y_2,y_3,c$. Note that the spoke $P_{cz}$ of length at least two in $T$ is also a spoke of $T_2$, and another spoke $C_{ac}[z,a,y_3]$ of $T_2$ has length at least two.
 By Lemma~\ref{Tspoke1},
 the other spoke $zP_{bz}y_2$ of $T_2$
has length exactly one, i.e., $zy_2\in E(G)$ and 
$P_{bz}=by_2z$.
Since $\ell (D)=2k+1$, we have $\ell(C_{ac}[z,a,y_3])=\ell(D)-\ell (y_3xy_2z)=2k-2$, which implies that $\ell (C_{ac}[z,c,y_3])=\ell(C_{ac})-\ell (C_{ac}[z,a,y_3])=3$. 
Thus, $P_{cz}$ is a path of length two and $y_3c\in E(G)$ because  $\ell_c\geq2$ and $y_3$ is on the path $\mathbb{P}_{ac}-\{a,c\}$.  This implies that the spokes in  have length  1,1,2, respectively. 
Also notice that that the spokes in $T_2$ have lengths $1,2,2k-2$, respectively.
Note that $D$ and $C_{bc}$
share a common edge $zy_2$,
and the facial cycles of $T_2$ are $D,C_{bc},C_{ac}$. If \(D \cup C_{bc}\) is  a diagonal \(\Phi_{4k,1}\), then by Corollary~\ref{phi4k1isM4kcor}, then  $G[V(D)\cup V(C_{bc})]$ is an induced copy of $M_{4k}$, which implies that $ab$ is a diagonal chord in this $M_{4k}$. But then $ab$ is chord in the odd cycle $C_{ab}$, which yields an odd cycle of length at most $2k-1$, a contradiction.
So, \(D \cup C_{bc}\) is not a diagonal \(\Phi_{4k,1}\), then, by an argument similar to that for \(T\), we can conclude that the spokes of \(T_2\) have lengths \(1, 2, 2\), respectively.
Thus, $2k-2=2$, which implies that $k=2$.
Consider the cycle $C'':=xy_3cby_1azy_2x$ of length 8. Note that $xy_1,ay_3,by_2$ are three diagonal chords of $C''$, thus together with the rim cycle $C''$, it yields a copy of $\Phi_{8,3}$. By Lemma~\ref{phi4k3ism4k}, $G[V(C'')]$ is an induced copy of $M_8$, which implies that $zc$ is a diagonal chord in this $M_8$. But now $cP_{cz}zc$ is a triangle, a final contradiction.
This completes the proof of Lemma~\ref{T1spokeisM4k}.
    \end{proof}

	\section{Proof of Theorem \ref{MainResult}}
	

Let  $\Psi_1$  be a graph with $V(\Psi_1)=\{a_i:i=0,1,\ldots,7\},E(\Psi_1)=\{a_ia_{i+1}:i=0,1,\ldots,6\}\cup\{a_0a_3,a_3a_6,a_1a_4,a_4a_7,a_0a_5\}$, and let  $\Psi_2$ be a graph with 
$V(\Psi_2)=V(\Psi_1),E(\Psi_2)=E(\Psi_1)\cup\{a_2a_7\}$, as shown in Figure~\ref{graphOf8vIntro}a.


    \begin{figure}[h]
    \centering

    \begin{subfigure}{0.45\textwidth}
    \centering
    \begin{tikzpicture}[scale=0.8]
    \coordinate (a_5) at (0,0);
    \coordinate (a_0) at (3,0);
    \coordinate (a_1) at (4.5,1.732);  
    \coordinate (a_2) at (3,3.464);  
    \coordinate (a_3) at (0,3.464);
    \coordinate (a_6) at (-1.5,1.732);
    
    \draw[thick] (a_5) -- (a_0) -- (a_1) -- (a_2) -- (a_3) -- (a_6) -- cycle;
    \draw[thick] (a_1) -- (a_6);
    
    \foreach \point/\pos in {a_5/below left, a_0/below right, a_1/right, a_2/above right, a_3/above left, a_6/left} {
        \node[circle, fill=black, inner sep=1.5pt] at (\point) {};
        \node[\pos] at (\point) {$\point$};
    }
    
    \coordinate (a_7) at (0,1.732);
    \coordinate (a_4) at (0.8,1.732);
    
    \foreach \point in {a_4,a_7} {
        \node[circle, draw=black, fill=black, inner sep=1.5pt, minimum size=4pt] at (\point) {};
    }
    \node[below right] at (a_4) {$a_4$};
    \node[below] at (a_7) {$a_7$};
    \draw[thick] (a_3) -- (a_4);
    \draw[thick] (a_5) -- (a_4);
    \draw[dashed] (a_2) -- (a_7);
\draw[thick] (a_3) -- (a_0);

    \node[below] at (1.5,-1) {(a): $\Psi_1$ (without edge $a_2a_7$) and $\Psi_2$};
    \end{tikzpicture}
\end{subfigure}
\hfill
\begin{subfigure}{0.45\textwidth}
    \centering
    \begin{tikzpicture}[scale=0.8]
    \coordinate (a_5) at (0,0);
    \coordinate (a_0) at (3,0);
    \coordinate (a_1) at (4.5,1.732);  
    \coordinate (a_2) at (3,3.464);  
    \coordinate (a_3) at (0,3.464);
    \coordinate (a_6) at (-1.5,1.732);
    \coordinate (a_7) at (0,1.732);
    \coordinate (a_4) at (0.8,1.732);
    
    \draw[thick] (a_5) -- (a_0) -- (a_1) -- (a_2) -- (a_3) -- (a_6) -- cycle;
    \draw[thick] (a_1) -- (a_6);
    
    \draw[thick] (a_3) -- (a_4);
    \draw[thick] (a_3) -- (a_0);
    \draw[thick] (a_5) -- (a_4);
    \draw[dashed] (a_2) -- (a_7);

    \draw[purple,line width=2pt](a_6)arc(180:1:3);
    \draw[blue,line width=2pt](a_2)arc(9:-90:3);
    \draw[red,line width=2pt](a_7)arc(90:270:1.2);
    \coordinate (w1) at (0,-0.668);
    \coordinate (w2) at (2.332,-0.668);
    \draw[red,line width=2pt] (w1) -- (w2);
    \draw[red,line width=2pt](a_0)arc(0:-90:0.668);

    \foreach \point/\pos in {a_0/below right, a_1/right, a_2/above left, a_6/left} {
        \node[circle, fill=black, inner sep=1.5pt] at (\point) {};
        \node[\pos] at (\point) {$\point$};
    }
    \node[circle, fill=black, inner sep=1.5pt] at (a_3) {};
    \node[below,xshift=-2pt,yshift=-6pt] at (a_3) {$a_3$};
    \node[above,xshift=-2pt,yshift=6pt] at (a_5) {$a_5$};

    \foreach \point in {a_4,a_7,a_5} {
        \node[circle, draw=black, fill=black, inner sep=1.5pt, minimum size=4pt] at (\point) {};
    }
    \node[below right] at (a_4) {$a_4$};
    \node[below] at (a_7) {$a_7$};

    \node[red] at (-2,0) {$P_{a_0a_7}$};
    \node[purple] at (-2,3.5) {$P_{a_1a_6}$};
    \node[blue] at (3.2,1.1) {$P_{a_2a_5}$};

    \node[below] at (1.5,-1) {(b): $\Psi$};
    \end{tikzpicture}
\end{subfigure}
\caption{Graphs $\Psi_i,\Psi$}
\label{graphOf8vIntro}
\end{figure}
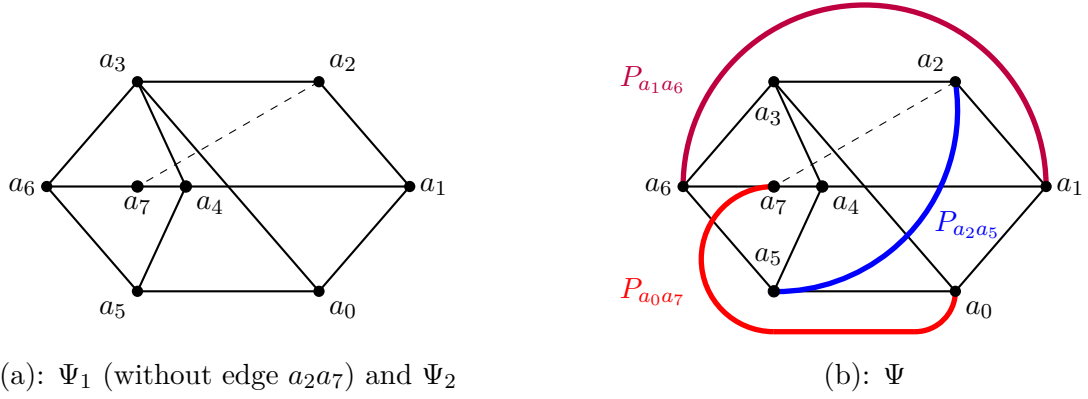
    

    \begin{lem}\label{8vGraphNotExist}
        If $G\in\mathcal{G}_{n,k}$, then $G$ contains no induced copy of $\Psi_1$ or $\Psi_2$.
    \end{lem}

    \begin{proof}
        Suppose, to the contrary, that $G$ contains an induced copy $H$ (say) of $\Psi_1$ or $\Psi_2$, labeled as shown in Figure~\ref{graphOf8vIntro}a. Note that $a_1a_6,a_2a_5,a_0a_7\notin E(G)$ and 
        $dist_{H}(a_1,a_6)=dist_{H}(a_2,a_5)=dist_{H}(a_0,a_7)=3$.
       Since $G$ is a maximal $\{C_3,\ldots, C_{2k-1}\}$-free graph, it follows that there
   exist three $(a_1,a_6),(a_2,a_5),(a_0,a_7)$-paths of length $2k-2$ in $G$, and we denote them by $P_{a_1a_6},P_{a_2a_5},P_{a_0a_7}$, respectively.

\begin{claim}\label{8vNoai}
For each $P_{a_ia_j}\in \{P_{a_1a_6},P_{a_2a_5},P_{a_0a_7}\}$, we have $V(H)\cap V(P_{a_ia_j})=\{a_i,a_j\}$.
\end{claim}

 \begin{proof}

Since  every vertex of $H$ is contained  in some $(a_1,a_6)$-path of length three in $H$, it follows that
$V(H)\cap V(P_{a_1a_6})= \{a_1,a_6\}$ by the assumption of odd girth.
Similarly, $V(H)\cap V(P_{a_2a_5})\subseteq \{a_2,a_5,a_7\}$ 
because   every vertex in $V(H)\setminus \{a_7\}$ is contained in  some $(a_2,a_5)$-path of length three in $H$.
 Moreover, we have $V(H)\cap V(P_{a_2a_5})= \{a_2,a_5\}$. Indeed, suppose that 
$a_7$ is on the path $P_{a_2a_5}$. If the sub-path $a_7P_{a_2a_5}a_5$ of $P_{a_2a_5}$
has odd length, then 
the odd path $a_7P_{a_2a_5}a_5$ has length at most $2k-3$, but this together with the path $a_5a_6a_7$ yields an odd cycle of length at most $2k-1$, a contradiction.
If the path $a_7P_{a_2a_5}a_5$ has even length, then
the  sub-path $a_2P_{a_2a_5}a_7$ of $P_{a_2a_5}$ has even  length at most $2k-4$, but this together with the path $a_2a_3a_6a_7$ yields an odd cycle of length at most $2k-1$, a contradiction.

Since every vertex in $V(H)\setminus \{a_2\}$ lies on some $(a_0,a_7)$-paths of length three in $H$, we conclude that 
$V(H)\cap V(P_{a_0a_7})\subseteq \{a_0,a_2,a_7\}$ by the assumption of odd girth.
Similarly to in the previous case, one can show that
 $V(H)\cap V(P_{a_0a_7})= \{a_0,a_7\}$. This proves the claim. 
 \end{proof}

\begin{claim}\label{8v2k-2disjoint}
The paths $P_{a_0a_7},P_{a_1a_6},P_{a_2a_5}$
are pairwise vertex-disjoint. 
\end{claim}
\begin{proof}
Note that $G[\{a_0,a_1,a_4,a_5,a_6,a_7\}]$ is an induced copy of $C_{6,1}$ with the rim cycle $a_0a_1a_4a_7a_6a_5a_0$ and diagonal chord
$a_4a_5$.
By Lemma~\ref{C61hasphi4k3}, $P_{a_0a_7}$ and $P_{a_1a_6}$ cannot intersect.
Notice also that
$G[\{a_0,a_1,a_2,a_3,a_5,a_6\}]$ is an induced copy of $C_{6,1}$ with rim cycle $a_0a_1a_2a_3a_6a_5a_0$ and diagonal chord
$a_0a_3$. Again by   Lemma~\ref{C61hasphi4k3},
$P_{a_2a_5}$ and $P_{a_1a_6}$ cannot intersect. Thus, we shall prove that
 $P_{a_0a_7}$ and $P_{a_2a_5}$ cannot intersect.

Suppose that  $P_{a_0a_7}$ and $P_{a_2a_5}$
have a common vertex $u$. If the two paths $uP_{a_0a_7}a_7$ and $uP_{a_2a_5}a_5$ have different
parity, then $uP_{a_0a_7}a_7a_6a_5P_{a_2a_5}u,
uP_{a_2a_5}a_2a_1a_0P_{a_0a_7}u$ are two closed odd walks whose lengths sum up to $4k$. Consequently, one of the walks has length at most $2k-1$ which yields an odd cycle of length at most $2k-1$, a contradiction.
If the two paths $uP_{a_0a_7}a_7,uP_{a_2a_5}a_5$ have the same parity, then  the four  paths $uP_{a_0a_7}a_7,uP_{a_2a_5}a_5,a_0P_{a_0a_7}u$ and $a_2P_{a_2a_5}u$ have the same parity. Together with the paths $a_2a_3a_4a_7$ and $a_0a_5$, this results in two closed odd walks $uP_{a_2a_5}a_2a_3a_4a_7P_{a_0a_7}u$ and $uP_{a_2a_5}a_5a_0P_{a_0a_7}u$ whose lengths sum up to $4k$. Consequently, one of the walks has length at most $2k-1$ which yields  an odd cycle of length at most $2k-1$, a contradiction. This proves the claim.
\end{proof}

Let $P_{a_0a_7}=v_0v_1v_2\ldots v_{2k-3}v_{2k-2}$ with $a_0=v_0,a_7=v_{2k-2}$, let 
$ P_{a_1a_6}=u_0u_1u_2\ldots u_{2k-3}\\u_{2k-2}$ with $a_1=u_0,a_6=u_{2k-2}$, and let
  $P_{a_2a_5}=w_0w_1w_2\ldots w_{2k-3}w_{2k-2}$
  with $a_2=w_0,a_5=w_{2k-2}$.
Define $\Psi=P_{a_0a_7}\cup P_{a_2a_5}\cup P_{a_1a_6}\cup H$, as shown in Figure~\ref{graphOf8vIntro}b.
By Claims~\ref{8vNoai} and~\ref{8v2k-2disjoint}, $\Psi$ has exactly $6k-1$ vertices.

    \begin{claim}\label{8vertexfour}
        For any vertex $x$ in $G$, we have $|N(x)\cap V(\Psi)|\leq 4$. 
    \end{claim}
    \begin{proof}
         Suppose, 
         to the contrary, that there exists a vertex $x$ in $G$ with five neighbors, say $y_1,\ldots,y_5$,
         in $\Psi$.
Let $C:=a_0P_{a_0a_7}a_7a_4a_1P_{a_1a_6}a_6a_3a_0$ be a cycle of length $4k$ in $\Psi$ shown in Figure~\ref{graphOf8v'sPhi4k3}a.
Clearly, $C\cup\{a_0a_1,a_3a_4,a_6a_7\}$ is a copy of $\Phi_{4k,3}$ with rim cycle $C$.
 By Lemma~\ref{phi4k3ism4k},  $G[V(C)]$ is a copy of $M_{4k}$. Therefore, by the odd girth assumption, $C$ contains at most three of $y_1,\ldots,y_5$.
Since the path $P_{a_2a_5}$ is contained in some odd cycles of length $2k+1$ in $\Psi$, it follows that $x$ has at most two neighbors on the path  $P_{a_2a_5}$. Moreover, since  $V(\Psi)=V(C)\cup V(P_{a_2a_5})$, it follows that  $C$ contains exactly three, of $y_1,\ldots,y_5$, say $y_1,y_2,y_3$, such that $\{y_1,y_2,y_3\}$ is a diagonal triple
of $C$, and $P_{a_2a_5}$ contains $y_4,y_5$. Let $D:=a_1P_{a_1a_6}a_6a_3a_4a_1$ be an odd cycle of length $2k+1$ in $G[V(C)]$. Note that $V(C)=V(D)\cup V(P_{a_0a_7})$. Since 
$y_1,y_2,y_3\in V(C)$ and 
$P_{a_0a_7}$ is contained in some odd cycle of length $2k+1$ in $\Psi$, it follows that $P_{a_0a_7}$ 
contains at most two of $y_1,y_2,y_3$ and therefore $D$ contains at least one of $y_1,y_2,y_3$.


    \begin{figure}[h]
    \centering

    \begin{subfigure}{0.45\textwidth}
    \centering
    \begin{tikzpicture}[scale=0.8]
    \coordinate (a_5) at (0,0);
    \coordinate (a_0) at (3,0);
    \coordinate (a_1) at (4.5,1.732);  
    \coordinate (a_2) at (3,3.464);  
    \coordinate (a_3) at (0,3.464);
    \coordinate (a_6) at (-1.5,1.732);
    \coordinate (a_7) at (0,1.732);
    \coordinate (a_4) at (0.8,1.732);
    
    \draw[black] (a_5) -- (a_0); 
    \draw[black] (a_1) -- (a_2) -- (a_3) -- (a_6) -- (a_5);

    \draw[black] (a_5) -- (a_4);
    \draw[dashed] (a_2) -- (a_7);

    \draw[red,line width=1.5pt](-1.53,1.732)arc(180:1:3.03);
    \draw[cyan,line width=1.5pt](-1.47,1.732)arc(180:1:2.97);
    
    \draw[black](a_2)arc(9:-90:3);
    \draw[red,line width=1.5pt](a_7)arc(90:270:1.2);
    \coordinate (w1) at (0,-0.668);
    \coordinate (w2) at (2.332,-0.668);
    \draw[red,line width=1.5pt] (w1) -- (w2);
    \draw[red,line width=1.5pt](a_0)arc(0:-90:0.668);

    \draw[cyan,line width=1.4pt] (4.5,1.772) -- (0.8,1.772);
    \draw[red,line width=1.4pt] (4.5,1.692) -- (0.8,1.692);

    \draw[red,line width=1.5pt] (a_4) -- (a_7);
    \draw[red,line width=1.5pt] (a_0) -- (0,3.534);
    
    \draw[cyan,line width=1.5pt] (0,3.524) -- (-1.5,1.792);
    \draw[red,line width=1.5pt] (0,3.414) -- (-1.5,1.682);

    \draw[cyan,line width=1.5pt] (0.01,3.514) -- (0.81,1.782);
    \draw[black,line width=1.5pt] (-0.0462,3.454) -- (0.7492,1.732);

    \draw[black,line width=1.5pt] (a_1) -- (a_0);
    \draw[black,line width=1.5pt] (a_7) -- (a_6);

    \foreach \point/\pos in {a_2/above left,a_0/below right, a_1/right, a_6/left} {
        \node[circle, fill=black, inner sep=1.5pt] at (\point) {};
        \node[\pos] at (\point) {$\point$};
    }
    \node[circle, fill=black, inner sep=1.5pt] at (a_3) {};
    \node[below,xshift=-2pt,yshift=-6pt] at (a_3) {$a_3$};

    \node[above,xshift=-2pt,yshift=6pt] at (a_5) {$a_5$};
    
    \foreach \point in {a_4,a_5} {
        \node[circle, draw=black, fill=black, inner sep=1.5pt, minimum size=4pt] at (\point) {};
    }
    
    \foreach \point in {a_4,a_7} {
        \node[circle, draw=black, fill=black, inner sep=1.5pt, minimum size=4pt] at (\point) {};
    }
    \node[below right] at (a_4) {$a_4$};
    \node[below] at (a_7) {$a_7$};

    \node[red] at (-2,0) {\Large\textbf{$C$}};
    \node[cyan] at (-2,4) {\Large\textbf{$D$}};

    \node[below] at (1.5,-1) {(a)};
    \end{tikzpicture}
\end{subfigure}
\hfill
\begin{subfigure}{0.45\textwidth}
    \centering
    \begin{tikzpicture}[scale=0.8]
    \coordinate (a_5) at (0,0);
    \coordinate (a_0) at (3,0);
    \coordinate (a_1) at (4.5,1.732);  
    \coordinate (a_2) at (3,3.464);  
    \coordinate (a_3) at (0,3.464);
    \coordinate (a_6) at (-1.5,1.732);
    \coordinate (a_7) at (0,1.732);
    \coordinate (a_4) at (0.8,1.732);
    
    \draw[black] (a_5) -- (a_0) -- (a_1);
    \draw[black] (a_2) -- (a_3) -- (a_6);
    \draw[black] (a_4) -- (a_6);
    
    \draw[black] (a_3) -- (a_0);
    \draw[dashed] (a_2) -- (a_7);

    \draw[red,line width=2pt](a_2)arc(9:-90:3);
    \draw[black](a_7)arc(90:270:1.2);
    \coordinate (w1) at (0,-0.668);
    \coordinate (w2) at (2.332,-0.668);
    \draw[black] (w1) -- (w2);
    \draw[black](a_0)arc(0:-90:0.668);

    \draw[red,line width=2pt] (a_4) -- (a_5);
    \draw[red,line width=2pt] (a_2) -- (a_3) -- (a_6);

    \draw[red,line width=1.5pt](-1.53,1.732)arc(180:1:3.03);
    \draw[cyan,line width=1.5pt](-1.47,1.732)arc(180:1:2.97);

    \draw[cyan,line width=1.2pt] (4.5,1.752) -- (0.8,1.752);
    \draw[red,line width=1.2pt] (4.5,1.712) -- (0.8,1.712);

    \draw[cyan,line width=1.5pt] (0,3.514) -- (-1.5,1.782);
    \draw[red,line width=1.5pt] (0,3.414) -- (-1.5,1.682);

    \draw[cyan,line width=1.5pt] (0.01,3.514) -- (0.81,1.782);
    \draw[black,line width=1.5pt] (-0.0462,3.454) -- (0.7492,1.732);

    \draw[black,line width=1.5pt] (a_1) -- (a_2);
    \draw[black,line width=1.5pt] (a_5) -- (a_6);

    \foreach \point/\pos in {a_7/below,a_0/below right,a_1/right, a_2/above left, a_6/left} {
        \node[circle, fill=black, inner sep=1.5pt] at (\point) {};
        \node[\pos] at (\point) {$\point$};
    }
    \node[circle, fill=black, inner sep=1.5pt] at (a_3) {};
    \node[below,xshift=-2pt,yshift=-6pt] at (a_3) {$a_3$};
    \node[above,xshift=-2pt,yshift=6pt] at (a_5) {$a_5$};
    
    \foreach \point in {a_4,a_5} {
        \node[circle, draw=black, fill=black, inner sep=1.5pt, minimum size=4pt] at (\point) {};
    }
    \node[below right] at (a_4) {$a_4$};

\node[cyan] at (-2,4) {\Large\textbf{$D$}};
    \node[red] at (3.3,1) {\Large\textbf{$C'$}};

    \node[below] at (1.5,-1) {(b)};
    \end{tikzpicture}
\end{subfigure}
\caption{$\Psi$}
\label{graphOf8v'sPhi4k3}
\end{figure}
    

On the other hand, let $C':=a_2P_{a_2a_5}a_5a_4a_1P_{a_1a_6}a_6a_3a_2$ be a cycle of length $4k$ in $\Psi$ shown in Figure~\ref{graphOf8v'sPhi4k3}b. Clearly, $C'\cup\{a_1a_2,a_3a_4,a_5a_6\}$ is a copy of $\Phi_{4k,3}$ with rim cycle $C'$.
 By Lemma~\ref{phi4k3ism4k},  $G[V(C')]$ is also a copy of $M_{4k}$, which implies that $C'$ contains $y_4,y_5$, and at most one of $y_1,y_2,y_3$.
Since $V(C')=V(D)\cup V(P_{a_2a_5})$ and $y_4,y_5\in V(P_{a_2a_5})$, it follows that $D$ contains at most one of $y_1,y_2,y_3$.  So, $D$ contains exactly one of $y_1,y_2,y_3$, say $y_3$. Since $y_3,y_4,y_5\in V(C')$, $\{y_3,y_4,y_5\}$ forms a diagonal triple of $C'$. 
Without loss of generality, we may assume that $\{y_4,y_5\}=\{w_{j-1},w_{j+1}\}$ for some $j\in\{1,2,\ldots,2k-3\}$, then  
$y_3$ is on the path $P_{a_1a_6}$, and therefore, 
$y_3=u_j$. This implies that $D$ cannot contain $y_1$ or $y_2$
because $D$ contains exactly one of $y_1,y_2,y_3$.
Recall that $\{y_1,y_2,y_3\}$ is  a diagonal triple of $C$.
Then $y_1,y_2$ are on the path $P_{a_0a_7}$,
and  $\{y_1,y_2\}=\{v_{j-1},v_{j+1}\}$.
Thus, the path $v_{j-1}P_{a_0a_7}a_0a_5P_{a_2a_5}w_{j+1}$
has length $(j-1)+1+(2k-2-(j+1))=2k-3$. Together with the path $v_{j-1}xw_{j+1}$, this yields an odd cycle of length $2k-1$, a contradiction.
This proves the claim.
        \end{proof}

        It follows from Claim~\ref{8vertexfour} that 
        $4n=\frac{4n}{6k-1}(6k-1)<|V(\Psi)|\delta(G)\leq\sum_{x\in V(\Psi)}|N(x)|=\sum_{x\in V(G)}|N(x)\cap V(\Psi)|\leq4n$, a contradiction.
    \end{proof}
    

We remark that the following result is concluded from the proof of \cite[Theorem 1.2]{MessutiC2k+1}.

\begin{lem}[\cite{MessutiC2k+1}]\label{NoC61AndTIsC2k+1}
      Let $k\geq 2$ be an integer, and 
    let $G$ be a maximal $\{C_3,\ldots,C_{2k-1}\}$-free graph. If $G$ contains neither an induced 
     copy of $C_{6,1}$ nor a tetrahedron in $\mathscr{T}_1$ (not necessarily  induced), then $G$ is homomorphic to $C_{2k+1}$.
    \end{lem}
    
Now we present the proof of  Theorem~\ref{MainResult}.
   	\begin{proof}[Proof of Theorem~\ref{MainResult}] 
        Let $G$ be a graph in $\mathcal{G}_{n,k}$. 
        If $G$ contains neither an induced 
     copy of $C_{6,1}$ nor a tetrahedron in $\mathscr{T}_1$ (not necessarily  induced), then by Lemma~\ref{NoC61AndTIsC2k+1},   
     $G$ is homomorphic to $C_{2k+1}$, and hence homomorphic to $M_{4k}$.        
     Hence we may assume that $G$ contains an induced copy of $C_{6,1}$ or a tetrahedron in $\mathscr{T}_1$. By Corollaries~\ref{C} and~\ref{T}, $G$ contains an induced copy of $M_{4k}$.  We denote by $\mathcal{M}$  a vertex-maximal blow-up of $M_{4k}$ contained in $G$.
Choose a representative induced copy \(M = G[\{v_0, \ldots, v_{4k-1}\}]\) of \(M_{4k}\) in \(\mathcal{M}\), as labeled 
in Figure~\ref{graphOfIntroM4k+Phi4k3}a, and let  
 $V_0,\ldots, V_{4k-1}$ be the vertex classes of $\mathcal{M}$ such that $v_i$ corresponds to $V_i$.
Clearly, each $V_i$ is an independent set. We shall prove $G=\mathcal{M}$. 
Suppose not, and let $x$ be a vertex in $V(G)\setminus V(\mathcal{M})$.
If $x$  has no neighbor in $\mathcal{M}$ then $x$ would be disconnected from $\mathcal{M}$, which violates the maximality of $G$.
So, $x$ has at least one neighbor in $\mathcal{M}$.
        Owing to the odd girth assumption on $G$, $x$ can have at most three neighbors in   $M$, which implies that the neighbors of $x$ lie in at most three of the vertex classes of $\mathcal{M}$.
If they lie in three such classes, then those classes must be of the form $\{V_{i-1},V_{i+1},V_{i+2k}\}$ for some $i=0,1,\ldots,4k-1$.

    \begin{claim}\label{M4kLessThan3neighbors}
       The neighbors of $x$ in $\mathcal{M}$ cannot lie in exactly three vertex classes of $\mathcal{M}$.
    \end{claim}
    \begin{proof}
        Without loss of generality, we may assume that the neighbors of $x$ in $\mathcal{M}$ belong to $V_{2k+1},V_{2k-1},V_0$, and that $v_{2k+1},v_{2k-1},v_0$ are  three such  neighbors. Then there exists  a vertex $y\in V_{2k+1}\cup V_{2k-1}\cup V_{0}$ such that $xy\notin E(G)$; otherwise, $x$ would belong to $\mathcal{M}$
        because  $\mathcal{M}$ is a vertex-maximal blow-up of $M_{4k}$.
        Note that the vertex set $\{x,y,v_{2k},v_0,v_{2k+1},v_1,v_{4k-1},v_{2k-1}\}$ induces a copy of $\Psi_2$ if $y\in V_0$, and a copy of $\Psi_1$ if $y\in V_{2k-1}\cup V_{2k+1}$. But this contradicts  Lemma~\ref{8vGraphNotExist}.
        This proves the claim.
    \end{proof}
    
    It follows from Claim~\ref{M4kLessThan3neighbors} that the neighbors of $x$ in $\mathcal{M}$ belong to at most two of the vertex classes of $\mathcal{M}$.
    We now distinguish two cases to obtain a contradiction.

{\bf \noindent Case 1.} The neighbors of $x$ in $\mathcal{M}$ belong to exactly one vertex class of $\mathcal{M}$.

Without loss of generality, we may assume that the neighbors of $x$ in $\mathcal{M}$ lie in $V_{2k}$, and that $v_{2k}$ is  the unique  neighbor of $x$ in $M$. Let $D_1:=v_0v_1v_{2k+1}v_{2k+2}\ldots v_{4k-1}v_0$. Clearly, $xv_{2k}v_0$ is a path of length two and $x$ has no neighbors in $D_1$. By Lemma~\ref{xCycleIsT}, there exist $(x,v_1),(x,v_{4k-1})$-paths, say $P_{xv_1},P_{xv_{4k-1}}$, respectively, of length $2k-2$ in $G$ such that 
$T:=D_1\cup P_{xv_1}\cup P_{xv_{4k-1}}\cup \{xv_{2k},v_{2k}v_0\}$ is a tetrahedron
in $\mathcal{T}_k$ with center $v_0$.
Let $D_2:=xP_{xv_1}v_1v_0v_{2k}x$ and $D_3:=xP_{xv_{4k-1}}v_{4k-1}v_0v_{2k}x$ be the other two odd cycles of length $2k+1$ in $T$, as shown in Figure~\ref{graphOfMainTheoremCase1}a. Clearly, $D_1,D_2,D_3$ are all facial cycles in $T$, and therefore, 
the three spokes of $T$ are $D_1\cap D_2,D_2\cap D_3$ and $D_1\cap D_3$. We denote by $a,b$ and $c$ the branch vertices lying on these spokes, respectively.



    \begin{claim}\label{InsectD2Only1}
      $V(P_{xv_1})\cap V(M)\subseteq \{v_1,\ldots,v_{2k-1}\}$ and   $V(P_{xv_{4k-1}})\cap V(M)\subseteq \{v_{2k+1},\ldots,v_{4k-1}\}$.   
      \end{claim}
\begin{proof}
 By symmetry, it suffices to show that $V(P_{xv_1})\cap V(M)\subseteq \{v_1,\ldots,v_{2k-1}\}$. 
 Clearly, $v_0,v_{2k}\notin V(P_{xv_1})$
 by the assumption of odd girth.
 Suppose, to the contrary, that $v_{2k+i}\in V(P_{xv_1})$ for some  $1\leq i\leq 2k-1$.    
Since  $v_0v_1\in E(D_1)\cap E(D_2)$, 
by Lemma~\ref{2C2k+1isspoke}, 
the path $v_1v_{2k+1}\ldots v_{2k+i}$ and the path
$v_{2k+i}P_{xv_1}v_1$ have the same length by Lemma~\ref{2C2k+1isspoke}.
We replace
 $v_{2k+i}P_{xv_1}v_1$ of $D_2$ by
 $v_1v_{2k+1}\ldots v_{2k+i}$ and then obtain an odd walk  $W:=xv_{2k}v_0v_1v_{2k+1}v_{2k+2}\ldots v_{2k+i}P_{xv_1}x$ of length $2k+1$ in $G$. But then $v_{2k}v_{2k+1}$ is a chord in the walk $W$, which yields an odd cycle of length at most $2k-1$,
 a contradiction. This proves the claim.
 \end{proof}

\begin{figure}[h]
\centering
\begin{minipage}{0.37\textwidth}
\centering
\begin{tikzpicture}[scale=0.7]
\coordinate (v_{4k-1}) at (-1,0);
\coordinate (v_{2k+1}) at (4,0);
\coordinate (v_{2k}) at (4.5,1.732);  
\coordinate (v_{2k-1}) at (4,3.464);  
\coordinate (v_1) at (-1,3.464);
\coordinate (v_0) at (-2,1.732);


\draw[red,thick] (v_{2k}) -- (v_0);

\coordinate (v_2) at (0,3.464);  
\coordinate (v_{2k-2}) at (3,3.464);  

\coordinate (v_s) at (2,3.464);  
\coordinate (v_{4k-2}) at (0,0);  
\coordinate (v_{2k+2}) at (3,0);  
\coordinate (v_t) at (1.9,0);

\draw[thick] (v_{4k-1}) -- (v_{2k+1});

\coordinate (b) at (7,2.2); 
\coordinate (x) at (5.3,1.8818);

\draw[thick] (v_{2k+1}) -- (v_{2k});
\draw[thick, dashed] (v_{2k}) -- (v_{2k-1}) -- (v_s);

\draw[red,line width=1.5pt](b)arc(0:-135:3.031);
\draw[blue,line width=1.5pt](b)arc(20:133:3.13);

\coordinate (v_{4k-1}'') at (-1,-0.05);
\coordinate (v_{4k-1}') at (-1,0.05);
\coordinate (v_t'') at (2,-0.05);
\coordinate (v_t') at (2,0.05);
\coordinate (b') at (6.97,2.23);
\coordinate (b'') at (7.03,2.17);
\coordinate (v_{2k}') at (4.5,1.762);
\coordinate (v_{2k}'') at (4.5,1.702);
\coordinate (v_0') at (-2,1.762);
\coordinate (v_0'') at (-2,1.702);
\coordinate (v_{2k+1}') at (4,0.05);
\coordinate (v_{2k+1}'') at (4,-0.05);
\coordinate (v_1') at (-1,3.494);
\coordinate (v_1'') at (-1,3.434);

\coordinate (v_0''') at (-1.98,1.702);
\coordinate (v_0'''') at (-2.02,1.762);

\draw[dashed] (v_{4k-1}) -- (v_{2k-1});
\draw[blue,line width=1.5pt] (b') -- (v_{2k}');
\draw[red,line width=1.5pt] (b'') -- (v_{2k}'');

\draw[red,line width=1.5pt] (v_{4k-1}') -- (v_t');
\draw[blue,line width=1.5pt] (v_1) -- (v_s);

\draw[red,line width=1.5pt] (v_{2k}'') -- (v_0'') -- (v_{4k-1}'') -- (v_t'');
\draw[blue,line width=1.5pt] (v_{2k}') -- (v_0') -- (v_1');
\draw[green,line width=1.5pt] (v_{2k+1}') -- (v_1'') -- (v_0'');
\draw[green,line width=1.5pt] (v_0') -- (v_{4k-1}') -- (v_{2k+1}');

\foreach \point/\pos in {v_{4k-1}/below left, v_{2k+1}/below, v_{2k}/above, v_0/left} {
    \node[circle, fill=black, inner sep=1.5pt] at (\point) {};
    \node[\pos] at (\point) {$\point$};
}
\node[circle, draw=black, fill=black, inner sep=1.5pt, minimum size=4pt] at (b) {};
\node[right] at (b) {$b$};

\node[circle, draw=black, fill=black, inner sep=1.5pt, minimum size=4pt] at (v_1) {};
\node[above left] at (v_1) {$a=v_1$};

\foreach \point in {v_s,v_t,x} {
    \node[circle, draw=black, fill=black, inner sep=1.5pt, minimum size=4pt] at (\point) {};
}
\node[above] at (v_t) {$c$};
\node[above] at (x) {$x$};

%
%
%
%
%


\node[] at (1.5,-2) {(a)};
\node[] at (1,-0.8) {\Large\textcolor{red}{$D_3$}};
\node[] at (-3,2.7) {\Large\textcolor{green}{$D_1$}};
\node[] at (2.6,2.7) {\Large\textcolor{blue}{$D_2$}};
\end{tikzpicture}
\end{minipage}
\hspace{0.2\textwidth} 
\begin{minipage}{0.37\textwidth}
\centering
\begin{tikzpicture}[scale=0.7]

\coordinate (c) at (0,0);
\coordinate (b) at (8,0);
\coordinate (v_1) at (4,4);  
\coordinate (v_0) at (4,1.5); 
\coordinate (v_{2k}) at (5,1.125); 
\coordinate (v_{2k+1}) at (3.3,3.3);

\coordinate (v_1') at (3.97,4); 
\coordinate (v_1'') at (4.03,4); 
\coordinate (v_0') at (3.97,1.5); 
\coordinate (v_0'') at (4.03,1.5); 
\coordinate (y) at (7.93,0.03);
\coordinate (y') at (7.93,-0.03);
\coordinate (v_{4k-1}') at (0,0.03);
\coordinate (v_{4k-1}'') at (0,-0.03);

\draw[black, thick] (v_0) -- (c);
\draw[black, thick] (v_0) -- (b);
\draw[black, thick] (v_0) -- (v_1);

\draw[red, line width=1.5pt] (4,1.45) -- (0,-0.05);
\draw[red, line width=1.5pt] (0,0) -- (8,0);
\draw[red, line width=1.5pt] (8,-0.05) -- (4,1.45);

\draw[blue, line width=1.5pt] (8,0) -- (4,4);
\draw[blue, line width=1.5pt] (4.05,4) -- (4.05,1.5);
\draw[blue, line width=1.5pt] (4,1.55) -- (8,0.05);

\draw[green, line width=1.5pt] (3.95,4) -- (3.95,1.5);
\draw[green, line width=1.5pt] (4,1.55) -- (0,0.05);
\draw[green, line width=1.5pt] (0,0) -- (4,4);

\draw[] (v_{2k}) -- (v_{2k+1});


\foreach \point/\pos in {v_0/below,  v_1/above, v_{2k}/below, v_{2k+1}/above left} {
    \node[circle, fill=black, inner sep=1.5pt] at (\point) {};
    \node[\pos] at (\point) {$\point$};
}
\node[circle, fill=black, inner sep=1.5pt] at (b) {};
    \node[below right] at (b) {$b$};
\node[circle, fill=black, inner sep=1.5pt] at (c) {};
    \node[below left] at (c) {$c$};
\node[circle, fill=black, inner sep=1.5pt] at (6,0.75) {};
    \node[below] at (6,0.75) {$x$};
\node[circle, fill=black, inner sep=1.5pt] at (3,1.125) {};
    \node[below,yshift=-4pt] at (3,1.125) {\small$v_{4k-1}$};

\node[] at (4,-2) {(b)};
\node[] at (2.8,2.1) {\Large\textcolor{green}{$D_1$}};
\node[] at (4.9,2.4) {\Large\textcolor{blue}{$D_2$}};
\node[] at (4,-0.8) {\Large\textcolor{red}{$D_3$}};
\end{tikzpicture}

\end{minipage}

\caption{$T\cup\{v_{2k}v_{2k+1}\}$}
\label{graphOfMainTheoremCase1}
\end{figure}
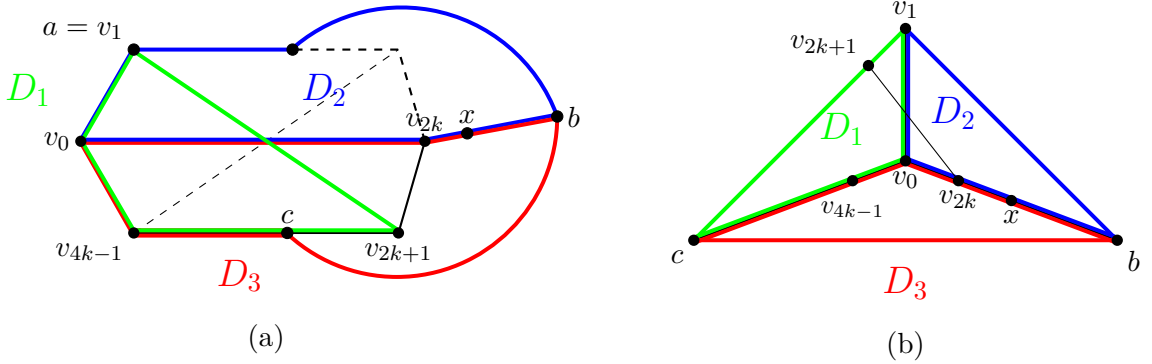

\begin{claim}\label{abcPos}
   We have $a=v_1,c=v_{4k-1},b=x$. Consequently, the spokes of $T$ are $v_0v_1,v_0v_{4k-1},v_0v_{2k}x$.  
      \end{claim}
\begin{proof}
By Claim~\ref{InsectD2Only1}, we have 
$D_1\cap D_2=v_0v_1$, and hence $a=v_1$ because $a$ is the branch vertex lying on the spoke $D_1\cap D_2$ in $T$.
Note that the rim cycle of $T$ is $C:=bP_{xv_1}v_1v_{2k+1}v_{2k+2}\ldots cP_{xv_{4k-1}}b$ and the spokes of $T$ are the edge $v_0v_1$ and the paths $v_0v_{2k}xP_{xv_1}b,v_0v_{4k-1}v_{4k-2}\ldots c$, as shown in Figure~\ref{graphOfMainTheoremCase1}b. Clearly, the spoke path $v_0v_{2k}xP_{xv_1}b$ in $T$ has length at least two.
Suppose first that $c\neq v_{4k-1}$.
Then 
$\ell(v_0v_{4k-1}P_{xv_{4k-1}}c)\geq 2$, and hence $T\in\mathscr{T}_1$.
Since the spoke $v_0v_1$ in $T$ is contained in two facial cycles $D_1,D_2$,
$G[V(D_1)\cup V(D_2)]$ is an induced copy of $M_{4k}$ by Lemma~\ref{T1spokeisM4k}, 
which implies that $xv_{2k+2}$ is a diagonal chord in 
 $G[V(D_1)\cup V(D_2)]$ shown in Figure~\ref{graphOfMainTheoremCase1}b.
 But this contradicts the fact that $v_{2k}$ is the unique neighbor of $x$ in $M$.

Now suppose that 
$x\neq b$.
Then 
$\ell(v_{2k}xP_{xv_1}b)\geq2$. Note that
$T\cup \{v_{2k}v_{2k+1}\}\setminus \{v_0v_1\}$ is a new tetrahedron in $\mathscr{T}_1$
 with rim cycle $C$, center $v_{2k}$ and branch vertices $v_{2k+1},b,c$. Then $v_{2k}v_{2k+1}$ is a spoke in this new tetrahedron. 
We conclude, again by Lemma~\ref{T1spokeisM4k},
that $G[V(D_1)\cup V(D_2)]$ is an induced copy of $M_{4k}$, which implies that $xv_{2k+2}$ is a diagonal chord in 
 $G[V(D_1)\cup V(D_2)]$. But this contradicts the fact that $v_{2k}$ is the unique neighbor of $x$ in $M$. This proves the claim.
\end{proof}

It follows from Claim~\ref{abcPos}
that the spokes of 
$T$ are $v_0v_1,v_0v_{4k-1},v_0v_{2k}x$,
and hence $D_1\cup D_2,D_1\cup D_3$ are copies of $\Phi_{4k,1}$.
 If $D_1\cup D_2$ is a diagonal $\Phi_{4k,1}$, then by Corollary~\ref{phi4k1isM4kcor},  $G[V(D_1)\cup V(D_2)]$ is 
 an induced copy of $M_{4k}$, which 
 implies that $xv_{2k+2}$ is a diagonal chord of $G[V(D_1)\cup V(D_2)]$. But this contradicts the fact that $v_{2k}$ is the unique neighbor of $x$ in $M$.
 So,  $D_1\cup D_2$ is not a diagonal  $\Phi_{4k,1}$. Then there exists a vertex $y$ in $G$ that has 
 three neighbors $z_1,z_2,z_3$ in $D_1\cup D_2$ such that $\{z_1,z_2,z_3\}$ is neither a diagonal triple of $D_1\oplus D_2$ nor a special triple of $D_1\cup D_2$.

 Without loss of generality, we may assume that $D_1$ contains exactly two of $z_i$, say $z_1,z_2$, and $D_2$ contains $z_3$.
If $\{z_1,z_2\} \cap \{v_0,v_1\}  \neq \emptyset$, then  by symmetry,
we may assume that $z_1 = v_0$. Since $dist_{D_1}(z_1,z_2)=2$,
it follows that 
 $z_2 \in \{v_{4k-2},v_{2k+1}\}$. Similarly,
 $z_3$ is either the neighbor of $v_1$ on $P_{xv_1}$, or $z_3=x$ as $dist_{D_2}(z_1,z_3)=2$. However, $\{z_1,z_2,z_3\}$ is a special triple of $D_1\cup D_2$ if $z_2=v_{4k-2}$ and $z_3=x$;
otherwise, it is a diagonal triple of $D_1\oplus D_2$, which contradicts our assumption. 


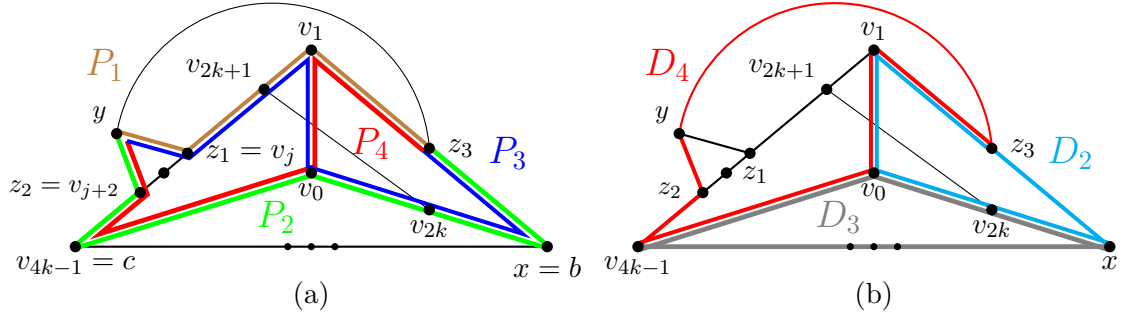
\begin{figure}[h]
\centering
\begin{subfigure}{0.48\textwidth}
\centering
\begin{tikzpicture}[scale=0.65]
\coordinate (v_{4k-1}) at (0,0);
\coordinate (x) at (9.6,0);
\coordinate (v_1) at (4.8,4);  
\coordinate (v_0) at (4.8,1.5); 
\coordinate (v_{2k}) at (7.2,0.75); 
\coordinate (v_{2k+1}) at (3.84,3.2); 

\coordinate (z_2) at (1.32,1.1);
\coordinate (z) at (1.8,1.5);
\coordinate (z_1) at (2.28,1.9);
\coordinate (y) at (0.84,2.3);
\coordinate (z_3) at (7.2,2);

\draw[black, thick] (z_2) -- (z_1) ;
\draw[] (y) arc (170:5:3.211);
\draw[black, thick] (v_{4k-1}) -- (x);
\draw[] (v_{2k}) -- (v_{2k+1});

\draw[brown, line width=1.5pt] (0.84,2.35) -- (2.28,1.95) -- (4.8,4.05) -- (7.2,2.03);
\draw[green, line width=1.8pt] (0.84,2.27) -- (1.284,1.1) -- (0,0.03);
\draw[green, line width=1.8pt] (0.036,-0.02) -- (4.8,1.45) -- (9.564,-0.02);
\draw[green, line width=1.8pt] (9.6,0.03) -- (7.2,2.03);
\draw[red, line width=1.8pt] (1.0506,2.1515) -- (1.464,1.062) -- (0.504,0.262) -- (4.872,1.63) -- (4.872,3.8) -- (7.08,1.9505);
\draw[blue, line width=1.5pt] (1.0506,2.1515) -- (2.316,1.8) -- (4.728,3.8) -- (4.728,1.63) -- (9.096,0.262) -- (7.08,1.9505);

\foreach \point/\pos in {v_0/below, v_1/above, v_{2k}/below, v_{2k+1}/above left} {
    \node[circle, fill=black, inner sep=1.5pt] at (\point) {};
    \node[\pos] at (\point) {$\point$};
}
\node[circle, fill=black, inner sep=1.5pt] at (z) {};
\node[circle, fill=black, inner sep=1.5pt] at (z_3) {};
\node[right,xshift=3pt] at (z_3) {$z_3$};
\node[circle, fill=black, inner sep=1.5pt] at (v_{4k-1}) {};
\node[below] at (v_{4k-1}) {$v_{4k-1}=c$};
\node[circle, fill=black, inner sep=1.5pt] at (x) {};
\node[below] at (x) {$x=b$};
\node[circle, fill=black, inner sep=1.5pt] at (y) {};
\node[above left] at (y) {\small$y$};
\node[circle, fill=black, inner sep=1.5pt] at (z_1) {};
\node[right,xshift=3pt,yshift=-1pt] at (z_1) {$z_1=v_j$};
\node[circle, fill=black, inner sep=1.5pt] at (z_2) {};
\node[left,xshift=-4pt,yshift=0.9pt] at (z_2) {\small$z_2=v_{j+2}$};

\coordinate (line4) at (4.32,0);
\coordinate (line5) at (4.8,0);
\coordinate (line6) at (5.28,0);
\foreach \point in {line4,line5,line6} {
    \node[circle, draw=black, fill=black, inner sep=0.8pt, minimum size=2pt] at (\point) {};
}

\node[] at (6,2.1) {\Large\textcolor{red}{$P_4$}};
\node[] at (4.08,0.6) {\Large\textcolor{green}{$P_2$}};
\node[] at (8.8,1.9) {\Large\textcolor{blue}{$P_3$}};
\node[] at (0.6,3.7) {\Large\textcolor{brown}{$P_1$}};

\node[] at (4.8,-1) {(a)};
\end{tikzpicture}
\end{subfigure}
\hfill
\begin{subfigure}{0.48\textwidth}
\centering
\begin{tikzpicture}[scale=0.65]  
\coordinate (v_{4k-1}) at (0,0);
\coordinate (x) at (9.6,0);
\coordinate (v_1) at (4.8,4);  
\coordinate (v_0) at (4.8,1.5); 
\coordinate (v_{2k}) at (7.2,0.75); 
\coordinate (v_{2k+1}) at (3.84,3.2); 

\coordinate (z_2) at (1.32,1.1);
\coordinate (z) at (1.8,1.5);
\coordinate (z_1) at (2.28,1.9);
\coordinate (y) at (0.84,2.3);
\coordinate (z_3) at (7.2,2);

\draw[black, thick] (z_2) -- (z_1) ;
\draw[black, thick] (z_2) -- (y) ;
\draw[black, thick] (y) -- (z_1);
\draw[black, thick] (v_1) -- (z_1);
\draw[black, thick] (z_2) -- (v_{4k-1}) -- (x) -- (z_3);
\draw[] (v_{2k}) -- (v_{2k+1});

\draw[gray, line width=1.8pt] (4.8,1.43) -- (0,-0.07);
\draw[gray, line width=1.8pt] (v_{4k-1}) -- (x);
\draw[gray, line width=1.8pt] (9.6,-0.07) -- (4.8,1.43);

\draw[cyan, line width=1.5pt] (x) -- (z_3);
\draw[cyan, line width=1.5pt] (7.2,1.93) -- (4.8,3.93);
\draw[cyan, line width=1.5pt] (4.86,4) -- (4.86,1.5);
\draw[cyan, line width=1.5pt] (4.8,1.57) -- (9.6,0.07);

\draw[red,thick] (y) arc (170:5:3.211);
\draw[red, line width=1.5pt] (y) -- (z_2) -- (v_{4k-1});
\draw[red, line width=1.5pt] (0,0.07) -- (4.8,1.57);
\draw[red, line width=1.5pt] (4.74,1.5) -- (4.74,4);
\draw[red, line width=1.5pt] (4.8,4.07) -- (7.2,2.07);

\foreach \point/\pos in {v_0/below, v_1/above, v_{2k}/below, v_{2k+1}/above left} {
    \node[circle, fill=black, inner sep=1.5pt] at (\point) {};
    \node[\pos] at (\point) {$\point$};
}
\node[circle, fill=black, inner sep=1.5pt] at (z) {};
\node[circle, fill=black, inner sep=1.5pt] at (z_3) {};
\node[right,xshift=3pt] at (z_3) {$z_3$};
\node[circle, fill=black, inner sep=1.5pt] at (v_{4k-1}) {};
\node[below] at (v_{4k-1}) {$v_{4k-1}$};
\node[circle, fill=black, inner sep=1.5pt] at (x) {};
\node[below] at (x) {$x$};
\node[circle, fill=black, inner sep=1.5pt] at (y) {};
\node[above left] at (y) {\small$y$};
\node[circle, fill=black, inner sep=1.5pt] at (z_1) {};
\node[below,xshift=2pt,yshift=-1pt] at (z_1) {$z_1$};
\node[circle, fill=black, inner sep=1.5pt] at (z_2) {};
\node[left,xshift=-4pt,yshift=0.9pt] at (z_2) {\small$z_2$};

\coordinate (line4) at (4.32,0);
\coordinate (line5) at (4.8,0);
\coordinate (line6) at (5.28,0);
\foreach \point in {line4,line5,line6} {
    \node[circle, draw=black, fill=black, inner sep=0.8pt, minimum size=2pt] at (\point) {};
}

\node[] at (4.08,0.6) {\Large\textcolor{gray}{$D_3$}};
\node[] at (8.8,1.9) {\Large\textcolor{cyan}{$D_2$}};
\node[] at (0.6,3.7) {\Large\textcolor{red}{$D_4$}};

\node[] at (4.8,-1) {(b)};

\end{tikzpicture}

\end{subfigure}
\caption{$T\cup\{yz_1,yz_2,yz_3,v_{2k+1}v_{2k+1}\}$}
\label{graphOfMainTheoremCase1P1P2P3P4}
\end{figure}

 So, $z_1,z_2$ are on the path $v_{2k+1}v_{2k+2}\ldots v_{4k-1}$, say $z_1=v_j$
 and $z_2=v_{j+2}$ for some $j\in\{2k+1,\ldots,4k-3\}$. 
Note that $y$ has at most two neighbors in an odd cycle of length $2k+1$ by the odd girth assumption.
 Since the odd cycles  $v_0v_{2k}v_{2k+1}\ldots v_{4k-1}v_0$ and $D_1$ both have length $2k+1$ and these two odd cycles contain both $z_1,z_2$,
 it follows that 
 $z_3$ is on the path $P_{xv_1}-\{v_1\}$.
Consider the four paths,
$$P_1=z_3P_{xv_1}v_1v_{2k+1}v_{2k+2}\ldots v_{j-1}z_1y,\text{   }P_2=yz_2D_1[z_2,c,v_0]v_0D_2[v_0,v_1,z_3]z_3,$$$$P_3=yz_1D_1[z_1,v_1,v_0]v_0D_2[v_0,x,z_3]z_3,\text{ and }P_4=yz_2D_1[z_2,v_{4k-1},v_0]v_0D_2[v_0,v_1,z_3]z_3,$$
as shown in Figure~\ref{graphOfMainTheoremCase1P1P2P3P4}a.
Let $D_1'$ be the cycle obtained from $D_1$ by replacing $v_{j+1}$ with $y$. Clearly, $D_1'$ has length $2k+1$.
Note that  $D_1'\cup D_2$ is  a $(2k+1)$-Theta graph, say $\Theta$. Clearly, $P_1,P_2$ are the two $(y,z_3)$-outer paths, and
$P_3,P_4$ are the two $(y,z_3)$-intersecting paths in $\Theta$. 
By Lemma~\ref{P1P2P3P4property}, we have that either  cycles $yP_1z_3y,yP_2z_3y$ both have length $2k+1$, or exactly one of the cycles $yP_3z_3y,yP_4z_3y$ has length $2k+1$ while the other has length $2k+3$.  If the first case holds, then $\{z_1,z_2,z_3\}$ is a diagonal triple of $D_1'\oplus D_2$. It is easy to see that $\{z_1,z_2,z_3\}$ is also a diagonal triple of $D_1\oplus D_2$, a contradiction. 
So,
one of the cycles $yP_3z_3y,yP_4z_3y$ must have length $2k+1$.
If $yP_3z_3y$ has length $2k+1$,  then $v_{2k}v_{2k+1}$ is a chord of $yP_3z_3y$, which yields an odd cycle of length $2k-1$, a contradiction. 
So, the cycle $yP_4z_3y$ has length $2k+1$. Set $D_4:=yP_4z_3y$.   Note that the graph $T':=D_2\cup D_3\cup D_4$ is a tetrahedron in $\mathscr{T}_1$ with rim cycle $z_3v_{j+1}v_{j+2}\ldots v_{4k-1}P_{xv_{4k-1}}xP_{xv_1}z_3$ $xP_{xv_1}z_3v_{j+1}v_{j+2}\ldots v_{4k-1}$, center $v_0$ and branch vertices $z_3,v_{4k-1},x$. Since the spoke $v_0v_{4k-1}$ of length one in $T'$ is contained in two facial cycles $D_3,D_4$,
it follows from 
 Lemma~\ref{T1spokeisM4k}
 that $G[V(D_3)\cup V(D_4)]$ is an induced copy  of $M_{4k}$. 
 If $j<4k-3$, then $v_{2k}v_{4k-2}$
 is a diagonal chord of $G[V(D_3)\cup V(D_4)]$, 
 which implies that 
 $v_{2k}v_{2k+1}\ldots v_{4k-2}v_{2k}$ is an odd cycle of length $2k-1$, a contradiction.
If $j=4k-3$, then $v_{2k}y$ is a diagonal chord of $G[V(D_3)\cup V(D_4)]$, 
 which implies that $v_{2k}v_{2k+1}\ldots v_{4k-3}yv_{2k}$ is an odd cycle of length $2k-1$, a contradiction.

 {\bf \noindent Case 2.} The neighbors of $x$ in $\mathcal{M}$ belong to exactly two vertex classes of $\mathcal{M}$.

Without loss of generality, we may assume that 
 $x$ has exactly two neighbors, say $y_1,y_2$, in $M$.
Owing to the odd girth assumption, we have $dist_{M}(y_1,y_2)=2$, and let $y_1yy_2$ be  a  path in $M$.   
Suppose first that $y_1yy_2$ does not contain any diagonal chord of $M$. Then $y_1y_2y_3$
belongs to the rim cycle $v_0\ldots v_{2k-1}v_{2k}\ldots v_{4k-1}v_0$ of $M$.
Without loss of generality, we may assume that $y_1=v_{2k-1},y=v_{2k},y_1=v_{2k+1}$. Since $k\geq 2$, if follows that $M$ contains at least three diagonal chords, which implies that
$G[V(M)\cup\{x\}\setminus \{v_{2k}\}]$
 contains a subgraph isomorphic to $\Phi_{4k,3}$. By Lemma~\ref{phi4k3ism4k}, $G[V(M)\cup\{x\}\setminus \{v_{2k}\}]$ is an induced copy of  $M_{4k}$, which implies that $xv_0$ is a diagonal chord in this new copy of  $M_{4k}$. But then
$x$ has at least three neighbors $v_0,v_{2k-1},v_{2k+1}$ in $M$, a contradiction.

So, $y_1yy_2$ contains exactly one diagonal chord, say $yy_2$, of $M$. Without loss of generality,
we may assume that $y_1=v_{2k-1},y=v_{2k},y_2=v_0$.
Let $D_1:=v_{2k}v_0\ldots v_{2k-1}v_{2k}$
and $D_2:=v_0v_{2k}\ldots v_{4k-1}v_0$ be two odd cycles of length $2k+1$ in $M$.
Note that $v_0$ is the only neighbor of $x$ in $D_2$, and $xv_0v_{4k-1}v_{4k-2},xv_0v_{2k}v_{2k+1}$ are two paths of length three. So, by Lemma~\ref{xCycleIsT}, there exist two paths $P_{xv_{2k+1}},P_{xv_{4k-2}}$ of length $2k-2$
such that $D_2\cup P_{xv_{4k-2}}\cup P_{xv_{2k+1}}\cup \{xv_0\} $ form a tetrahedron $T\in \mathscr{T}_1$
 with center $v_0$.
 Clearly, $xv_0$ is the  spoke of length one  in $T$, i.e., $x$ is a branch vertex of $T$. Let $v_s,v_t$ be other branch vertices of $T$  with $2k+1\leq t<s\leq 4k-2$.
Let $C_{xv_s}:=xP_{xv_{4k-2}}v_{4k-2}v_{4k-1}v_0x,C_{xv_t}:=xP_{xv_{2k+1}}v_{2k+1}v_{2k}v_0x$ be two odd cycles of length $2k+1$ in $T$. Clearly, the facial cycles of $T$ are $D_2,C_{xv_s},C_{xv_t}$.
Let
$C:=C_{xv_s}\oplus C_{xv_t},P_{xv_s}=xP_{xv_{4k-2}}v_s,P_{xv_t}=xP_{xv_{2k+1}}v_t$. Since the spoke
$xv_0$ of $T$ is contained in the two facial cycles $C_{xv_s},C_{xv_t}$,
$G[V(C)]$ is an induced copy of 
$M_{4k}$ by Lemma~\ref{T1spokeisM4k}.

\begin{claim}\label{InsectD_1Only1Case2}
      $V(P_{xv_{2k+1}})\cap V(D_1)=\emptyset$ and $V(P_{xv_{4k-2}})\cap V(D_1)=\emptyset$.  
\end{claim}
\begin{proof}
By symmetry, it suffices to show that $V(P_{xv_{2k+1}})\cap V(D_1)=\emptyset$. Clearly, $v_0,v_{2k}\notin V(P_{xv_{2k+1}})$. 
 Suppose, to the contrary,  that $v_{i}\in V(P_{xv_{2k+1}})$ for some $1\leq i\leq 2k-1$.    
Since  $v_{2k}v_{0}\in E(D_1)\cap E(C_{xv_t})$,  
the paths $v_0v_{1}\ldots v_{i}$ and $v_{0}xP_{xv_{2k+1}}v_i$ have the same length by Lemma~\ref{2C2k+1isspoke}.
So replacing
 $v_{0}xP_{xv_{2k+1}}v_i$ of $C_{xv_t}$ by
 $v_0v_{1}\ldots v_{i}$ yields an odd closed walk $v_0v_1\ldots v_iP_{xv_{2k+1}}v_{2k+1}v_{2k}v_0$ of length $2k+1$. But then $v_{1}v_{2k+1}$ is a chord in this walk, which yields an odd cycle of length at most $2k-1$,
 a contradiction. This proves the claim.
 \end{proof}


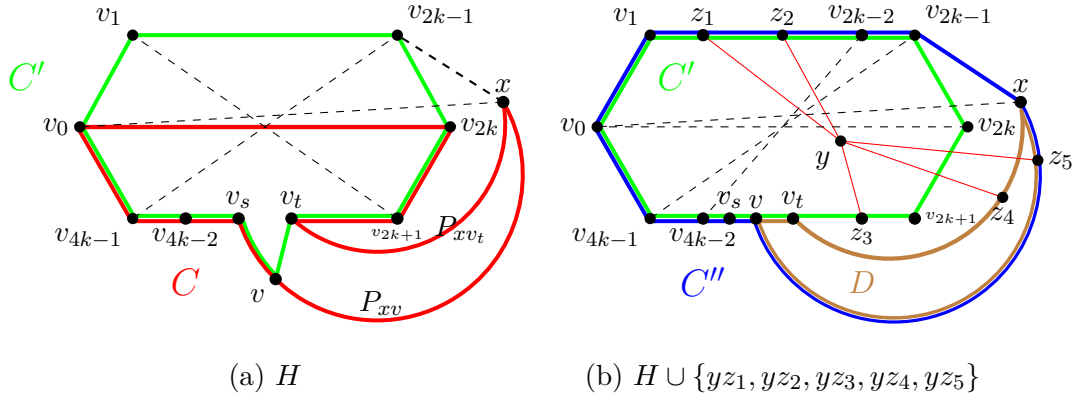
\begin{figure}[h]
\centering
\begin{minipage}{0.37\textwidth}
\centering
\begin{tikzpicture}[scale=0.7]
\coordinate (v_{4k-1}) at (0,0);
\coordinate (v_{2k+1}) at (5,0);
\coordinate (v_{2k}) at (6,1.732);  
\coordinate (v_{2k-1}) at (5,3.464);  
\coordinate (v_1) at (0,3.464);
\coordinate (v_0) at (-1,1.732);

\coordinate (v_2) at (1,3.464);  
\coordinate (v_{2k-2}) at (4,3.464);  

\coordinate (v_s) at (2,0);  
\coordinate (v_{4k-2}) at (1,0);  
\coordinate (v_{2k+2}) at (4,0);  
\coordinate (v_t) at (3,0); 

\coordinate (x) at (7,2.2); 
\coordinate (v) at (2.7,-1.144);

\draw[green,line width=1.5pt] (v) -- (3,0.05) -- (4.98,0.05)  -- (5.9513,1.732)  -- (v_{2k-1});
\draw[green,line width=1.5pt] (v_{2k-1}) -- (v_1)  ;

\draw[green,line width=1.5pt] (v_1)-- (-0.9513,1.732)  -- (0.02,0.05)  -- (2,0.05);
\draw[green,line width=1.5pt](2.75,-1.12) arc(-135:-163:2.75);

\draw[red,line width=1.5pt] (-0.02,-0.05) -- (2,-0.05);
\draw[red,line width=1.5pt] (3,-0.05) -- (5.02,-0.05) -- (6.0487,1.732)  -- (-1.0487,1.732) -- (-0.02,-0.05);
\draw[red,line width=1.5pt] (-1,1.725) -- (6,1.725);

\draw[dashed] (x) -- (v_0);
\draw[red,line width=1.5pt](x)arc(30:-163:2.75);
\draw[red,line width=1.5pt](x)arc(10:-133:2.41);

\foreach \point/\pos in {v_{4k-1}/below left, v_{2k}/right, v_{2k-1}/above right, v_1/above left, v_0/left} {
    \node[circle, fill=black, inner sep=1.5pt] at (\point) {};
    \node[\pos] at (\point) {$\point$};
}

\foreach \point in {v_{2k+1}, v_{4k-2},v_t,v_s} {
    \node[circle, draw=black, fill=black, inner sep=1.5pt, minimum size=4pt] at (\point) {};
}

\node[below] at (v_{4k-2}) {$v_{4k-2}$};
\node[above] at (v_s) {$v_s$};
\node[above] at (v_t) {$v_t$};
\node[below] at (v_{2k+1}) {\tiny$v_{2k+1}$};

\draw[dashed] (v_{2k+1}) -- (v_1);
\draw[dashed] (v_{4k-1}) -- (v_{2k-1});

\node[circle, draw=black, fill=black, inner sep=1.5pt, minimum size=4pt] at (x) {};
\node[above] at (x) {$x$};
\draw[thick,dashed] (x) -- (v_{2k-1});

\node[circle, draw=black, fill=black, inner sep=1.5pt, minimum size=4pt] at (v) {};
\node[below left] at (v) {$v$};


\node[] at (2.5,-3) {(a) $H$};
\node[] at (1,-1.2) {\Large\textcolor{red}{$C$}};
\node[] at (4.7,-1.6) {{$P_{xv}$}};
\node[] at (6.2,-0.2) {{$P_{xv_t}$}};
\node[] at (-2,2.7) {\Large\textcolor{green}{$C'$}};
\end{tikzpicture}
\end{minipage}
\hspace{0.1\textwidth} 
\begin{minipage}{0.37\textwidth}
\centering
\begin{tikzpicture}[scale=0.7]
\coordinate (v_{4k-1}) at (0,0);
\coordinate (v_{2k+1}) at (5,0);
\coordinate (v_{2k}) at (6,1.732);  
\coordinate (v_{2k-1}) at (5,3.464);  
\coordinate (v_1) at (0,3.464);
\coordinate (v_0) at (-1,1.732);
\coordinate (x) at (7,2.2); 

\coordinate (v_2) at (1,3.464);  
\coordinate (v_{2k-2}) at (4,3.464);  

\coordinate (v) at (2,0);  
\coordinate (v_s) at (1.5,0);  
\coordinate (v_{4k-2}) at (1,0);  
\coordinate (v_{2k+2}) at (4,0);  
\coordinate (v_t) at (2.7,0); 

\coordinate (z_2) at (2.5,3.464); 
\coordinate (y) at (3.6,1.464); 

\coordinate (z_4) at (6.66,0.4); 
\coordinate (z_5) at (7.324,1.1);

\draw[green,line width=1.5pt] (4.98,0.05)  -- (5.9513,1.732)  -- (4.98,3.414) -- (0.02,3.414) -- (-0.9513,1.732)  -- (0.02,0.05)  -- (4.98,0.05);

\draw[brown, line width=1.3pt] (7,2.04) arc (27:-161.3:2.694);
\draw[brown, line width=1.6pt] (x)arc(10:-135:2.52);
\draw[brown, line width=1.5pt] (2,-0.05) -- (2.7,-0.05);

\draw[blue,line width=1.5pt] (2,-0.05) -- (-0.02,-0.05) -- (-1.0487,1.732) -- (-0.02,3.514) -- (5.02,3.514) -- (x) ;

\draw[blue,line width=1.2pt](x)arc(30:-163:2.77);

\draw[red] (y) -- (v_{2k+2});
\draw[red] (y) -- (z_2);
\draw[red] (y) -- (v_2);
\draw[red] (y) -- (z_4);
\draw[red] (y) -- (z_5);

\draw[dashed] (v_{2k}) -- (v_0);
\draw[dashed] (v_{4k-1}) -- (v_{2k-1});
\draw[dashed] (v_{2k-2}) -- (v_{4k-2});
\draw[dashed] (x) -- (v_0);

\foreach \point/\pos in {z_4/below, z_5/right, y/below left, z_2/above, v_{4k-1}/below left, v_{2k}/right, v_{2k-1}/above right, v_1/above left, v_0/left} {
    \node[circle, fill=black, inner sep=1.5pt] at (\point) {};
    \node[\pos] at (\point) {$\point$};
}

\foreach \point in {v_s, v_{2k+1}, v_{4k-2}, v_{2k+2}, v_2, v_{2k-2},v,v_t} {
    \node[circle, draw=black, fill=black, inner sep=1.5pt, minimum size=4pt] at (\point) {};
}

\node[below] at (v_{4k-2}) {$v_{4k-2}$};
\node[below] at (v_{2k+2}) {\small$z_{3}$};
\node[above] at (v_t) {$v_t$};
\node[above] at (v_s) {$v_s$};
\node[above] at (v_2) {$z_1$};
\node[above] at (v) {$v$};
\node[above] at (v_{2k-2}) {$v_{2k-2}$};
\node[right] at (v_{2k+1}) {\tiny$v_{2k+1}$};

\node[circle, draw=black, fill=black, inner sep=1.5pt, minimum size=4pt] at (x) {};
\node[above] at (x) {$x$};

\node[] at (2.5,-3) {(b) $H\cup\{yz_1, yz_2,yz_3,yz_4,yz_5\}$};
\node[] at (1,-1.2) {\Large\textcolor{blue}{$C''$}};
\node[] at (0.5,2.7) {\Large\textcolor{green}{$C'$}};
\node[] at (4,-1.2) {\large\textcolor{brown}{$D$}};

\end{tikzpicture}

\end{minipage}

\caption{The proof of Case 2}
\label{graphOfMainTheoremCase2}
\end{figure}

Recall that $G[V(C)]$ is an induced copy of $M_{4k}$, and let
$v_{t}v$ be the diagonal chord in this $M_{4k}$. Since $D_2$ and $C_{xv_t}$ are two odd cycles of length $2k+1$, it follows that $vv _t$ cannot be the chord in these two odd cycles, which implies that
$v$ is on the path $P_{xv_{s}}-\{x\}$. 
Let 
$D_2':=v_tvP_{xv_{4k-2}}v_{4k-2}v_{4k-1}v_0v_{2k}\ldots v_t$ and $D:=v_tP_{xv_{t}}xP_{xv_{s}}vv_t$
be two odd cycles of $G[V(C)]$.
Since $vv_t$ is a diagonal chord of 
 $G[V(C)]$, it follows that $D_2',D$ both have length $2k+1$ and 
$D\cup D_2'=C\cup\{vv_t\}$.
Possibly $D_2=D_2'$. 
Let $C'=D_1\oplus D_2'$, let $P_{xv}:=xP_{xv_s}v$ be the sub-path  of $P_{xv_s}$, and 
let $H=G[V(C')\cup V(P_{xv})\cup V(P_{xv_t})]$, as shown in Figure~\ref{graphOfMainTheoremCase2}a.
Note that
$M':=G[V(C')]$ contains a $\Phi_{4k,3}$ with rim cycle $D_1\oplus D_2'$ and diagonal chords $v_1v_{2k+1},v_0v_{2k},v_{2k-1}v_{4k-1}$. 
By Lemma~\ref{phi4k3ism4k},  $M'$ is an induced copy of $M_{4k}$.
Note that $D$ is an odd cycle of length $2k+1$ in $T$.
It follows from Claim~\ref{InsectD_1Only1Case2} that
$H$ has exactly $6k-1$ vertices.

\begin{claim}\label{four}
For any vertex $y$ in $G$, $y$ has at most four neighbors in $H$. 
\end{claim}

\begin{proof}
Suppose not, and let $y$ be a vertex in $G$ with at least five neighbors, say $z_1,z_2,z_3,z_4,z_5$, in $H$.
Note that $V(H)=V(C)\cup \{v_1,\ldots,v_{2k-1}\}$.
Owing to the odd girth assumption, we conclude that 
$y$ has exactly two neighbors, say $z_1,z_2$, on the path $v_1v_2\ldots v_{2k-1}$ and 
exactly three neighbors, say $z_3,z_4,z_5$, in $C$ because $G[V(C)]$ is an induced copy of $M_{4k}$.
Since $G[V(C')]$ is an induced copy $M'$ of $M_{4k}$, it follows that $C'$ contains at most three of $z_1,z_2,z_3,z_4,z_5$
and already includes $z_1,z_2$. Note that $D_2'$  is contained in $M'$ and $z_1,z_2\notin V(D_2')$. This implies that $D_2'$ contains at most one of $z_3,z_4,z_5$. Notice also that $V(H)=\{v_1,\ldots,v_{2k-1}\}\cup V(D_2')\cup V(D)$. By the odd girth assumption, 
$D$ contains exactly two of $z_3,z_4,z_5$, say $z_4,z_5$,
and therefore, $D_2'$ contains $z_3$. 
Clearly, $z_3$
is 
on the path
$v_{2k}v_{2k+1}\ldots v_t$ or  $vP_{xv_{4k-2}}v_{4k-2}v_{4k-1}v_0$. Without loss of generality, we may assume that
$z_3$ is on the path $v_{2k}v_{2k+1}\ldots v_t$.
Note that $C'':=xP_{xv_{4k-2}}v_{4k-2}v_{4k-1}v_0v_1\ldots v_{2k-1}x$ is an even cycle of length $4k$, 
as shown in Figure~\ref{graphOfMainTheoremCase2}b. Let $M''=G[V(C'')]$. Clearly,
$M''$ contains a $\Phi_{4k,3}$ with rim cycle $C''$ and  diagonal chords $xv_0,v_{4k-1}v_{2k-1},v_{4k-2}v_{2k-2}$. By Lemma~\ref{phi4k3ism4k},
$M''$ is an induced copy of $M_{4k}$.
Since $C_{xv_t}$
is an odd cycle of length $2k+1$ and contains the neighbor $z_3$ of $y$,
it follows that $P_{xv_t}$ contains at most one of $z_4,z_5$.
Suppose first that $P_{xv_t}$ contains no neighbors of $y$. Then $P_{xv}$ contains 
 $z_4,z_5$.
But now $M''$
contains $z_1,z_2,z_4,z_5$, which yields 
an odd cycle of length smaller than $2k+1$, a contradiction.

Now suppose that $P_{xv_t}$ contains 
exactly one of $z_4,z_5$, say $z_4$.
Then $P_{xv}$ contains the neighbor $z_5$
of $y$. Since $dist_{D}(z_4,z_5)=2$, it follows that 
$z_4,z_5$ are neighbors of $x$ in $D$.
 Since we already know that $z_1,z_2,z_5$ are neighbors of $y$ in $M''$, it follows that $\{z_1,z_2,z_5\}$ is a  diagonal triple of $C''$ with $dist_{C''}(z_1,z_2)=2$. Since $M''$ is a copy of $M_{4k}$, it follows that $z_5v_1$
is a diagonal chord in $M''$ because the sub-path $v_1v_2\ldots v_{2k-1}xz_5$ of $C''$ has length $2k-2$.  However, this implies that $\{z_1,z_2\}=\{v_0,v_2\}$, a contradiction.
 This proves the claim.
\end{proof}

By Claim~\ref{four}, we have
 $4n=\frac{4n}{6k-1}(6k-1)<|V(H)|\delta(G)\leq\sum_{x\in V(H)}|N(x)|=\sum_{x\in V(G)}|N(x)\cap V(H)|\leq4n$, and we obtain a final contradiction.
 This completes the proof of Theorem~\ref{MainResult}
    \end{proof}

\section{Concluding Remarks}

It would be interesting to extend our result to graphs with smaller minimum degree. For every \(k \ge 4\), Letzter and Snyder~\cite{k=3case} asked whether there exists some \(\varepsilon > 0\) such that every  \(n\)-vertex graph $G$ with odd girth at least $2k+1$ and minimum degree \(\delta(G) \ge \bigl(\frac{1}{2k-1} + \varepsilon\bigr)n\)  is homomorphic to \(F_{\ell,k}\) for some \(\ell \ge 1\). 
On the other hand, 
 Ebsen and Schacht \cite{EbsenSchacht} constructed a specific blow‑up of a tetrahedron in \(\mathscr{T}_1\) that is not homomorphic to any such \(F_{\ell,k}\) and has minimum degree slightly smaller than \(\frac{n}{2k-2}\). Consequently, 
they asked the following for every $k\geq 4$: is
every  \(n\)-vertex graph $G$ with odd girth at least $2k+1$ and minimum degree \(\delta(G) \ge \frac{n}{2k-2}\) necessarily 
homomorphic to \(F_{\ell,k}\) for some \(\ell \ge 1\)?

\end{document}